\documentclass[11pt,reqno]{amsart}
\usepackage{}
\usepackage{mathrsfs}
\usepackage{amsfonts}
\usepackage{a4wide}
\allowdisplaybreaks \numberwithin{equation}{section}
\usepackage{color}

\newtheorem{theorem}{Theorem}[section]
\newtheorem{proposition}[theorem]{Proposition}
\newtheorem{corollary}[theorem]{Corollary}
\newtheorem{lemma}[theorem]{Lemma}

\theoremstyle{definition}

\newtheorem{remark}[theorem]{Remark}

\newcommand{\R}{\mathbb{R}}
\newcommand{\ds}{\displaystyle}

\begin{document}
\title[periodic boundary condition]
{solutions for a critical elliptic system with periodic boundary condition}
 \author{ Qingfang Wang, Wenju Wu$^\dag$ and Mingxue Zhai}

\address{School of Mathematics and Computer Science, Wuhan Polytechnic University, Wuhan 430079, P. R. China }
\email{wangqingfang@whpu.edu.cn}

\address{School of Mathematics and Statistics, Central China Normal University, Wuhan, 430079, P. R. China}
\email{wjwu@mails.ccnu.edu.cn}

\address{School of Mathematics and Statistics, Central China Normal University, Wuhan, 430079, P. R. China}
\email{mxzhai@mails.ccnu.edu.cn}

\thanks{The research was supported by NSFC (No.~12126356) and NSFC (No.~12471106).}
\thanks{$\dag$Corresponding author: Wenju Wu.}
%\date{\today}
\begin{abstract}

In this paper, we consider the following nonlinear critical Schr\"odinger system:
\begin{eqnarray*}\begin{cases}
-\Delta u=K_1(y)u^{2^*-1}+\frac{1}{2} u^{\frac{2^*}{2}-1}v^\frac{2^*}{2}, \,\,\,\,\,y\in\Omega,\,\,\,\,\,u>0,\cr
-\Delta v=K_2(y)v^{2^*-1}+\frac{1}{2} v^{\frac{2^*}{2}-1}u^\frac{2^*}{2}, \,\,\,\,\,y\in\Omega,\,\,\,\,\,v>0,\cr
u(y'+Le_j,y'')=u(y), \,\,\,\,\,\frac{\partial u(y'+Le_j,y'')}{\partial y_j}=\frac{\partial u(y)}{\partial y_j}, \,\,\,\,\,if\,\, y'=-\frac{L}{2}e_j,\,\,\,j=1, \ldots, k,\cr
v(y'+Le_j,y'')=v(y), \,\,\,\,\,\frac{\partial v(y'+Le_j,y'')}{\partial y_j}=\frac{\partial v(y)}{\partial y_j}, \,\,\,\,\,if\,\, y'=-\frac{L}{2}e_j,\,\,\,j=1, \ldots, k,\cr
u,v \to 0 \,\,as \,\,|y''|\to \infty,
\end{cases}
\end{eqnarray*}
where $K_1(y),\,K_2(y)$ satisfy some periodic conditions and $\Omega$ is a strip. Under some conditions which are weaker than Li, Wei and Xu(J. Reine Angew. Math. 743: 163-211, 2018), we prove that there exists a single bubbling solution for the above system. Moreover, as the appearance of the coupling terms, we construct different forms of solutions, which makes it more interesting.
Since there are periodic boundary conditions, this expansion for the difference between the standard bubbles
and the approximate bubble can not be obtained  by using the comparison theorem as one usually does for Dirichlet boundary condition. To overcome this difficulty, we will use the Green's function of $-\Delta$ in $\Omega$ with periodic boundary conditions which helps us find the approximate bubble. Due to the lack of the Sobolev inequality, we will introduce a suitable weighted space to carry out the reduction.


\medskip\noindent
{\bf AMS:} 35J60; 35B33.
\end{abstract}
\maketitle
\section{Introduction}
In this paper, we will consider the following critical problem with periodic boundary conditions
\begin{align}\begin{split}\label{eqs1.1}
\begin{cases}
-\Delta u=K_1(y)u^{2^*-1}+\frac{1}{2} u^{\frac{2^*}{2}-1}v^\frac{2^*}{2}, \,\,\,\,\,y\in\Omega,\,\,\,\,\,u>0,\cr
-\Delta v=K_2(y)v^{2^*-1}+\frac{1}{2} v^{\frac{2^*}{2}-1}u^\frac{2^*}{2}, \,\,\,\,\,y\in\Omega,\,\,\,\,\,v>0,\cr
u(y'+Le_j,y'')=u(y), \,\,\,\,\,\frac{\partial u(y'+Le_j,y'')}{\partial y_j}=\frac{\partial u(y)}{\partial y_j}, \,\,\,\,\,if\,\, y'=-\frac{L}{2}e_j,\,\,\,j=1, \ldots, k,\cr
v(y'+Le_j,y'')=v(y), \,\,\,\,\,\frac{\partial v(y'+Le_j,y'')}{\partial y_j}=\frac{\partial v(y)}{\partial y_j}, \,\,\,\,\,if\,\, y'=-\frac{L}{2}e_j,\,\,\,j=1, \ldots, k,\cr
u,v \to 0 \,\,as \,\,|y''|\to \infty,
\end{cases}
\end{split}
\end{align}
where $2^*=\frac{2N}{N-2}$, $N \ge 5$, $L>0$ is a constant, $\Omega$ is a strip defined as
$$\Omega=\Bigl\{(y',y''):|y_i| \le \frac{L}{2},\,i=1,\ldots,k,\,y' \in \mathbb{R}^{k},\,y''\in \mathbb{R}^{N-k}\Bigr\},$$
and $e_1=(1,0,\ldots,0)$, $e_2=(0,1,\ldots,0), \ldots , e_N=(0,0,\ldots,1)$.

For the general scalar curvature equation,
\begin{align}\label{weqs1.2}
-\Delta u=K(x)u^{\frac{N+2}{N-2}},\,\,u>0,\,\,\hbox{in}\,\,\R^N,
\end{align}
when $N=3$, Yan \cite{Yan} proved that there was no solution whose critical value is close to the mountain pass value, under the assumption that $K(x)$ has a sequence of local maximum points $p_j\rightarrow+\infty$, satisfying
\begin{align*}
K(x)=K_j-a_j|x-p_j|^\beta+O(|x-p_j|^{\beta+\sigma}),\,\,x\in B_\delta(p_j),
\end{align*}
where $\beta\in (N-2,N)$, $a_j>0$, $\delta>0$ are some small constants, and the equation has a solution concentrating at any two points.
 This result can be extended to the existence of $n$-bubbling solutions for any integer $n>0$. When $K(x)$ is positive and periodic, Li \cite{Li-Yan-2,Li-Yan-3} proved that \eqref{weqs1.2} has infinitely many multi-bump solutions for $n\geq3$ by gluing approximate solutions with masses concentrating near isolated sets of maximum points of $K(x)$.  When $K(x)$ is positive and periodic, Li et al. constructed in \cite{Li-Wei} multi-bump solutions with mass concentrating near critical point of $K$ including saddle points, see also \cite{Xu}. When $K(x)$ is a positive radial function with a strict local maximum at $|x|=r_0>0$ and satisfies
$$
K(r)=K(r_0)-C_0|r-r_0|^m+O(|r-r_0|^{m+\theta})
$$
for some constant $C_0>0$, $\theta>0$ and $m\in [2,n-2)$ near $|x|=r_0$. Wei and Yan \cite{Wei-Yan} constructed solutions with a large number bumps concentrating near the sphere $|x|=r_0$ for $n\geq 5$. Also, in \cite{Li-Wei}, they constructed multi-bump solutions of \eqref{weqs1.2} near critical point of $K(x)$ and bumps can be placed on arbitrary many or even infinitely many lower dimensional lattice points. More precisely, $K(x)$ satisfies the following conditions \\
$(H_1):0<\inf\limits_{\R^N}K\leq \sup\limits_{\R^N}K<+\infty$;\\
$(H_2):K\in C^1(\R^N)$, $K$ is 1-periodic in its first $k$ variables;\\
$(H_3):0$ is a critical point of $K$ satisfying: there exists some real number $\beta\in(n-2,n)$ such that near $0$,
$K(x)=K(0)+\sum\limits_{i=1}^na_i|x_i|^\beta+R(x)
$
where $a_i\neq0$, $\sum\limits_{i=1}^na_i<0$ and $R(y)$ is $C^{[\beta]-1,1}$ near $0$ and satisfies $\sum\limits_{s=0}^{[\beta]}|\nabla^sR(y)||y|^{-\beta+s}=o(1)$, as $y$ tends to $0$. Here $\nabla^s$ denotes all possible partial derivatives of orders $s$.

Moreover, in \cite{Guo-Yan}, Guo and Yan proved that the prescribed scalar curvature problem in $\R^N$ has solutions which are periodic in some variables, if the scalar curvature $K(y)$ is periodic, where $\Omega$ satisfying the periodic boundary condition. Also, $K(x)$ satisfies the condition is different in $(A_2)$ with \cite{Li-Wei}, since $\beta=\min\limits_{1\leq i\leq N}\beta_i\leq \beta_M:=\max\limits_{1\leq i\leq N}\beta_i<\beta(1+\frac{1}{N-2})$.
Also Guo et al. deal with the local uniqueness and non degeneracy of multi-bubble solutions of \eqref{weqs1.2} in \cite{G-P-Y} and \cite{Guo-Mu-Peng-Yan} respectively.
More results about the critical exponent of equation can refer to \cite{Loins,Loins2,P-W-Y}.

Systems of nonlinear Schr\"odinger equations have been the subject of extensive mathematical studies in recent years, for example, \cite{l-w,l-w2,liu-w,sirakov,t-v,t-w,w-y}. We also refer to \cite{c-z,c-z2,g-l-w,p-p-w,p-w} for more references therein about systems with both critical and subcritical exponents. So far, the existence results of using the Lyapunov-Schmidt reduction argument to construct concentrated solutions of an elliptic system have been basically about lower dimensional cases with $N<5$, such as \cite{g-l-w,p-p-w} . We further point out that Peng, Peng and Wang \cite{p-p-w} obtained a uniqueness result on the least energy solution and showed that a manifold of the synchronized type of positive solutions is non-degenerate for system.
Also, Guo, Wang and Wang \cite{guo-wang-wang} using the reduction method considered the synchronized solutions in higher dimensions.
Very recently, Guo et al. \cite{Guo-Wu-Yan} and Guo and Wu \cite{Guo-Wu2} considered the system of Hamiltonian type in a strip in $\R^N$, satisfying the periodic boundary condition for the first $k$ variables.

It seems difficult to find a solution for \eqref{eqs1.1} in the variational framework, because the following inequality does not hold for all the functions $u$ satisfying the boundary condition in \eqref{eqs1.1}
\begin{align}\label{weqs1.3}
\Bigl(\int_{\Omega}|u|^{2^*}\Bigr)^{\frac{1}{2^*}}\leq C\Bigl(\int_{\Omega}|\nabla u|^2\Bigr)^{\frac{1}{2}}.
\end{align}
Indeed, if we insert $u=\varphi(y_{k+1},\cdots,y_N)(\sin\frac{2\pi y_1}{L}+n)\cdots(\sin\frac{2\pi y_k}{L}+n)$ into \eqref{weqs1.3}, where $\varphi\in C_0^\infty(\R^{N-k})$, letting $n\rightarrow+\infty$, it yields
\begin{align*}
\Bigl(\int_{\R^{N-k}}|\varphi|^{2^*}dx\Bigr)^{\frac{1}{2^*}}\leq C\Bigl(\int_{\R^{N-k}}|\nabla \varphi|^2\Bigr)^{\frac{1}{2}},\,\,\,\varphi\in C_0^\infty(\R^{N-k}),
\end{align*}
which is clearly untrue.

In this paper,we will construct directly a solution $(u,v)$ for \eqref{eqs1.1} if $L>0$ is a large integer and
$K_1(x)$, $K_2(x)$ satisfy the following conditions:

$(A_1)$ $K_1$ and $K_2$ are both 1-periodic with respect to their first $k$ variables;

$(A_2)$ $K_1$, $K_2 \in L^\infty(\R^N)$, and there exist constants $a_{li}\ne 0$, $\beta_{li}\in (N-2,N)$, $l=1,2$, $i=1,\ldots,N$, such that
$\beta_l:=\min\limits_{1\le i\le N} \beta_{li} \le \beta_{lM}:=\max\limits_{1\le i\le N} \beta_{li} \le \beta_l(1+\frac{1}{N-2}),$

$$K_1(x)=1+\sum_{i=1}^{N} a_{1i} |x_i|^{\beta_{1i}}+O(|x|^{\beta_{1M}+\sigma_1}),\,\,\,x\in B_{\delta}(0)$$
and
$$K_2(x)=1+\sum_{i=1}^{N} a_{2i} |x_i|^{\beta_{2i}}+O(|x|^{\beta_{2M}+\sigma_2}),\,\,\,x\in B_{\delta}(0),$$
where $\delta>0$ and $\sigma_1, \sigma_2>0$ are some small constants. Furthermore, $\sum_{j\in \{j:\beta_{lj}=\beta_l\}} a_{lj}<0$.
\begin{theorem}\label{th1.1}
Suppose that $K_1$, $K_2$ satisfy the conditions $(A_1)$ and $(A_2)$. If $N\ge5$ and $1\le k<\frac{N-2}{2}$, then \eqref{eqs1.1} has a solution $(u_L,v_L)$, if the integer $L>0$ is large enough.
\end{theorem}

A direct consequence of Theorem \ref{th1.1} is the existence of solutions for
\begin{align}\begin{split}\label{eq1.2}
\begin{cases}
-\Delta u=K_1(y)u^{2^*-1}+\frac{1}{2} u^{\frac{2^*}{2}-1}v^\frac{2^*}{2}, \,\,\,\,\,u>0\,\,\hbox{in}\,\,\R^N,\cr
-\Delta v=K_2(y)v^{2^*-1}+\frac{1}{2} v^{\frac{2^*}{2}-1}u^\frac{2^*}{2}, \,\,\,\,\,v>0\,\,\hbox{in}\,\,\R^N,\cr
\end{cases}
\end{split}
\end{align}
which are periodic in its first k variables.

\begin{corollary}\label{cor1.2}
Under the same conditions as in Theorem \ref{th1.1}, \eqref{eq1.2} has infinitely many solutions, which are $L$-periodic in its first $k$ variables for any large integer $L>0$.
\end{corollary}

Since we consider periodic boundary conditions, this expansion for the difference between the standard bubbles
and the approximate bubble can not be obtained  by using the comparison theorem as one usually does for Dirichlet boundary condition. To overcome this difficulty, we find the Green's function for $-\Delta$ with periodic boundary condition and the approximate bubble is built via Green's function.
Also, due to the lack of the Sobolev inequality mentioned before, it is not suitable to work in a standard Sobolev space. Based on the goal expansion for the approximate bubble, we introduce a suitable weighted space (see \eqref{eqs1.7} later) to carry out the reduction argument.
We use Fredholm theorem to discuss the invertibility of the linear operator for the approximate bubble. This step is rather difficult and it depends very much on the correct choice of the weight space.

We now outline the proof of Theorem \ref{th1.1}. Let
$$
W_{x,\mu}=\frac{C_N\mu^\frac{N-2}{2}}{(1+\mu^2|y-x|^2)^\frac{N-2}{2}},\,\,\,x\in \R^N,\,\,\,\mu>0,
$$
where $C_N=(N(N-2))^\frac{N-2}{4}$. Then $W_{x,\mu}$ gives all the solutions to the problem
$$
-\Delta u=u^\frac{N+2}{N-2},\,\,\,u>0 \,\, \hbox{in} \,\, \R^N.
$$

To prove Theorem \ref{th1.1}, we will use the finite dimensional reduction method to construct a solution $(u_L,v_L)$ for \eqref{eqs1.1}, which satisfies $(u_L,v_L)\approx (U_{x,\mu},V_{x,\mu})$ for some $x$ close to 0 and $\mu>0$ large. Denote
$$(U_{x_0,\lambda},V_{x_0,\lambda})=(sW_{x,\mu},tW_{x,\mu}),\,\,\,t=\kappa s,\,\,\,x_0 \in \R^N,\,\,\,\lambda>0$$
is a non-degenerate solution (see \cite{p-p-w}, Theorem 1.1) of the following equations
\begin{eqnarray*}\begin{cases}
-\Delta u_1=u_{1}^{2^*-1}+\frac{1}{2} u_{1}^{\frac{2^*}{2}-1}u_{2}^\frac{2^*}{2}, \,\,\,\,\,y\in\R^N,\cr
-\Delta u_2=u_{2}^{2^*-1}+\frac{1}{2} u_{2}^{\frac{2^*}{2}-1}u_{1}^\frac{2^*}{2}, \,\,\,\,\,y\in\R^N,\cr
 u_1, u_2>0,\,\,\, u_1, u_2\in D^{12}(\R^N),
\end{cases}
\end{eqnarray*}
where $s$ and $\kappa$ satisfy that $s^{2^*-2}=\frac{2}{2+\kappa^\frac{2^*}{2}}$, $2+\kappa^\frac{2^*}{2}-\kappa^{\frac{2^*}{2}-1}-2\kappa^{2^*-1}=0.$

In view of the boundary conditions in \eqref{eqs1.1}, we construct the approximate solution $(PU_{x,\mu}, PV_{x,\mu})$ which solves the following problem
\begin{equation}\begin{cases}\label{eqs1.2}
-\Delta u=U_{x,\mu}^{2^*-1}+\,\frac{1}{2} U_{x,\mu}^{\frac{2^*}{2}-1}V_{x,\mu}^\frac{2^*}{2}, \,\,\,\,\,\text{in}\,\Omega,\cr
-\Delta v=V_{x,\mu}^{2^*-1}+\,\frac{1}{2} V_{x,\mu}^{\frac{2^*}{2}-1}U_{x,\mu}^\frac{2^*}{2}, \,\,\,\,\,\text{in}\,\Omega,\cr
(u,v)\, \text{\,satisfies the boundary condition in \,\eqref{eqs1.1}}.
\end{cases}
\end{equation}

To see that \eqref{eqs1.2} is solvable, we need to find the Green's function  $G(y, z)$  of  $-\Delta$  in  $\Omega$  satisfying the boundary conditions in \eqref{eqs1.1}.

For any integer  $k \in[1, N]$, we define the  $k$-dimensional lattice by:
 $$
 Q_{k}:=\Bigl\{(y_{1}, \ldots, y_{k}, 0): y_{i} \,\,\hbox{is}\,\, \hbox{an}\,\, \hbox{integer},  i=1, \ldots, k\Bigr\} , \quad \hbox{where} \,\, 0 \in \mathbb{R}^{N-k}.
 $$
We order all the points in  $Q_{k}$  as  $\left\{P_{j}\right\}_{j=1}^{\infty}$  and  $P_{0}=0$. The Green's function $G(y,z)$ of $-\Delta$ in $\Omega$ satisfying the boundary conditions in \eqref{eqs1.1} is given by
$$G(y,z)=\sum_{j=0}^{\infty}\Gamma(y+LP_j,z),\,\,\,y,\,z\in \Omega,$$
where
$$\Gamma(y,z)=\frac{1}{(N-2)\omega_{N-1}|y-z|^{N-2}}$$
and $\omega_{N-1}$ is the area of $\mathbb{S}^{N-1}$. Then
\begin{equation}\label{eqs1.5}
\binom{PU_{x,\mu}}{PV_{x,\mu}}=\binom{\ds\int_\Omega G(y,z)(U_{x,\mu}^{2^*-1}+\frac{1}{2} U_{x,\mu}^{\frac{2^*}{2}-1}V_{x,\mu}^\frac{2^*}{2})dz}{\ds\int_\Omega G(y,z)(V_{x,\mu}^{2^*-1}+\frac{1}{2} V_{x,\mu}^{\frac{2^*}{2}-1}U_{x,\mu}^\frac{2^*}{2})dz}
\end{equation}
solves \eqref{eqs1.2}. Let us also point out that bounded solutions of \eqref{eqs1.2} are unique. In fact, for any bounded solutions $(u_1,v_1)$ and $(u_2,v_2)$ of \eqref{eqs1.2}, $(u,v)=(u_1-u_2,v_1-v_2)$, then $-\Delta u=0$, $-\Delta v=0$ in $\Omega$, which can be extended periodically to $\R^N$. Thus $(u,v)$ must be constant, which, together with $(u,v) \to 0$ as $|y''| \to +\infty$, gives $(u,v)=(0,0)$.
\begin{remark}
We aim to construct a solution of the form
$$(u_L,v_L)=(PU_{x_L,\mu_L}+\omega_{1L},PV_{x_L,\mu_L}+\omega_{2L})$$
for some $x_L$ close to 0, and some large $\mu_L>0$, where the correction term $(\omega_{1L},\omega_{2L})$ is small in some norm.
\end{remark}
Using \eqref{eqs1.5}, we can prove
\begin{equation}\label{eqs1.6}
PU_{x,\mu}=U_{x,\mu}+\frac{A_1}{\mu^\frac{N-2}{2}}\sum_{j=1}^{\infty}\frac{1}{|y-x-LP_j|^{N-2}}\Bigl(\int_{\R^N} U_{0,1}^{2^*-1}+o(1)\Bigr)
\end{equation}
and
\begin{equation}\label{eqs1.6(1)}
PV_{x,\mu}=V_{x,\mu}+\frac{A_2}{\mu^\frac{N-2}{2}}\sum_{j=1}^{\infty}\frac{1}{|y-x-LP_j|^{N-2}}\Bigl(\int_{\R^N} V_{0,1}^{2^*-1}+o(1)\Bigr),
\end{equation}
where $A_1=\kappa^{2^*-1}A_2$.

To take \eqref{eqs1.6} and \eqref{eqs1.6(1)} into account, we define the following norm:
\begin{equation}\label{eqs1.7}
\|u\|_{*}=\sup _{y \in \Omega}\Bigl(\sigma(y) \sum_{j=0}^{\infty} \frac{\mu^{\frac{N-2}{2}}}{\left(1+\mu\left|y-x_{j}\right|\right)^{\frac{N-2}{2}+\tau}}\Bigr)^{-1}|u(y)|,
\end{equation}
where $x_j=x-LP_j$, $x\in B_1(0)$, $\sigma(y)=\min \{1,(\frac{1+\mu|y-x|}{\mu})^{\tau-1}\}$, $\tau=\frac{N-2}{2}-\vartheta$, and $\vartheta>0$ is a fixed small constant.
In order to improve the error term $\omega_L$ in the case $k\geq2$, we add an extra weight $\sigma(y)$ in $\|u\|_{*}$ which was used in \cite{Guo-Peng-Yan, Guo-Wu, Guo-Wu-Yan, Guo-Yan, Li-Wei}.

We point out that with our choice  $\tau=\frac{N-2}{2}-\vartheta$, it holds that for  $y \in \Omega \backslash B_{1}(0)$,
\begin{align*}
\sum_{j=0}^{\infty} \frac{1}{|y-x_{j}|^{\tau}} & \leq \sum_{j=0}^{\infty} \frac{C}{\left(1+\left|y-x_{j}\right|\right)^{\tau}} \leq C \int_{\mathbb{R}^{N}} \frac{1}{\left(1+|y-z|^{2}\right)^{\frac{\tau}{2}}} d z \\
& =C \int_{\mathbb{R}^{N}} \frac{1}{\left(1+\left|y^{\prime \prime}\right|^{2}+|z|^{2}\right)^{\frac{\tau}{2}}} d z \leq \frac{C}{\left(1+\left|y^{\prime \prime}\right|\right)^{\tau-k}},
\end{align*}
since  $k<\frac{N-2}{2}$, and
$$\left|y-x_{j}\right| \geq \frac{1}{4}\left(1+\left|y-x_{j}\right|\right), \quad y \in \Omega \backslash B_{1}(0).$$

We define the space  $\mathbf{X}$  as
$$\mathbf{X}=\Bigl\{(\phi_1,\phi_2): (\phi_1,\phi_2) \text { satisfies the boundary conditions in (1.1) and }\|(\phi_1,\phi_2)\|_{*}<+\infty\Bigr\},$$
where $\|(\phi_1,\phi_2)\|_{*}:=\|\phi_1\|_{*}+\|\phi_2\|_{*}$.

We will carry out the reduction procedure in $\mathbf{X}$. For this purpose, we need to study the invertibility of the linear operator
\begin{eqnarray}\begin{cases}\label{weqs1.8}
\phi_1-(2^*-1)(-\Delta)^{-1}\Bigl(K_1(y)(PU_{x,\mu})^{2^*-2}\phi_1&-\frac{1}{N-2}(PU_{x,\mu})^{\frac{4-N}{N-2}}(PV_{x,\mu})^{\frac{N}{N-2}}\phi_1\cr
&-\frac{N}{2(N-2)}(PU_{x,\mu})^{\frac{2}{N-2}}(PV_{x,\mu})^{\frac{2}{N-2}}\phi_2\Bigr),\cr
\phi_2-(2^*-1)(-\Delta)^{-1}\Bigl(K_2(y)(PV_{x,\mu})^{2^*-2}\phi_2&-\frac{1}{N-2}(PV_{x,\mu})^{\frac{4-N}{N-2}}(PU_{x,\mu})^{\frac{N}{N-2}}\phi_2\cr
&-\frac{N}{2(N-2)}(PV_{x,\mu})^{\frac{2}{N-2}}(PU_{x,\mu})^{\frac{2}{N-2}}\phi_1\Bigr),\cr
\end{cases}
\end{eqnarray}
where $(\phi_1,\phi_2) \in \mathbf{X}$ and the linear operator  $(-\Delta)^{-1}$  is defined as
$$(-\Delta)^{-1}(f,g)=\bigl(\int_{\Omega} G(z, y)f(z)dz,\int_{\Omega} G(z, y)g(z)dz\bigl),$$
for any  $(f,g)$  satisfies $\|(f,g)\|_{**}<+\infty$, where
$$\|(f,g)\|_{**}=\|f\|_{**}+\|g\|_{**}$$
and
\begin{equation}\label{eqs1.9}
\|f\|_{**}:=\sup _{y \in \Omega}\Bigl(\sigma(y) \sum_{j=0}^{\infty} \frac{\mu^{\frac{N+2}{2}}}{\left(1+\mu\left|y-x_{j}\right|\right)^{\frac{N+2}{2}+\tau}}\Bigr)^{-1}|f(y)|.
\end{equation}
To use the Fredholm theories, we must prove that  $$(-\Delta)^{-1}\binom{K_1(y)(PU_{x,\mu})^{2^*-2}\phi_1-\frac{1}{N-2}(PU_{x,\mu})^{\frac{4-N}{N-2}}(PV_{x,\mu})^{\frac{N}{N-2}}\phi_1-\frac{N}{2(N-2)}(PU_{x,\mu})^{\frac{2}{N-2}}(PV_{x,\mu})^{\frac{2}{N-2}}\phi_2}
{K_2(y)(PV_{x,\mu})^{2^*-2}\phi_2-\frac{1}{N-2}(PV_{x,\mu})^{\frac{4-N}{N-2}}(PU_{x,\mu})^{\frac{N}{N-2}}\phi_2-\frac{N}{2(N-2)}(PV_{x,\mu})^{\frac{2}{N-2}}(PU_{x,\mu})^{\frac{2}{N-2}}\phi_1}$$  is compact in  $\mathbf{X}$. This will require some delicate estimates on the decay rate at infinity for the above function with  $(\phi_1,\phi_2) \in \mathbf{X}$. We would like to stress that compared with  \cite{Guo-Yan}, many complicated computations are involved in our study since there are coupled terms of nonlinearities in system \eqref{eqs1.1}.

The paper is organized as follows. In section 2, we will prove some estimates for the approximate solutions  $(PU_{x, \mu},PV_{x, \mu})$. These estimates play an essential role in the discussion of the solvability of the reduced finite dimensional problems. Section 3 is devoted to the study of the invertibility of the linear operator defined in \eqref{weqs1.8}. In section 4, we will prove Theorem 1.1.

\section{some essential estimates for approximate solutions}
Throughout this paper, we always assume that $(x,\mu)$ satisfies
\[
	\vert x\vert\le\frac{1}{\mu^{1+\theta}},\,\,\,\mu L^{-\frac{N-2}{\beta-N+2}}\in\Bigl[\frac{C_{0}}{2},2C_{0}\Bigr]
\]
for some $C_{0} >0$ and $\beta=\min\{\beta_1,\beta_2\}$, where $\theta >0$ is a small constant. We denote
$x_{j}=x-LP_{j},$ $j=0,1,\ldots$.

 Recall that $(PU_{x,\mu},PV_{x,\mu})$ is given in \eqref{eqs1.5}. We have
\begin{lemma}\label{lm2.1}
It holds
$$0\leq PU_{x,\mu}\leq C_{1,\kappa} \sum_{j=0}^{\infty}U_{x_{j},\mu},\,\,\,0\leq PV_{x,\mu}\leq C_{2,\kappa} \sum_{j=0}^{\infty}V_{x_{j},\mu},$$
where $C_{1,\kappa}$ and $C_{2,\kappa}$ are positive constants.
\end{lemma}
\begin{proof}
It follows from \eqref{eqs1.5} that
\begin{align*}
PU_{x,\mu}(y)&=\int_{\Omega}G(y,z)\Bigl(U_{x,\mu}^{2^{*}-1}+\frac{\beta}{2} U_{x,\mu}^{\frac{2^*}{2}-1}V_{x,\mu}^\frac{2^*}{2}\Bigr)dz=C_{1,\kappa} \sum_{j=0}^{\infty}\int_{\Omega}\Gamma(z+L P_{j},y)U_{x,\mu}^{2^{*}-1}dz\\
&\le C_{1,\kappa} \sum_{j=0}^{\infty}\int_{\mathbb{R}^{N}}\Gamma(z+LP_{j},y)U_{x,\mu}^{2^{*}-1}dz=C_{1,\kappa} \sum_{j=0}^{\infty}U_{x,\mu}(y+L P_{j}).
\end{align*}
Similarly, we can estimate $PV_{x,\mu}$.
\end{proof}
Let $(\varphi_1,\varphi_2)=\Bigl(U_{x,\mu}-PU_{x,\mu},V_{x,\mu}-PV_{x,\mu}\Bigr)$.
Then
$(\varphi_1,\varphi_2)$ has the following expansion.

\begin{lemma}\label{lm2.2}
For $y\in B_{1}(0),$ it holds
\begin{equation}\label{eqs2.1}
\varphi_l(y)=-B_l\sum_{j=1}^{\infty}\frac{\Gamma(x+L P_{j},y)}{\mu^{\frac{N-2}{2}}}+O\Bigl(\frac{1}{L^{N-2}\mu^{\frac{N+2}{2}}}\Bigr),
\end{equation}
\begin{equation}\label{eqs2.2}
\frac{\partial\varphi_l}{\partial x_{h}}(y)=-B_l\frac{1}{\mu^{\frac{N-2}{2}}}\sum_{j=1}^{\infty}\frac{\partial\Gamma(x+L P_{j},y)}
{\partial x_{h}}+O\Bigl(\frac{1}{L^{N-2}\mu^{\frac{N+2}{2}}}\Bigr)
\end{equation}
and
\begin{equation}\label{eqs2.3}
\frac{\partial\varphi_l}{\partial\mu}(y)=B_l\sum_{j=1}^{\infty}\frac{(N-2)\Gamma(x+LP_{j},y)}{2\mu^{\frac{N}{2}}}
+\frac{1}{\mu}O\Bigl(\frac{1}{L^{N-2}\mu^{\frac{N+2}{2}}}\Bigr),
\end{equation}
where $l=1,2$, $B_1=\ds\int_{\mathbb{R}^{N}}U_{0,1}^{2^{*-1}}$ and $B_2=\ds\int_{\mathbb{R}^{N}}V_{0,1}^{2^{*-1}}$.
\end{lemma}
\begin{proof}
We have
\begin{equation}\label{eqs2.4}
PU_{x,\mu}(y)=C_\kappa \sum_{j=0}^{\infty}\int_{\Omega}\Gamma(z+LP_{j},y)U_{x,\mu}^{2^{*}-1}dz.
\end{equation}
We may assume $C_\kappa =1$. For $j=0$, it holds
\begin{align*}
&\int_{\Omega}\Gamma(z+LP_{0},y)U_{x,\mu}^{2^{*}-1}=\int_{\R^N}\Gamma(z,y)U_{x,\mu}^{2^*-1}-\int_{\R^N \backslash \Omega}\Gamma(z,y)U_{x,\mu}^{2^*-1}\cr
&=U_{x,\mu}(y)+O\Bigl(\int_{\R^N \backslash \Omega}\frac{1}{|y-z|^{N-2}}\frac{1}{|z|^{N+2}\mu^{\frac{N+2}{2}}}dz\Bigr)\cr
&=U_{x,\mu}(y)+O\Bigl(\int_{\R^N \backslash \Omega}\frac{1}{|z|^{2N-2}}\frac{1}{\mu^{\frac{N+2}{2}}}dz\Bigr)
=U_{x,\mu}(y)+O\Bigl(\frac{1}{L^{N-2}\mu^{\frac{N+2}{2}}}\Bigr).
\end{align*}
For $j\neq0,$ we have
\begin{align*}
&\int_{\Omega}\Gamma(z+LP_{j},y)U_{x,\mu}^{2^*-1}\\
&=\int_{B_{\delta}(x)}\Gamma(z+LP_{j},y)U_{x,\mu}^{2^*-1}+O\Bigl(\int_{\Omega \backslash B_{\delta}(x)}\Gamma(y,z+LP_{j})\frac{1}{|z-x|^{N+2}}\frac{1}{\mu^{\frac{N+2}{2}}}\Bigr)\\
&=\frac{1}{\mu^{\frac{N-2}{2}}}{\int}_{B_{\delta\mu}(0)}\Gamma(\mu^{-1}z+x+LP_{j},y)U_{0,1}^{2^*-1}+O\Bigl(\frac{1}{|LP_{j}|^{N-2}}\frac{1}{\mu^{\frac{N+2}{2}}}\Bigr)\\
&=\frac{B_1\Gamma(x+LP_{j},y)}{\mu^{\frac{N-2}{2}}}+O\Bigl(\frac{1}{|LP_{j}|^{N}}\frac{\ln\mu}{\mu^{\frac{N+2}{2}}}+\frac{1}{|LP_{j}|^{N-2}}\frac{1}{\mu^{\frac{N+2}{2}}}\Bigr),
\end{align*}
where $B_1=\ds\int_{\mathbb{R}^{N}}U_{0,1}^{2^{*-1}}$. So we have proved
$$\varphi_1(y)=-B_1\sum_{j=1}^{\infty}\frac{\Gamma(x+L P_{j},y)}{\mu^{\frac{N-2}{2}}}+O\Bigl(\frac{1}{L^{N-2}\mu^{\frac{N+2}{2}}}\Bigr).$$

To prove \eqref{eqs2.2}, we have
\begin{equation}\label{eqs2.5}
\frac{\partial PU_{x,\mu}(y)}{\partial x_{h}}=(2^{*}-1)\sum_{j=0}^{\infty}\int_{\Omega}\Gamma(z+LP_{j},y)U_{x,\mu}^{2^{*}-2}\frac{\partial U_{x,\mu}}{\partial x_{h}}.
\end{equation}

For $j=0,$ it holds
\begin{align*}
&(2^{*}-1)\int_{\Omega}\Gamma(z+LP_{0},y)U_{x,\mu}^{2^{*}-2}\frac{\partial U_{x,\mu}}{\partial x_{h}}\\
&=(2^{*}-1)\int_{\R^{N}}\Gamma(z,y)U_{x,\mu}^{2^{*}-2}\frac{\partial U_{x,\mu}}{\partial x_{h}}-(2^{*}-1)\int_{\R^N \backslash \Omega}\Gamma(z,y)U_{x,\mu}^{2^{*}-2}\frac{\partial U_{x,\mu}}{\partial x_{h}}\\
&=\frac{\partial U_{x,\mu}}{\partial x_{h}}+O\Bigl(\frac{1}{L^{N-1}\mu^{\frac{N+2}{2}}}\Bigr),
\end{align*}
while for $j\neq0$, we have
\begin{align*}
&(2^{*}-1)\int_{\Omega}\Gamma(z+LP_{j},y)U_{x,\mu}^{2^{*}-2}{\frac{\partial U_{x,\mu}}{\partial x_{h}}}\cr
=&-(2^{*}-1)\int_{B_{\delta}(x)}\Gamma(z+LP_{j},y)U_{x,\mu}^{2^{*}-2}{\frac{\partial U_{x,\mu}}{\partial z_{h}}}
+O\Bigl(\int_{\Omega\backslash B_{\delta}(x)}\Gamma(z+LP_{j},y)\frac{1}{|z-x|^{N+3}}\frac{1}{\mu^{\frac{N+2}{2}}}\Bigr)\cr
=&\int_{B_{\delta}(x)}{\frac{\partial\Gamma(z+L P_{j},y)}{\partial z_h}}U_{x,\mu}^{2^{*}-1}+O\Bigl({\frac{1}{|LP_{j}|^{N-2}}}{\frac{1}{\mu^{\frac{N+2}{2}}}}\Bigr)\cr
=&B_1\frac{\partial\Gamma(x+L P_{j},y)}{\partial x_{h}}\frac{1}{\mu^{\frac{N-2}{2}}}+O\Bigl(\frac{1}{|LP_{j}|^{N}}\frac{\ln\mu}{\mu^{\frac{N+2}{2}}}+\frac{1}{|LP_{j}|^{N-2}}\frac{1}{\mu^{\frac{N+2}{2}}}\Bigr).
\end{align*}
So we have proved
$$\frac{\partial\varphi_1}{\partial x_{h}}(y)=-B_1\frac{1}{\mu^{\frac{N-2}{2}}}\sum_{j=1}^{\infty}\frac{\partial\Gamma(x+L P_{j},y)}
{\partial x_{h}}+O\Bigl(\frac{1}{L^{N-2}\mu^{\frac{N+2}{2}}}\Bigr).$$

Using
\begin{align*}
\frac{\partial PU_{x,\mu}(y)}{\partial\mu}=(2^{*}-1)\sum_{j=0}^{\infty}\int_{\Omega}\Gamma(z+LP_{j},y)U_{x,\mu}^{2^{*}-2}\frac{\partial U_{x,\mu}}{\partial\mu},
\end{align*}
in a similar way, we can prove
$$\frac{\partial\varphi_1}{\partial\mu}(y)=B_1\sum_{j=1}^{\infty}\frac{(N-2)\Gamma(x+LP_{j},y)}{2\mu^{\frac{N}{2}}}
+\frac{1}{\mu}O\Bigl(\frac{1}{L^{N-2}\mu^{\frac{N+2}{2}}}\Bigr).$$
Similarly, we can estimate $\varphi_2$.
\end{proof}

Using Lemma 2.2, we are ready to estimate the following quantities, which will be used to determine $x_{L}$ and $\mu_{L}$ for the bubbling solution of \eqref{eqs1.1}:
\begin{equation}\label{eqs2.6}
-\int_{\Omega}\Delta(PU_{x,\mu})\partial_{h}(PU_{x,\mu})-\int_{\Omega}K_1(y)(PU_{x,\mu})^{2^{*}-1}\partial_{h}(PU_{x,\mu})-\frac{1}{2}\int_{\Omega} (PU_{x,\mu})^{\frac{2^*}{2}-1}(PV_{x,\mu})^\frac{2^*}{2}\partial_{h}(PU_{x,\mu})
\end{equation}
and
\begin{equation}\label{eqs2.6(1)}
-\int_{\Omega}\Delta(PV_{x,\mu})\partial_{h}(PV_{x,\mu})-\int_{\Omega}K_2(y)(PV_{x,\mu})^{2^{*}-1}\partial_{h}(PV_{x,\mu})-\frac{1}{2}\int_{\Omega} (PV_{x,\mu})^{\frac{2^*}{2}-1}(PU_{x,\mu})^\frac{2^*}{2}\partial_{h}(PV_{x,\mu}),
\end{equation}
 where
$$
\partial_{h}(PU_{x,\mu})=\frac{\partial PU_{x,\mu}}{\partial x_h} \quad \hbox{if} \quad h=1,\ldots,N, \quad \hbox{and} \quad
\partial_{N+1}(PU_{x,\mu})=\frac{\partial PU_{x,\mu}}{\partial \mu}
$$
and
$$
\partial_{h}(PV_{x,\mu})=\frac{\partial PV_{x,\mu}}{\partial x_h}\quad \hbox{if} \quad h=1,\ldots,N, \quad \hbox{and} \quad
\partial_{N+1}(PV_{x,\mu})=\frac{\partial PV_{x,\mu}}{\partial \mu}.
$$
We will use the following notations:
\begin{equation}\label{eqs2.7}
\alpha(h)=1,\;\;\;h=1,\ldots,N,\,\,\,\alpha(N+1)=-1.
\end{equation}
First, we have
\begin{align}\begin{split}\label{eqs2.8}
&-\int_{\Omega}\Bigl(\Delta(PU_{x,\mu})\partial_{h}(PU_{x,\mu})-K_1(y)(PU_{x,\mu})^{2^{*}-1}\partial_{h}(PU_{x,\mu})\Bigr)\\
&\quad -\frac{1}{2}\int_{\Omega} (PU_{x,\mu})^{\frac{2^*}{2}-1}(PV_{x,\mu})^\frac{2^*}{2}\partial_{h}(PU_{x,\mu})\cr
=&\int_{\Omega}(U_{x,\mu}^{2^{*}-1}+\frac{1}{2}(U_{x,\mu})^{\frac{2^*}{2}-1}(V_{x,\mu})^\frac{2^*}{2})\partial_{h}(P U_{x,\mu})-\int_{\Omega}K_1(y)(PU_{x,\mu})^{2^{*}-1}\partial_{h}(PU_{x,\mu})\cr
&\,\,\,\,\,-\frac{1}{2}\int_{\Omega} (PU_{x,\mu})^{\frac{2^*}{2}-1}(PV_{x,\mu})^\frac{2^*}{2}\partial_{h}(PU_{x,\mu})\cr
=&\int_{\Omega}\Bigl(U_{x,\mu}^{2^{*}-1}\partial_{h}(PU_{x,\mu})-K_1(y)(PU_{x,\mu})^{2^{*}-1}\partial_{h}(PU_{x,\mu})\Bigr)
-\frac{1}{2}\int_{\Omega} \Bigl((PU_{x,\mu})^{\frac{2^*}{2}-1}(PV_{x,\mu})^\frac{2^*}{2}\cr
&\quad -U_{x,\mu}^{\frac{2^*}{2}-1}V_{x,\mu}^\frac{2^*}{2}\Bigr)\partial_{h}(PU_{x,\mu}).
\end{split}
\end{align}
Note that $|y-x_j|\geq |y-x|$ if $y\in\Omega$ and $x$ is close to 0. Thus, for $y\in\Omega\backslash B_{1}(x),$ it holds
\begin{align*} \sum_{j=1}^{\infty}U_{x_{j},\mu}&\le\frac{C\mu^{\frac{N-2}{2}}}{(1+\mu|y-x|)^{\frac{N-2}{2}+\theta}}\sum_{j=1}^{\infty}\frac{1}{(1+\mu|y-x_{j}|)^{\frac{N-2}{2}-\theta}}\\
&\le\frac{C\mu^{\frac{N-2}{2}}}{(1+\mu|y-x|)^{\frac{N-2}{2}+\theta}}\frac{1}{(\mu L)^{\frac{N-2}{2}-\theta}}.
\end{align*}

Moreover, using Lemma 2.1, we see
\begin{align}\begin{split}\label{eqs2.9}
\Bigl|\partial_{h}(PU_{x,\mu})\Bigr|=&\Bigl|\int_{\Omega}G(z,y)\partial_{h}(U_{x,\mu}^{2^{*}-1}+\frac{1}{2}U_{x,\mu}^{\frac{2^*}{2}-1}V_{x,\mu}^\frac{2^*}{2})\Bigr|\\
\leq& C\mu^{\alpha(h)}\int_{\Omega}G(z,y)(U_{x,\mu}^{2^{*}-1}+\frac{1}{2}U_{x,\mu}^{\frac{2^*}{2}-1}V_{x,\mu}^\frac{2^*}{2}) \cr
=&C\mu^{\alpha(h)}PU_{x,\mu}\leq C\mu^{\alpha(h)}\sum_{j=0}^{\infty}U_{x_j,\mu}.
\end{split}
\end{align}
So we have
\begin{align}\begin{split}\label{eqs2.10}
&\Bigl|\int_{\Omega \backslash B_{1}(x)} U_{x, \mu}^{2^{*}-1} \partial_{h}(PU_{x, \mu})-\int_{\Omega \backslash B_{1}(x)} K_1(y)(PU_{x, \mu})^{2^{*}-1} \partial_{h}(PU_{x, \mu})\\
&\quad \quad -\frac{1}{2}\int_{\Omega \backslash B_{1}(x)} ((PU_{x,\mu})^{\frac{2^*}{2}-1}(PV_{x,\mu})^\frac{2^*}{2}-U_{x,\mu}^{\frac{2^*}{2}-1}V_{x,\mu}^\frac{2^*}{2})\partial_{h}(PU_{x,\mu})\Bigr| \\
&\quad \leq C \mu^{\alpha(h)} \int_{\Omega \backslash B_{1}(x)}\Bigr(\sum_{j=0}^{\infty} U_{x_{j}, \mu}\Bigr)^{2^{*}} \\
&\quad \leq C \mu^{\alpha(h)+N} \int_{\Omega \backslash B_{1}(x)}\Bigl[\frac{1}{(1+\mu|y-x|)^{N-2}}+\frac{1}{(1+\mu|y-x|)^{\frac{N-2}{2}+\theta}} \frac{1}{(\mu L)^{\frac{N-2}{2}-\theta}}\Bigr]^{2^{*}} \\
&\quad \leq C \mu^{\alpha(h)-N}.
\end{split}
\end{align}
We write
\begin{align}\begin{split}\label{eqs2.11}
&\int_{B_{1}(x)} U_{x, \mu}^{2^{*}-1} \partial_{h}(PU_{x, \mu})-\int_{B_{1}(x)} K_1(y)(PU_{x, \mu})^{2^{*}-1} \partial_{h}(PU_{x, \mu})\cr
&\quad -\frac{1}{2}\int_{B_{1}(x)} \Bigl[(PU_{x,\mu})^{\frac{2^*}{2}-1}(PV_{x,\mu})^\frac{2^*}{2}-U_{x,\mu}^{\frac{2^*}{2}-1}V_{x,\mu}^\frac{2^*}{2}\Bigr]\partial_{h}(PU_{x,\mu}) \cr
=&\frac{2+\kappa}{2}\int_{B_{1}(x)}\Bigl(U_{x, \mu}^{2^{*}-1}-(PU_{x, \mu})^{2^{*}-1}\Bigr) \partial_{h}(PU_{x, \mu})-\int_{B_{1}(x)}(K_1(y)-1)(PU_{x, \mu})^{2^{*}-1} \partial_{h}(PU_{x, \mu}) \cr
:=&\frac{2+\kappa}{2}J_{1}-J_{2}.
\end{split}
\end{align}
By Lemma 2.2, we have
\begin{equation}\label{eqs2.12}
\int_{B_{1}(x)} U_{x, \mu}^{2^{*}-2}|\varphi_{x, \mu}|^{2} \leq \frac{C}{\mu^{N-2} L^{2(N-2)}} \int_{B_{1}(x)} U_{x, \mu}^{2^{*}-2} \leq \frac{C}{\mu^{N}}
\end{equation}
and
\begin{equation}\label{eqs2.13}
\int_{B_{1}(x)}|K_1(y)-1| U_{x, \mu}^{2^{*}-1}|\varphi_{x, \mu}| \leq \frac{C}{\mu^{\frac{N-2}{2}} L^{N-2}} \int_{B_{1}(x)}|y|^{\beta_1} U_{x, \mu}^{2^{*}-1} \leq \frac{C}{\mu^{N}}+\frac{C|x|^{\beta_1}}{(\mu L)^{N-2}} .
\end{equation}
Thus, we get
\begin{align}\begin{split}\label{eqs2.14}
J_{1} & =\int_{B_{1}(x)}\Bigl(U_{x, \mu}^{2^{*}-1}-(PU_{x, \mu})^{2^{*}-1}\Bigr) \partial_{h} U_{x, \mu}-\int_{B_{1}(x)}\Bigl(U_{x, \mu}^{2^{*}-1}-(PU_{x, \mu})^{2^{*}-1}\Bigr) \partial_{h} \varphi_{1} \\
& =(2^{*}-1) \int_{B_{1}(x)} U_{x, \mu}^{2^{*}-2} \varphi_{1} \partial_{h} U_{x, \mu}+O\Bigl(\mu^{\alpha(h)} \int_{B_{1}(x)} U_{x, \mu}^{2^{*}-2}|\varphi_{x, \mu}|^{2}\Bigr) \\
& =(2^{*}-1) \int_{B_{1}(x)} U_{x, \mu}^{2^{*}-2} \varphi_{1} \partial_{h} U_{x, \mu}+\mu^{\alpha(h)} O\Bigl(\frac{1}{\mu^{N}}\Bigr)
\end{split}
\end{align}
and
\begin{align}\begin{split}\label{eqs2.15}
J_{2} & =\int_{B_{1}(x)}(K_1(y)-1) U_{x, \mu}^{2^{*}-1} \partial_{h} U_{x, \mu}-\int_{B_{1}(x)}(K_1(y)-1) U_{x, \mu}^{2^{*}-1} \partial_{h} \varphi_{1} \\
& =\int_{B_{1}(x)}(K_1(y)-1) U_{x, \mu}^{2^{*}-1} \partial_{h} U_{x, \mu}+\mu^{\alpha(h)} O\Bigl(\frac{1}{\mu^{N}}+\frac{|x|^{\beta_1}}{\mu^{\beta_1}}\Bigr).
\end{split}
\end{align}

So, combining \eqref{eqs2.8}, \eqref{eqs2.10}, \eqref{eqs2.11}, \eqref{eqs2.14} and \eqref{eqs2.15}, we obtain
\begin{align}\begin{split}\label{eqs2.17}
&-\int_{\Omega}\Delta(PU_{x,\mu})\partial_{h}(PU_{x,\mu})-\int_{\Omega}K_1(y)(PU_{x,\mu})^{2^{*}-1}\partial_{h}(PU_{x,\mu})\\
&\quad -\frac{1}{2}\int_{\Omega} (PU_{x,\mu})^{\frac{2^*}{2}-1}(PV_{x,\mu})^\frac{2^*}{2}\partial_{h}(PU_{x,\mu}) \\
=&\frac{(2+\kappa)(2^{*}-1)}{2} \int_{B_{1}(x)} U_{x, \mu}^{2^{*}-2} \varphi_{1} \partial_{h} U_{x, \mu}-\int_{B_{1}(x)}(K_1(y)-1) U_{x, \mu}^{2^{*}-1} \partial_{h} U_{x, \mu} +\mu^{\alpha(h)}O\Bigl(\frac{1}{\mu^{N}}\Bigr),
\end{split}
\end{align}
since  $|x|^{\beta} \mu^{-\beta} \leq \mu^{-2 \beta}=o(\mu^{-N})$.

\begin{proposition}\label{pro2.3}
For  h=1, \ldots, N , we have
\begin{align}\label{eqs2.18}
\eqref{eqs2.6}=\frac{B_{1h}\mu x_{h}}{\mu^{\beta_{1h}-1}}+O\Bigl(\frac{1}{\mu^{\beta_{1M}-1+\sigma}}+\frac{\left(\mu x_{h}\right)^{2}}{\mu^{\beta_{1h}-1}}+\frac{1}{\mu^{\beta_1} L}\Bigr)
\end{align}
and
\begin{align}\label{eqs2.19}
\eqref{eqs2.6(1)}=\frac{B_{2h}\mu x_{h}}{\mu^{\beta_{2h}-1}}+O\Bigl(\frac{1}{\mu^{\beta_{2M}-1+\sigma}}+\frac{\left(\mu x_{h}\right)^{2}}{\mu^{\beta_{2h}-1}}+\frac{1}{\mu^{\beta_2} L}\Bigr),
\end{align}
where  $B_{1h}$ and $B_{2h}$  are non-zero constants,  $h=1, \ldots, N$.
\end{proposition}

\begin{proof}
For $h=1, \ldots, N$, we have
\begin{align}\begin{split}\label{eqs2.20}
&(2^{*}-1) \int_{B_{1}(x)} U_{x, \mu}^{2^{*}-2} \varphi_{1} \frac{\partial U_{x, \mu}}{\partial x_{h}}dy\\
&\quad =-\frac{(2^{*}-1) B_1}{\mu^{\frac{N-2}{2}}} \sum_{j=1}^{\infty} \int_{B_{1}(x)} U_{x, \mu}^{2^{*}-2} \frac{\partial U_{x, \mu}}{\partial x_{h}} \Gamma(x+L P_{j}, y)+O\Bigl(\frac{1}{\mu^{N-1} L^{N-2}}\Bigr) \\
&\quad =\frac{B_1}{\mu^{\frac{N-2}{2}}} \sum_{j=1}^{\infty} \int_{B_{1}(x)} \frac{\partial U_{x, \mu}^{2^{*}-1}}{\partial y_{h}} \Gamma(x+LP_{j}, y)+O\Bigl(\frac{1}{\mu^{\beta_1+1}}\Bigr) \\
&\quad =-\frac{B_1}{\mu^{\frac{N-2}{2}}} \sum_{j=1}^{\infty} \int_{B_{1}(x)} U_{x, \mu}^{2^{*}-1} \frac{\partial \Gamma(x+LP_{j}, y)}{\partial y_{h}}+O\Bigl(\frac{1}{\mu^{\beta_1+1}}\Bigr) \\
&\quad =-\left.\frac{B_1^{2}}{\mu^{N-2} L^{N-1}} \sum_{j=1}^{\infty} \frac{\partial \Gamma(P_{j}, y)}{\partial y_{h}}\right|_{y=0}+O\Bigl(\frac{1}{\mu^{\beta_1+1}}\Bigr)=O\Bigl(\frac{1}{\mu^{\beta_1} L}\Bigr).
\end{split}
\end{align}
On the other hand, we have
\begin{align}\begin{split}\label{eqs2.21}
&\int_{B_{1}(x)}(K_1(y)-1) U_{x, \mu}^{2^{*}-1} \frac{\partial U_{x, \mu}}{\partial x_{h}}\\
&\quad =\int_{B_{\delta}(x)} \sum_{i=1}^{N} a_{i}|y_{i}|^{\beta_{1i}} U_{x, \mu}^{2^{*}-1} \frac{\partial U_{x, \mu}}{\partial x_{h}}+O\Bigl(\frac{1}{\mu^{\beta_{1M}-1+\sigma}}\Bigr) \\
&\quad =-\int_{B_{\delta \mu}(0)} \sum_{i=1}^{N} a_{i}|y_{i}+\mu x_{i}|^{\beta_{1i}} \frac{1}{\mu^{\beta_{1i}-1}} U_{0,1}^{2^{*}-1} \frac{\partial U_{0,1}}{\partial y_{h}}+O\Bigl(\frac{1}{\mu^{\beta_{1M}-1+\sigma}}\Bigr) \\
&\quad =-\int_{B_{\delta \mu}(0)} a_{h}|y_{h}+\mu x_{h}|^{\beta_{1h}} \frac{1}{\mu^{\beta_{1h}-1}} U_{0,1}^{2^{*}-1} \frac{\partial U_{0,1}}{\partial y_{h}}+O\Bigl(\frac{1}{\mu^{\beta_{1M}-1+\sigma}}\Bigr) \\
&\quad =-\frac{a_{h}\mu x_{h}}{\mu^{\beta_{1h}-1}} \int_{B_{\delta \mu}(0)}\beta_{1h}|y_{h}|^{\beta_{1h}-2} y_{h} U_{0,1}^{2^{*}-1} \frac{\partial U_{0,1}}{\partial y_{h}}+O\Bigl(\frac{1}{\mu^{\beta_{1M}-1+\sigma}}+\frac{(\mu x_{h})^{2}}{\mu^{\beta_{1h}-1}}\Bigr).
\end{split}
\end{align}
Combining \eqref{eqs2.17}, \eqref{eqs2.20} and \eqref{eqs2.21}, we obtain the desired result in Proposition 2.3.
\end{proof}

Throughout this paper, we will use the following notations
$$\psi_{0}=\left.\frac{\partial U_{0, \lambda}}{\partial \mu}\right|_{\lambda=1}, \quad \psi_{j}=\frac{\partial U_{0,1}}{\partial y_{j}},\quad\,\,\,\xi _{0}=\left.\frac{\partial V_{0, \lambda}}{\partial \mu}\right|_{\mu=1}, \quad \xi _{j}=\frac{\partial V_{0,1}}{\partial y_{j}},  \quad j=1, \ldots, N.$$
\begin{proposition}\label{pro2.4}
It holds
\begin{align}
- & \int_{\Omega} \Delta(PU_{x, \mu}) \frac{\partial PU_{x, \mu}}{\partial \mu}-\int_{\Omega} K_1(y)(PU_{x, \mu})^{2^{*}-1} \frac{\partial PU_{x, \mu}}{\partial \mu} -\frac{1}{2}\int_{\Omega} (PU_{x,\mu})^{\frac{2^*}{2}-1}(PV_{x,\mu})^\frac{2^*}{2}\partial_{\mu}(PU_{x,\mu})\cr
= & -\frac{1}{\mu^{\beta_1+1}} \sum_{i \in J} a_{1i} \ds\int_{\mathbb{R}^{N}}|y_i|^{\beta_1} U_{0,1}^{2^{*}-1} \psi_{0}-\frac{(2+\kappa)(2^{*}-1) B_1 \ds\int_{\mathbb{R}^{N}} U_{0,1}^{2^{*}-2} \psi_{0}}{2\mu^{N-1} L^{N-2}} \sum_{j=1}^{\infty} \Gamma(P_{j}, 0)
+o\Bigl(\frac{1}{\mu^{\beta_1+1}}\Bigr)
\end{align}
and
\begin{align}
- & \int_{\Omega} \Delta(PV_{x, \mu}) \frac{\partial PV_{x, \mu}}{\partial \mu}-\int_{\Omega} K_2(y)(PV_{x, \mu})^{2^{*}-1} \frac{\partial PV_{x, \mu}}{\partial \mu} -\frac{1}{2}\int_{\Omega} (PV_{x,\mu})^{\frac{2^*}{2}-1}(PU_{x,\mu})^\frac{2^*}{2}\partial_{\mu}(PV_{x,\mu})\cr
= & -\frac{1}{\mu^{\beta_2+1}} \sum_{i \in J} a_{2i} \int_{\mathbb{R}^{N}}|y_i|^{\beta_2} V_{0,1}^{2^{*}-1} \xi_{0}-\frac{(2+\kappa)(2^{*}-1) B_2 \ds\int_{\mathbb{R}^{N}} V_{0,1}^{2^{*}-2} \xi_{0}}{2\mu^{N-1} L^{N-2}} \sum_{j=1}^{\infty} \Gamma(P_{j}, 0)+o\Bigl(\frac{1}{\mu^{\beta_2+1}}\Bigr),
\end{align}
where $J_l=\left\{j: \beta_{lj}=\beta_l\right\}, l=1,2.$
\end{proposition}

\begin{proof}
Similar to the proof of \eqref{eqs2.20}, we have

\begin{align}\begin{split}\label{eqs2.24}
\int_{B_{1}(x)} U_{x, \mu}^{2^{*}-2} \varphi_{1} \frac{\partial U_{x, \mu}}{\partial \mu} dy
&=-\frac{B_1}{\mu^{\frac{N-2}{2}}} \sum_{j=1}^{\infty} \int_{B_{1}(x)} U_{x, \mu}^{2^{*}-2} \frac{\partial U_{x, \mu}}{\partial \mu} \Gamma\left(x+L P_{j}, y\right)+O\Bigl(\frac{1}{\mu^{N+1} L^{N-2}}\Bigr) \cr
&=-\frac{B_1}{\mu^{\frac{N-2}{2}}} \sum_{j=1}^{\infty} \int_{\mathbb{R}^{N}} U_{x, \mu}^{2^{*}-2} \frac{\partial U_{x, \mu}}{\partial \mu} \Gamma\left(x+L P_{j}, x\right)+O\Bigl(\frac{1}{\mu^{\beta_1+2}}\Bigr) \cr
&=-\frac{B_1 \ds\int_{\mathbb{R}^{N}} U_{0,1}^{2^{*}-2} \psi_{0}}{\mu^{N-1} L^{N-2}} \sum_{j=1}^{\infty} \Gamma\left(P_{j}, 0\right)+O\Bigl(\frac{1}{\mu^{\beta_1+2}}\Bigr).
\end{split}
\end{align}
On the other hand, we have
\begin{align}\begin{split}\label{eqs2.25}
\int_{B_{1}(x)}(K_1(y)-1) U_{x, \mu}^{2^{*}-1} \frac{\partial U_{x, \mu}}{\partial \mu}
&=\int_{B_{\delta}(x)} \sum_{i=1}^{N} a_{1i}\left|y_{i}\right|^{\beta_{i}} U_{x, \mu}^{2^{*}-1} \frac{\partial U_{x, \mu}}{\partial \mu}+O\Bigl(\frac{1}{\mu^{\beta_{1M}+1+\sigma}}\Bigr) \cr
&=\int_{B_{\delta \mu}(0)} \sum_{i=1}^{N} a_{1i}\left|y_{i}+\mu x_{i}\right|^{\beta_{i}} \frac{1}{\mu^{\beta_{i}+1}} U_{0,1}^{2^{*}-1} \psi_{0}+O\Bigl(\frac{1}{\mu^{\beta_{1M}+1+\sigma}}\Bigr) \cr
&=\int_{B_{\delta \mu}(0)} \sum_{i=1}^{N} a_{1i}\left|y_{i}\right|^{\beta_{i}} \frac{1}{\mu^{\beta_{i}+1}} U_{0,1}^{2^{*}-1} \psi_{0}+o\Bigl(\frac{1}{\mu^{\beta_1+1}}\Bigr) \cr
&=\frac{1}{\mu^{\beta+1}} \sum_{i \in J} a_{1i} \int_{\mathbb{R}^{N}}|y_i|^{\beta} U_{0,1}^{2^{*}-1} \psi_{0}+o\Bigl(\frac{1}{\mu^{\beta_1+1}}\Bigr).
\end{split}
\end{align}
So the result follows.
\end{proof}

\section{the invertibility of the linear operator}
We want to find a solution of the form  $(PU_{x,\mu}+\omega_{1},PV_{x,\mu}+\omega_{2})$  for \eqref{eqs1.1}. Then  $(\omega_{1},\omega_{2})$  satisfies
\begin{equation}\label{eqs3.1}
L(\omega_{1},\omega_{2})=\ell+N(\omega_{1},\omega_{2}),
\end{equation}
where
\begin{equation}\label{eqs3.2}
L(\omega_{1},\omega_{2})=(L_1(\omega_{1},\omega_{2}),L_2(\omega_{1},\omega_{2})),
\end{equation}

\begin{equation*}\begin{split}\label{eqs3.2(1)}
L_1(\omega_{1},\omega_{2})=&-\Delta\omega_{1}-(2^{*}-1) K_1(y)(PU_{x, \mu})^{2^{*}-2} \omega_{1}-\frac{1}{N-2}(PU_{x,\mu})^{\frac{4-N}{N-2}}(PV_{x,\mu})^{\frac{N}{N-2}}\omega_{1}\\
& -\frac{N}{2(N-2)}(PU_{x,\mu})^{\frac{2}{N-2}}(PV_{x,\mu})^{\frac{2}{N-2}}\omega_{2},
\end{split}
\end{equation*}
\begin{equation*}\begin{split}\label{eqs3.2(2)}
L_2(\omega_{1},\omega_{2})=&-\Delta \omega_{2}-(2^{*}-1) K_2(y)(PV_{x, \mu})^{2^{*}-2} \omega_{2}-\frac{1}{N-2}(PV_{x,\mu})^{\frac{4-N}{N-2}}(PU_{x,\mu})^{\frac{N}{N-2}}\omega_{2}\\
& -\frac{N}{2(N-2)}(PV_{x,\mu})^{\frac{2}{N-2}}(PU_{x,\mu})^{\frac{2}{N-2}}\omega_{1},
\end{split}
\end{equation*}
\begin{equation}\label{eqs3.3}
\ell=(\ell_1,\ell_2),
\end{equation}

\begin{equation*}\label{wqf1}
\ell_1=K_1(y)(PU_{x, \mu})^{2^{*}-1}-U_{x, \mu}^{2^{*}-1}-\frac{1}{2}\Bigl(U_{x,\mu}^{\frac{2^*}{2}-1}V_{x, \mu}^\frac{2^*}{2}-(PU_{x,\mu})^{\frac{2^*}{2}-1}(PV_{x,\mu})^\frac{2^*}{2}\Bigr),
\end{equation*}
\begin{equation*}\label{wqf2}
\ell_2=K_2(y)(PV_{x, \mu})^{2^{*}-1}-V_{x, \mu}^{2^{*}-1}-\frac{1}{2}\Bigl(V_{x,\mu}^{\frac{2^*}{2}-1}U_{x, \mu}^\frac{2^*}{2}-(PV_{x,\mu})^{\frac{2^*}{2}-1}(PU_{x,\mu})^\frac{2^*}{2}\Bigr),
\end{equation*}

\begin{equation}\label{eqs3.4}
N(\omega_{1},\omega_{2})=(N_1(\omega_{1},\omega_{2}),N_2(\omega_{1},\omega_{2})),
\end{equation}

\begin{equation*}\begin{split}\label{eqs3.3(1)}
N_1(\omega_{1},\omega_{2})&=K_1(y)\Bigl((PU_{x, \mu}+\omega_{1})_{+}^{2^{*}-1}-(PU_{x, \mu})^{2^{*}-1}-(2^{*}-1)(PU_{x, \mu})^{2^{*}-2}\omega_{1}\Bigr)\\
&\quad +\frac{1}{2}\Bigl((PU_{x,\mu}+\omega_{1})^{\frac{2}{N-2}}(PV_{x,\mu}+\omega_{2})^{\frac{N}{N-2}}-(PU_{x,\mu})^{\frac{2}{N-2}}(PV_{x,\mu})^{\frac{N}{N-2}}\Bigr)\\
&\quad -\frac{1}{N-2}(PU_{x,\mu})^{\frac{4-N}{N-2}}(PV_{x,\mu})^{\frac{N}{N-2}}\omega_{1}-\frac{N}{2(N-2)}(PU_{x,\mu})^{\frac{2}{N-2}}(PV_{x,\mu})^{\frac{2}{N-2}}\omega_{2}
\end{split}
\end{equation*}
and
\begin{equation*}\begin{split}\label{eqs3.3(2)}
N_2(\omega_{1},\omega_{2})
&=K_2(y)\Bigl((PV_{x,\mu}+\omega_{2})_{+}^{2^{*}-1}-(PV_{x, \mu})^{2^{*}-1}-(2^{*}-1)(PV_{x, \mu})^{2^{*}-2}\omega_{2}\Bigr)\\
&\quad +\frac{1}{2}\Bigl((PV_{x,\mu}+\omega_{2})^{\frac{2}{N-2}}(PU_{x,\mu}+\omega_{1})^{\frac{N}{N-2}}-(PV_{x,\mu})^{\frac{2}{N-2}}(PU_{x,\mu})^{\frac{N}{N-2}}\Bigr)\\
&\quad -\frac{1}{N-2}(PV_{x,\mu})^{\frac{4-N}{N-2}}(PU_{x,\mu})^{\frac{N}{N-2}}\omega_{2}-\frac{N}{2(N-2)}(PV_{x,\mu})^{\frac{2}{N-2}}(PU_{x,\mu})^{\frac{2}{N-2}}\omega_{1}.
\end{split}
\end{equation*}

We define the space  $\mathbf{Y}$  as
$$\mathbf{Y}=\Bigl\{(f,g):\|(f,g)\|_{**}<+\infty\Bigr\},$$
where  $\|(f,g)\|_{**}=\|f\|_{**}+\|g\|_{**}$  and $\|f\|_{**}$ is given in \eqref{eqs1.9}.
For  $(f,g)\in \mathbf{Y}$, we recall that
$$(u,v):=(-\Delta)^{-1} (f,g)=\Bigl(\int_{\Omega} G(z, y) f(z) dz,\int_{\Omega} G(z, y) g(z) dz\Bigr).$$
Note that if we extend  $(f,g)$  periodically to  $\mathbb{R}^{N}$ , we have
\begin{align*}
(-\Delta)^{-1} (f,g) & =\Bigl(\int_{\Omega} G(z, y) f(z) dz,\int_{\Omega} G(z, y) g(z) dz\Bigr)\\
&=\Bigl(\sum_{j=0}^{\infty} \int_{\Omega} \Gamma\left(z+L P_{j}, y\right) f(z) dz,\sum_{j=0}^{\infty} \int_{\Omega} \Gamma\left(z+L P_{j}, y\right) g(z) dz \Bigr)\\
&=\Bigl(\sum_{j=0}^{\infty} \int_{\Omega-L P_{j}} \Gamma(z, y) f(z) d z,\sum_{j=0}^{\infty} \int_{\Omega-L P_{j}} \Gamma(z, y) g(z) d z\Bigr)\\
&=\Bigl(\int_{\mathbb{R}^{N}} \Gamma(z, y) f(z) dz,\int_{\mathbb{R}^{N}} \Gamma(z, y) g(z) dz\Bigr).
\end{align*}

\begin{lemma}\label{lm3.1}
$(-\Delta)^{-1}$  is a bounded linear operator from  $\mathbf{Y}$  to  $\mathbf{X}$ .
\end{lemma}
\begin{proof}
 We extend the function  $\sigma(z)$  periodically to  $\mathbb{R}^{N}$ . For any  $(f,g) \in \mathbf{Y}$ , we have
\begin{align*}
|u(y)| & \leq\|f\|_{**} \int_{\Omega} G(z, y) \sigma(z) \sum_{j=0}^{\infty} \frac{\mu^{\frac{N+2}{2}}}{\left(1+\mu\left|z-x_{j}\right|\right)^{\frac{N+2}{2}+\tau}} d z \\
& =\|f\|_{**} \int_{\mathbb{R}^{N}} \Gamma(z, y) \sigma(z) \sum_{j=0}^{\infty} \frac{\mu^{\frac{N+2}{2}}}{\left(1+\mu\left|z-x_{j}\right|\right)^{\frac{N+2}{2}+\tau}} d z \\
& \leq C\|f\|_{**} \sigma(y) \sum_{j=0}^{\infty} \frac{\mu^{\frac{N-2}{2}}}{\left(1+\mu\left|y-x_{j}\right|\right)^{\frac{N-2}{2}+\tau}}.
\end{align*}
This gives  $\|u\|_{*} \leq C\|f\|_{**}$. In the same way, we can get $\|v\|_{*} \leq C\|g\|_{**}$.
\end{proof}

For $x \in B_{1}(0)$  and  $\mu>0$  large, we denote
$$Y_{h}=\frac{\partial U_{x,\mu}}{\partial x_{h}},\quad h=1, \ldots, N,\, Y_{N+1}=\frac{\partial U_{x, \mu}}{\partial \mu},$$
$$Z_{h}=\frac{\partial V_{x,\mu}}{\partial x_{h}},\quad h=1, \ldots, N,\, Z_{N+1}=\frac{\partial V_{x,\mu}}{\partial \mu},$$
$$\partial_{h} PU_{x, \mu}=\frac{\partial PU_{x, \mu}}{\partial x_{h}}, \quad h=1, \ldots, N,\quad \partial_{N+1} PU_{x, \mu}=\frac{\partial PU_{x, \mu}}{\partial \mu}$$
and
$$\partial_{h} PV_{x, \mu}=\frac{\partial PV_{x, \mu}}{\partial x_{h}}, \quad h=1, \ldots, N,\quad \partial_{N+1} PV_{x, \mu}=\frac{\partial PV_{x, \mu}}{\partial \mu}.$$
Let
\begin{equation}\label{eqs3.6}
\mathbf{E}:=\Bigl\{(u,v)\in \mathbf{X}: \int_{\Omega} u U_{x, \mu}^{2^{*}-2} Y_{h}=0,\,\int_{\Omega}v V_{x, \mu}^{2^{*}-2} Z_{h}=0, \,h=1, \ldots, N+1\Bigr\}
\end{equation}
and
\begin{equation}\label{eqs3.7}
\mathbf{F}:=\Bigl\{(f,g)\in \mathbf{Y}: \int_{\Omega} f \partial_{h} PU_{x, \mu}=0,\,\int_{\Omega} g \partial_{h} PV_{x, \mu}=0,\,h=1, \ldots, N+1\Bigr\}.
\end{equation}

Note that
\begin{align*}
\int_{\Omega} f \partial_{m} PU_{x, \mu}&=-\int_{\Omega} \Delta u \partial_{m} PU_{x, \mu}=-\int_{\Omega} \Delta(\partial_{m} PU_{x, \mu})u\\
&= \int_{\Omega} \Bigl((2^{*}-1) U_{x, \mu}^{2^{*}-2} Y_{m}+\frac{1}{N-2}U_{x, \mu}^{\frac{2^*}{2}-2}V_{x, \mu}^\frac{2^*}{2}Y_{m}+\frac{ N}{2(N-2)} U_{x, \mu}^{\frac{2^*}{2}-1}V_{x, \mu}^{\frac{2^*}{2}-1}Z_{m}\Bigr) u.
\end{align*}
Since $V_{x, \mu}=\kappa U_{x, \mu}$, we have
$$\int_{\Omega} f \partial_{m} PU_{x, \mu}=C\int_{\Omega} u U_{x, \mu}^{2^{*}-2} Y_{m}.$$
Similarly, we have
$$\int_{\Omega} g \partial_{m} PV_{x, \mu}=C\int_{\Omega} v V_{x, \mu}^{2^{*}-2} Z_{m}.$$
So, we see that  $(f,g) \in \mathbf{F}$  if and only if  $(u,v)=((-\Delta)^{-1} f, (-\Delta)^{-1} g)\in \mathbf{E}$.

Let  $\mathbf{P}$  be the operator defined as follows:
$$
\mathbf{P} (f,g)=\Bigl(f+\sum_{h=1}^{N+1} c_{h}U_{x, \mu}^{2^{*}-2}Y_{h},\quad g+\sum_{h=1}^{N+1} c_{h}V_{x, \mu}^{2^{*}-2}Z_{h} \Bigr), \quad (f,g) \in \mathbf{Y},
$$
where $c_{h}$  are chosen such that  $\mathbf{P}(f,g) \in \mathbf{F}$ . Then it is easy to check that
$$
\|\mathbf{P}(f,g)\|_{**} \leq C\|(f,g)\|_{**}.
$$

In the following, for fixed  $(x, \mu)$ , instead of \eqref{eqs3.1}, we consider the following problem
$$(-\Delta)^{-1}(\mathbf{P}L(\omega_{1},\omega_{2}))=(-\Delta)^{-1}(\mathbf{P}\ell)+(-\Delta)^{-1}(\mathbf{P}N(\omega_{1},\omega_{2})),$$
which is equivalent to
\begin{equation}\label{eqs3.8}
(\omega_{1},\omega_{2})-T (\omega_{1},\omega_{2})=(-\Delta)^{-1} (\mathbf{P}\ell+\mathbf{P} N(\omega_{1},\omega_{2})),
\end{equation}
where
\begin{equation}\begin{split}\label{eqs3.9}
T (\omega_{1},\omega_{2})&=(2^{*}-1)(-\Delta)^{-1}\mathbf{P}\Bigl(K_1(y)(PU_{x, \mu})^{2^{*}-2} \omega_{1}-\frac{1}{N+2}(PU_{x,\mu})^{\frac{4-N}{N-2}}(PV_{x,\mu})^{\frac{N}{N-2}}\omega_{1}\\
&\quad -\frac{N}{2(N+2)}(PU_{x,\mu})^{\frac{2}{N-2}}(PV_{x,\mu})^{\frac{2}{N-2}}\omega_{2},\quad
K_2(y)(PV_{x, \mu})^{2^{*}-2} \omega_{2}\\
&\quad -\frac{1}{N+2}(PV_{x,\mu})^{\frac{4-N}{N-2}}(PU_{x,\mu})^{\frac{N}{N-2}}\omega_{2}
-\frac{N}{2(N+2)}(PV_{x,\mu})^{\frac{2}{N-2}}(PU_{x,\mu})^{\frac{2}{N-2}}\omega_{1}\Bigr)\\
&=(2^{*}-1)(-\Delta)^{-1}\Bigl(K_1(y)(PU_{x, \mu})^{2^{*}-2} \omega_{1}-\frac{1}{N+2}(PU_{x,\mu})^{\frac{4-N}{N-2}}(PV_{x,\mu})^{\frac{N}{N-2}}\omega_{1}\\
&\quad -\frac{N}{2(N+2)}(PU_{x,\mu})^{\frac{2}{N-2}}(PV_{x,\mu})^{\frac{2}{N-2}}\omega_{2}+\sum_{h=1}^{N+1} c_{h}U_{x, \mu}^{2^{*}-2} Y_{h},\\
&\quad K_2(y)(PV_{x, \mu})^{2^{*}-2} \omega_{2}-\frac{1}{N+2}(PV_{x,\mu})^{\frac{4-N}{N-2}}(PU_{x,\mu})^{\frac{N}{N-2}}\omega_{2}\\
&\quad -\frac{N}{2(N+2)}(PV_{x,\mu})^{\frac{2}{N-2}}(PU_{x,\mu})^{\frac{2}{N-2}}\omega_{1}+\sum_{h=1}^{N+1} c_{h}V_{x, \mu}^{2^{*}-2} Z_{h}\Bigr),\quad (\omega_{1},\omega_{2}) \in \mathbf{E}.
\end{split}\end{equation}
It is easy to see that the constants  $c_{h}$  in \eqref{eqs3.9} are determined by
\begin{align}\begin{split}\label{eqs3.11}
&\sum_{m=1}^{N+1} c_{m} \int_{\Omega} U_{x, \mu}^{2^{*}-2} Y_{m} \partial_{h}(PU_{x, \mu})\\
&\quad =-\int_{\Omega} \Bigl(K_1(y)(PU_{x, \mu})^{2^{*}-2} \omega_1-\frac{1}{N+2}(PU_{x,\mu})^{\frac{4-N}{N-2}}(PV_{x,\mu})^{\frac{N}{N-2}}\omega_1\\
&\quad \quad -\frac{N}{2(N+2)}(PU_{x,\mu})^{\frac{2}{N-2}}(PV_{x,\mu})^{\frac{2}{N-2}}\omega_2\Bigr)\partial_{h}(PU_{x, \mu}),\quad h=1, \ldots, N+1 \\
\end{split}
\end{align}
and
\begin{align}\begin{split}\label{eqs3.12}
&\sum_{m=1}^{N+1} c_{m} \int_{\Omega} V_{x, \mu}^{2^{*}-2} Z_{m} \partial_{h}(PV_{x, \mu})\\
&\quad =-\int_{\Omega} \Bigl(K_2(y)(PV_{x, \mu})^{2^{*}-2} \omega_2-\frac{1}{N+2}(PV_{x,\mu})^{\frac{4-N}{N-2}}(PU_{x,\mu})^{\frac{N}{N-2}}\omega_2\\
&\quad \quad -\frac{N}{2(N+2)}(PV_{x,\mu})^{\frac{2}{N-2}}(PU_{x,\mu})^{\frac{2}{N-2}}\omega_1\Bigr)\partial_{h}(PV_{x, \mu}), \quad h=1, \ldots, N+1.
\end{split}
\end{align}

 We will use Fredholm theorem to prove that the linear operator $I-T$ is invertible. Then, it is sufficient to prove that T  is bounded and compact and $I-T$  is injective.
%\begin{proposition}\label{pro3.2}
%$I-T$  is a bijective bounded linear operator from  $\mathbf{E}$  to itself.
%\end{proposition}

First, we prove the following result.
\begin{lemma}\label{lem3.3}
For any $(\omega_1,\omega_2) \in \mathbf{E}$, we have
\begin{align}\begin{split}\label{eqs3.13}
&\Bigl|\int_{\Omega} \Bigl(K_1(y)(PU_{x, \mu})^{2^{*}-2} \omega_1-\frac{1}{N+2}(PU_{x,\mu})^{\frac{4-N}{N-2}}(PV_{x,\mu})^{\frac{N}{N-2}}\omega_1\\
&\quad -\frac{N}{2(N+2)}(PU_{x,\mu})^{\frac{2}{N-2}}(PV_{x,\mu})^{\frac{2}{N-2}}\omega_2\Bigr)\partial_{h}(PU_{x, \mu})\Bigr| \leq \frac{C \mu^{\alpha(h)}\|(\omega_1,\omega_2)\|_{*}}{\mu^{\beta_1}},
\end{split}
\end{align}
\begin{align}\begin{split}\label{eqs3.14}
&\Bigl|\int_{\Omega} \Bigl(K_2(y)(PV_{x, \mu})^{2^{*}-2} \omega_2-\frac{1}{N+2}(PV_{x,\mu})^{\frac{4-N}{N-2}}(PU_{x,\mu})^{\frac{N}{N-2}}\omega_2\\
&\quad -\frac{N}{2(N+2)}(PV_{x,\mu})^{\frac{2}{N-2}}(PU_{x,\mu})^{\frac{2}{N-2}}\omega_1\Bigr)\partial_{h}(PV_{x, \mu})\Bigr| \leq \frac{C \mu^{\alpha(h)}\|(\omega_1,\omega_2)\|_{*}}{\mu^{\beta_2}}
\end{split}
\end{align}
and
\begin{equation}\label{eqs3.15}
\left|c_{h}\right| \leq \frac{C}{\mu^{\alpha(h)+\beta}}\|(\omega_1,\omega_2)\|_{*},
\end{equation}
where  $\alpha(h)$  is defined in \eqref{eqs2.7}.
\end{lemma}

\begin{proof}
We first prove \eqref{eqs3.13}. Write
\begin{align}\begin{split}\label{eqs3.16}
&\int_{\Omega} \Bigl(K_1(y)(PU_{x, \mu})^{2^{*}-2} \omega_1-\frac{(PU_{x,\mu})^{\frac{4-N}{N-2}}(PV_{x,\mu})^{\frac{N}{N-2}}}{N+2}\omega_1\\
&\quad \quad -\frac{N(PU_{x,\mu})^{\frac{2}{N-2}}(PV_{x,\mu})^{\frac{2}{N-2}}}{2(N+2)}\omega_2\Bigr)\partial_{h}(PU_{x, \mu}) \\
&\quad=\int_{\Omega} K_1(y)\Bigl((PU_{x, \mu})^{2^{*}-2}-U_{x, \mu}^{2^{*}-2}\Bigr) \omega_1 \partial_h (PU_{x, \mu})+\int_{\Omega}(K_1(y)-1) U_{x, \mu}^{2^{*}-2} \omega_1 \partial_h (PU_{x, \mu}) \\
&\quad \quad+\int_{\Omega} U_{x, \mu}^{2^{*}-2} \omega_1 \partial_h (PU_{x, \mu})-\frac{1}{N+2}\int_{\Omega}(PU_{x,\mu})^{\frac{4-N}{N-2}}(PV_{x,\mu})^{\frac{N}{N-2}}\omega_1\partial_{h}(PU_{x, \mu})\\
&\quad \quad-\frac{N}{2(N+2)}\int_{\Omega}(PU_{x,\mu})^{\frac{2}{N-2}}(PV_{x,\mu})^{\frac{2}{N-2}}\omega_2 \partial_{h}(PU_{x, \mu})\\
&\quad=\int_{\Omega} K_1(y)\Bigl((PU_{x, \mu})^{2^{*}-2}-U_{x, \mu}^{2^{*}-2}\Bigr) \omega_1 \partial_h (PU_{x, \mu})+\int_{\Omega}(K_1(y)-1) U_{x, \mu}^{2^{*}-2} \omega_1 \partial_h (PU_{x, \mu})\\
&\quad \quad+\int_{\Omega} (C_1\omega_1 + C_2\omega_2)((PU_{x, \mu})^{2^{*}-2}-U_{x, \mu}^{2^{*}-2})\partial_h (PU_{x, \mu}) + \int_{\Omega} (C_3\omega_1+ C_4\omega_2)U_{x, \mu}^{2^{*}-2}\partial_h (PU_{x, \mu})\\
&\quad :=J_{1}+J_{2}+J_{3}+J_{4}.
\end{split}
\end{align}

By Lemma 2.1, we have
\begin{align}\begin{split}\label{eqs3.17}
\left|J_{1}\right| & \leq C \mu^{\alpha(h)} \int_{\Omega}\Bigl|(PU_{x, \mu})^{2^{*}-2}-U_{x, \mu}^{2^{*}-2}\Bigr||\omega_1| PU_{x, \mu} \\
& \leq C\|\omega_1\|_{*} \mu^{\alpha(h)} \int_{\Omega}\Bigl|(P U_{x, \mu})^{2^{*}-2}-U_{x, \mu}^{2^{*}-2}\Bigr| \sum_{j=0}^{\infty} \frac{\mu^{\frac{N-2}{2}}}{(1+\mu\left|y-x_{j}\right|)^{\frac{N-2}{2}+\tau}} P U_{x, \mu} \\
& \leq C\|\omega_1\|_{*} \mu^{\alpha(h)} \int_{\Omega}(P U_{x, \mu})^{2^{*}-2}\left|\varphi_{1}\right| \sum_{j=0}^{\infty} \frac{\mu^{\frac{N-2}{2}}}{(1+\mu\left|y-x_{j}\right|)^{\frac{N-2}{2}+\tau}} \\
& \leq C\|\omega_1\|_{*} \mu^{\alpha(h)} \int_{\Omega}\Bigl(\sum_{j=0}^{\infty} U_{x_{j}, \mu}\Bigr)^{2^{*}-2}\left|\varphi_{1}\right| \sum_{j=0}^{\infty} \frac{\mu^{\frac{N-2}{2}}}{(1+\mu\left|y-x_{j}\right|)^{\frac{N-2}{2}+\tau}}.
\end{split}
\end{align}
For  $y \in B_{1}(x)$, we have
\begin{align}\begin{split}\label{eqs3.18}
\sum_{j=0}^{\infty} U_{x_{j}, \mu} & \leq \frac{C \mu^{\frac{N-2}{2}}}{(1+\mu|y-x|)^{N-2}}+C \sum_{j=1}^{\infty} \frac{1}{\mu^{\frac{N-2}{2}}\left|y-x_{j}\right|^{N-2}} \\
& \leq \frac{C \mu^{\frac{N-2}{2}}}{(1+\mu|y-x|)^{N-2}}+\frac{C}{\mu^{\frac{N-2}{2}} L^{N-2}} \leq \frac{C \mu^{\frac{N-2}{2}}}{(1+\mu|y-x|)^{N-2}}
\end{split}
\end{align}
and
\begin{align}\begin{split}\label{eqs3.19}
\sum_{j=0}^{\infty} \frac{\mu^{\frac{N-2}{2}}}{\left(1+\mu\left|y-x_{j}\right|\right)^{\frac{N-2}{2}+\tau}} & \leq C \mu^{\frac{N-2}{2}}\Bigl(\frac{1}{(1+\mu|y-x|)^{\frac{N-2}{2}+\tau}}+\frac{1}{(\mu L)^{\frac{N-2}{2}+\tau}}\Bigr) \\
& \leq \frac{C \mu^{\frac{N-2}{2}}}{(1+\mu|y-x|)^{\frac{N-2}{2}+\tau}}.
\end{split}
\end{align}

Using \eqref{eqs3.18}, \eqref{eqs3.19}, and Lemma 2.2, we obtain
\begin{align}\begin{split}\label{eqs3.20}
&\int_{B_{1}(x)}\Bigl(\sum_{j=0}^{\infty} U_{x_{j}, \mu}\Bigr)^{2^{*}-2}\left|\varphi_{1}\right| \sum_{j=0}^{\infty} \frac{\mu^{\frac{N-2}{2}}}{\left(1+\mu\left|y-x_{j}\right|\right)^{\frac{N-2}{2}+\tau}} \\
&\quad \leq \frac{C}{\mu^{\frac{N-2}{2}} L^{N-2}} \int_{B_{1}(x)}\Bigl(\frac{\mu^{\frac{N-2}{2}}}{(1+\mu|y-x|)^{N-2}}\Bigr)^{2^{*}-2} \frac{\mu^{\frac{N-2}{2}}}{(1+\mu|y-x|)^{\frac{N-2}{2}+\tau}} \\
&\quad \leq \frac{C}{\mu^{N-2} L^{N-2}}.
\end{split}
\end{align}
On the other hand, using
$$\left|\varphi_{1}\right| \leq C \sum_{j=0}^{\infty} U_{x_{j}, \mu},$$
we find
\begin{align}\begin{split}\label{eqs3.21}
&\int_{\Omega \backslash B_{1}(x)}\Bigl(\sum_{j=0}^{\infty} U_{x_{j}, \mu}\Bigr)^{2^{*}-2}\left|\varphi_{1}\right| \sum_{j=0}^{\infty} \frac{\mu^{\frac{N-2}{2}}}{(1+\mu\left|y-x_{j}\right|)^{\frac{N-2}{2}+\tau}} \\
&\quad \leq \int_{\Omega \backslash B_{1}(x)}\Bigl(\sum_{j=0}^{\infty} U_{x_{j}, \mu}\Bigr)^{2^{*}-1} \sum_{j=0}^{\infty} \frac{\mu^{\frac{N-2}{2}}}{(1+\mu\left|y-x_{j}\right|)^{\frac{N-2}{2}+\tau}}.
\end{split}
\end{align}
We have
\begin{align}\begin{split}\label{eqs3.22}
\int_{\Omega \backslash B_{1}(x)} \frac{\mu^{\frac{N-2}{2}}}{\left(1+\mu|y-x|\right)^{\frac{N-2}{2}+\tau}} U_{x, \mu}^{2^{*}-1} \leq \int_{\mathbb{R}^{N} \backslash B_{\mu}(0)} \frac{1}{(1+|y|)^{N+2+\frac{N-2}{2}+\tau}} \leq \frac{C}{\mu^{N-\vartheta}}
\end{split}
\end{align}
and
\begin{align}\begin{split}\label{eqs3.23}
&\int_{\Omega \backslash B_{1}(x)} \sum_{j=1}^{\infty} \frac{\mu^{\frac{N-2}{2}}}{\left(1+\mu\left|y-x_{j}\right|\right)^{\frac{N-2}{2}+\tau}} U_{x, \mu}^{2^{*}-1}\\
&\quad \leq \frac{C \mu^{N}}{(\mu L)^{\frac{N-2}{2}+\tau}} \int_{\Omega \backslash B_{1}(x)} \frac{1}{(1+\mu|y-x|)^{N+2}} \cr
&\quad \leq \frac{C}{(\mu L)^{\frac{N-2}{2}+\tau}} \int_{\mathbb{R}^{N} \backslash B_{\mu}(0)} \frac{1}{(1+|y|)^{N+2}} \leq \frac{C}{\mu^{2}(\mu L)^{N-2-\vartheta}},
\end{split}
\end{align}
since $\tau=\frac{N-2}{2}-\vartheta.$

From
\begin{align*}
\sum_{j=1}^{\infty} U_{x_{j}, \mu} & \leq \frac{C \mu^{\frac{N-2}{2}}}{(1+\mu|y-x|)^{(2+2 \vartheta) /\left(2^{*}-1\right)}} \sum_{j=1}^{\infty} \frac{1}{\left(1+\mu\left|y-x_{j}\right|\right)^{N-2-(2+2 \vartheta) /\left(2^{*}-1\right)}} \\
& \leq \frac{C \mu^{\frac{N-2}{2}}}{(1+\mu|y-x|)^{(2+2 \vartheta) /\left(2^{*}-1\right)}} \frac{1}{(\mu L)^{N-2-(2+2 \vartheta) /\left(2^{*}-1\right)}},
\end{align*}
we obtain
\begin{align}\begin{split}\label{eqs3.24}
&\int_{\Omega \backslash B_{1}(x)} \frac{\mu^{\frac{N-2}{2}}}{(1+\mu|y-x|)^{\frac{N-2}{2}+\tau}}\Bigl(\sum_{j=1}^{\infty} U_{x_{j}, \mu}\Bigr)^{2^{*}-1} \\
&\quad \leq \int_{\mathbb{R}^{N} \backslash B_{\mu}(0)} \frac{1}{(1+|y|)^{2+2 \vartheta+\frac{N-2}{2}+\tau}} \frac{1}{(\mu L)^{\left(2^{*}-1\right)(N-2)-(2+2 \vartheta)}} \leq \frac{C}{(\mu L)^{N-2 \vartheta}}.
\end{split}
\end{align}
Similarly, we have
\begin{align}\begin{split}\label{eqs3.25}
&\int_{\Omega \backslash B_{1}(x)} \sum_{j=1}^{\infty} \frac{\mu^{\frac{N-2}{2}}}{(1+\mu|y-x_j|)^{\frac{N-2}{2}+\tau}}\Bigl(\sum_{j=1}^{\infty} U_{x_{j}, \mu}\Bigr)^{2^{*}-1} \\
&\quad \leq C \mu^{N} \int_{\Omega \backslash B_{1}(x)} \frac{1}{(1+\mu|y-x|)^{\frac{N-2}{2}}} \frac{1}{(\mu L)^{\tau}}\Bigl(\frac{1}{(1+\mu|y-x|)^{\frac{N-2}{2}+\frac{\vartheta}{2^{*}-1}}} \frac{1}{(\mu L)^{\frac{N-2}{2}-\frac{\vartheta}{2^{*}-1}}}\Bigr)^{2^{*}-1} \\
&\quad \leq C \int_{\mathbb{R}^{N} \backslash B_{\mu}(0)} \frac{1}{(1+|y|)^{N+\vartheta}} \frac{1}{(\mu L)^{\tau+\frac{N+2}{2}-\vartheta}} \leq \frac{C}{(\mu L)^{N-2 \vartheta}}.
\end{split}
\end{align}
Combining \eqref{eqs3.21}-\eqref{eqs3.25}, we find
\begin{equation}\label{eqs3.26}
\int_{\Omega \backslash B_{1}(x)}\Bigl(\sum_{j=0}^{\infty} U_{x_{j}, \mu}\Bigr)^{2^{*}-2}\left|\varphi_{1}\right| \sum_{j=0}^{\infty} \frac{\mu^{\frac{N-2}{2}}}{\left(1+\mu\left|y-x_{j}\right|\right)^{\frac{N-2}{2}+\tau}} \leq \frac{C}{\mu^{N-\vartheta}},
\end{equation}
which, together with \eqref{eqs3.20}, gives
\begin{equation}\label{eqs3.27}
\left|J_{1}\right| \leq \frac{C \mu^{\alpha(m)}\|(\omega_1,\omega_2)\|_{*}}{\mu^{\beta}}.
\end{equation}
Similar as $J_{3}$, we have
\begin{equation}\label{eqs3.28}
\left|J_{3}\right| \leq \frac{C \mu^{\alpha(m)}\|(\omega_1,\omega_2)\|_{*}}{\mu^{\beta}}.
\end{equation}
To estimate  $J_{2}$, using \eqref{eqs3.18} and \eqref{eqs3.19}, similar to \eqref{eqs3.21}, we find
\begin{align}\begin{split}\label{eqs3.29}
&\int_{\Omega}|K_1(y)-1| U_{x, \mu}^{2^{*}-2} \sum_{j=0}^{\infty} \frac{\mu^{\frac{N-2}{2}}}{\left(1+\mu\left|y-x_{j}\right|\right)^{\frac{N-2}{2}+\tau}} \sum_{j=0}^{\infty} U_{x_{j}, \mu} \\
&\quad \leq C \mu^{N} \int_{B_{1}(x)}|y|^{\beta} \frac{1}{(1+\mu|y-x|)^{4+\frac{N-2}{2}+\tau+N-2}} \\
&\quad \quad +C \int_{\Omega \backslash B_{1}(x)} U_{x, \mu}^{2^{*}-2} \sum_{j=0}^{\infty} \frac{\mu^{\frac{N-2}{2}}}{\left(1+\mu\left|y-x_{j}\right|\right)^{\frac{N-2}{2}+\tau}} \sum_{j=0}^{\infty} U_{x_{j}, \mu} \\
&\quad \leq \frac{C\left(1+|\mu x|^{\beta}\right)}{\mu^{\beta}}+\frac{C}{\mu^{N-\vartheta}} \leq \frac{C}{\mu^{\beta}} .
\end{split}
\end{align}
Thus, we get
\begin{equation}\label{eqs3.30}
\left|J_{2}\right| \leq \frac{C \mu^{\alpha(m)}\|(\omega_1,\omega_2)\|_{*}}{\mu^{\beta}}.
\end{equation}

To estimate  $J_{4}$ , noting that  $(\omega_1,\omega_2) \in \mathbf{E}$ , we see
\begin{align}\begin{split}\label{eqs3.31}
\left|J_{4}\right| & =\left| \int_{\Omega} (C_1\omega_1+ C_2\omega_2)U_{x, \mu}^{2^{*}-2}\partial_h PU_{x, \mu}\right| \\
& \leq C \mu^{\alpha(h)}\|(\omega_1,\omega_2)\|_{*} \int_{\Omega} U_{x, \mu}^{2^{*}-2} \sum_{j=0}^{\infty} \frac{\mu^{\frac{N-2}{2}}}{\left(1+\mu\left|y-x_{j}\right|\right)^{\frac{N-2}{2}+\tau}} \sum_{j=0}^{\infty} U_{x_{j}, \mu}.
\end{split}
\end{align}
Using \eqref{eqs3.19}, we obtain
\begin{align}\begin{split}\label{eqs3.32}
&\int_{B_{1}(x)} U_{x, \mu}^{2^{*}-2} \sum_{j=0}^{\infty} \frac{\mu^{\frac{N-2}{2}}}{\left(1+\mu\left|y-x_{j}\right|\right)^{\frac{N-2}{2}+\tau}} \sum_{j=0}^{\infty} U_{x_{j}, \mu} \\
&\quad \leq C \mu^{N} \int_{B_{1}(x)} \frac{1}{(1+\mu|y-x|)^{\frac{N-2}{2}+\tau+4}} \frac{1}{(\mu L)^{N-2}} \leq \frac{C}{(\mu L)^{N-2}}.
\end{split}
\end{align}

Similar to \eqref{eqs3.21}, we can prove
\begin{equation}\label{eqs3.33}
\int_{\Omega \backslash B_{1}(x)} U_{x, \mu}^{2^{*}-2} \sum_{j=0}^{\infty} \frac{\mu^{\frac{N-2}{2}}}{\left(1+\mu\left|y-x_{j}\right|\right)^{\frac{N-2}{2}+\tau}} \sum_{j=0}^{\infty} U_{x_{j}, \mu} \leq \frac{C}{\mu^{N-\vartheta}}.
\end{equation}
Thus we have proved
\begin{equation}\label{eqs3.34}
\left|J_{4}\right| \leq \frac{C \mu^{\alpha(h)}\|(\omega_1,\omega_2)\|_{*}}{\mu^{\beta}}.
\end{equation}
So \eqref{eqs3.13} follows from \eqref{eqs3.27}, \eqref{eqs3.28}, \eqref{eqs3.30} and \eqref{eqs3.34}.

Now we prove \eqref{eqs3.15}. We have
\begin{align}\begin{split}\label{eqs3.35}
\int_{\Omega} U_{x, \mu}^{2^{*}-2} Y_{m} \partial_{h}\left(PU_{x, \mu}\right) & =\int_{\Omega} U_{x, \mu}^{2^{*}-2} Y_{m} Y_{h}+O\Bigl(\int_{\Omega} U_{x, \mu}^{2^{*}-2}\left|Y_{m}\right|\left|\partial_{h} \varphi_{1}\right|\Bigr) \\
& =\int_{\mathbb{R}^{N}} U_{x, \mu}^{2^{*}-2} Y_{m}^{2}\Bigl(\delta_{m h}+o(1)\Bigr)=\mu^{2 \alpha(m)} a_{m}\Bigl(\delta_{m h}+o(1)\Bigr),
\end{split}
\end{align}
where  $a_{m}>0$  is a constant. So, we can solve \eqref{eqs3.11} and use \eqref{eqs3.13} to obtain \eqref{eqs3.15}.
\end{proof}

Next, we are going to prove T is bounded and compact. Thus, we first give some estimates for $T (\omega_{1},\omega_{2})$ defined in \eqref{eqs3.9}.

\begin{lemma}\label{lem3.4}
Suppose  $N \geq 5$, $1 \leq k<\frac{N-2}{2}$. Then for any  $\varepsilon \in(0,1)$, there are constant  $C_{\varepsilon}>0$, depending on  $\varepsilon$, and constant  $C>0$, independent of  $\varepsilon$, such that for any  $y \in \Omega$, it holds
\begin{align*}
&\int_{\Omega} G(z, y)\Bigl(\sum_{j=0}^{\infty} U_{x_{j}, \mu}\Bigr)^{\frac{4}{N-2}} \sigma(z) \sum_{j=0}^{\infty} \frac{\mu^{\frac{N-2}{2}}}{\left(1+\mu\left|z-x_{j}\right|\right)^{\frac{N-2}{2}+\tau}} dz \\
&\quad \leq C_{\varepsilon} \sigma(y) \sum_{j=0}^{\infty} \frac{\mu^{\frac{N-2}{2}}}{\left(1+\mu\left|y-x_{j}\right|\right)^{\frac{N-2}{2}+\tau+\theta}}+\frac{C \varepsilon^{2-\frac{4 \tau}{N-2}}}{\mu^{2-\frac{4 \tau}{N-2}}} \sigma(y) \sum_{j=0}^{\infty} \frac{\mu^{\frac{N-2}{2}}}{\left(1+\mu\left|y-x_{j}\right|\right)^{\frac{N-2}{2}+\tau}},
\end{align*}
where  $\theta>0$  is a fixed constant.
\end{lemma}
We prove Lemma \ref{lem3.4} in Appendix A because its proof contains a lot of computations.

\begin{lemma}\label{lem3.5}
$T$  is a compact bounded linear operator from  $\mathbf{X}$  to itself.
\end{lemma}

\begin{proof}
First, it follows from Lemma \ref{lem3.4} that for any  $(\omega_{1},\omega_{2}) \in \mathbf{X}$, it holds

\begin{align}\begin{split}\label{eqs3.36}
&\Bigl|(-\Delta)^{-1}\Bigl[K_1(y)(PU_{x, \mu})^{2^{*}-2} \omega_1-\frac{(PU_{x,\mu})^{\frac{4-N}{N-2}}(PV_{x,\mu})^{\frac{N}{N-2}}}{N+2}\omega_1-\frac{N(PU_{x,\mu})^{\frac{2}{N-2}}(PV_{x,\mu})^{\frac{2}{N-2}}}{2(N+2)}\omega_2\Bigr]\Bigr| \\
&\quad=\Bigl|\int_{\Omega} G(z, y) \Bigl[K_1(z)(PU_{x, \mu})^{2^{*}-2} \omega_1-\frac{1}{N+2}(PU_{x,\mu})^{\frac{4-N}{N-2}}(PV_{x,\mu})^{\frac{N}{N-2}}\omega_1\\
&\quad \quad -\frac{N}{2(N+2)}(PU_{x,\mu})^{\frac{2}{N-2}}(PV_{x,\mu})^{\frac{2}{N-2}}\omega_2\Bigr] dz\Bigr| \\
&\quad \leq C\|(\omega_1,\omega_2)\|_{*} \int_{\Omega} G(z, y)\Bigl(\sum_{j=0}^{\infty} U_{x_{j}, \mu}\Bigr)^{2^{*}-2} \sigma(z) \sum_{j=0}^{\infty} \frac{\mu^{\frac{N-2}{2}}}{\left(1+\mu\left|z-x_{j}\right|\right)^{\frac{N-2}{2}+\tau}} dz \\
&\quad \leq\|(\omega_1,\omega_2)\|_{*} \sigma(y)\Bigl[C_{\varepsilon} \sum_{j=0}^{\infty} \frac{\mu^{\frac{N-2}{2}}}{\left(1+\mu\left|y-x_{j}\right|\right)^{\frac{N-2}{2}+\tau+\theta}}+\frac{C \varepsilon^{2-\frac{4 \tau}{N-2}}}{\mu^{2-\frac{4 \tau}{N-2}}} \sum_{j=0}^{\infty} \frac{\mu^{\frac{N-2}{2}}}{\left(1+\mu\left|y-x_{j}\right|\right)^{\frac{N-2}{2}+\tau}}\Bigr].
\end{split}
\end{align}
Similarly, we have
\begin{align}\begin{split}\label{eqs3.36(1)}
&\Bigl|(-\Delta)^{-1}\Bigl[K_2(y)(PV_{x, \mu})^{2^{*}-2} \omega_2-\frac{(PV_{x,\mu})^{\frac{4-N}{N-2}}(PU_{x,\mu})^{\frac{N}{N-2}}}{N+2}\omega_2-\frac{N(PV_{x,\mu})^{\frac{2}{N-2}}(PU_{x,\mu})^{\frac{2}{N-2}}}{2(N+2)}\omega_1\Bigr]\Bigr| \\
&\quad \leq\|(\omega_1,\omega_2)\|_{*} \sigma(y)\Bigl[C_{\varepsilon} \sum_{j=0}^{\infty} \frac{\mu^{\frac{N-2}{2}}}{\left(1+\mu\left|y-x_{j}\right|\right)^{\frac{N-2}{2}+\tau+\theta}}+\frac{C \varepsilon^{2-\frac{4 \tau}{N-2}}}{\mu^{2-\frac{4 \tau}{N-2}}} \sum_{j=0}^{\infty} \frac{\mu^{\frac{N-2}{2}}}{\left(1+\mu\left|y-x_{j}\right|\right)^{\frac{N-2}{2}+\tau}}\Bigr].
\end{split}
\end{align}

Using Lemma 3.3, we obtain
\begin{align}\begin{split}\label{eqs3.37}
\Bigl|c_{h} \int_{\Omega} G(z, y) U_{x, \mu}^{2^{*}-2} Y_{h} dz\Bigr|
&\leq \frac{C}{\mu^{\alpha(h)+\beta}}\|(\omega_1,\omega_2)\|_{*} \mu^{\alpha(h)} \int_{\Omega} G(z, y) U_{x, \mu}^{2^{*}-1} \\
&\leq \frac{C}{\mu^{\beta}}\|(\omega_1,\omega_2)\|_{*} \sum_{j=0}^{\infty} \frac{\mu^{\frac{N-2}{2}}}{\left(1+\mu\left|y-x_{j}\right|\right)^{N-2}}
\end{split}
\end{align}
and similarly, we get
\begin{align}\label{eqs3.37(1)}
\Bigl|c_{h} \int_{\Omega} G(z, y) V_{x, \mu}^{2^{*}-2} Z_{h} dz\Bigr|
\leq \frac{C}{\mu^{\beta}}\|(\omega_1,\omega_2)\|_{*} \sum_{j=0}^{\infty} \frac{\mu^{\frac{N-2}{2}}}{\left(1+\mu\left|y-x_{j}\right|\right)^{N-2}}.
\end{align}
Combining \eqref{eqs3.36}, \eqref{eqs3.36(1)}, \eqref{eqs3.37} and \eqref{eqs3.37(1)}, we find  $\|T (\omega_{1},\omega_{2})\|_{*} \leq C\|(\omega_{1},\omega_{2})\|_{*}$. This shows that  $T$  is a bounded linear operator from  $\mathbf{X}$  to itself.

Now we prove that  $T$  is compact. Suppose that  $(\omega_{1,n},\,\omega_{2,n})$  is a bounded sequence in  $\mathbf{X}$. We denote $T (\omega_{1},\omega_{2})=(T_{1} (\omega_{1},\omega_{2}),T_{2} (\omega_{1},\omega_{2}))$. For any  $\varepsilon>0$, from \eqref{eqs3.36}-\eqref{eqs3.37(1)}, we see that there exists  $R>0$  large enough, such that
\begin{equation}\label{eqs3.38}
\Bigl(\sigma(y) \sum_{j=0}^{\infty} \frac{\mu^{\frac{N-2}{2}}}{\left(1+\mu\left|y-x_{j}\right|\right)^{\frac{N-2}{2}+\tau}}\Bigr)^{-1}\left|T_{1} (\omega_{1,n},\omega_{2,n})\right|<\varepsilon
\end{equation}
and
\begin{equation}\label{eqs3.38(1)}
\Bigl(\sigma(y) \sum_{j=0}^{\infty} \frac{\mu^{\frac{N-2}{2}}}{\left(1+\mu\left|y-x_{j}\right|\right)^{\frac{N-2}{2}+\tau}}\Bigr)^{-1}\left|T_{2} (\omega_{1,n},\omega_{2,n})\right|<\varepsilon
\end{equation}
for any $ y \in \Omega \backslash B_{R}(0)$.

On the other hand, from
\begin{align*}\begin{split}
&(-\Delta T_{1} (\omega_{1,n},\omega_{2,n}), -\Delta T_{2} (\omega_{1,n},\omega_{2,n}))\\
&\quad =(2^{*}-1)\Bigl(K_1(y)(PU_{x, \mu})^{2^{*}-2} \omega_{1}-\frac{1}{N+2}(PU_{x,\mu})^{\frac{4-N}{N-2}}(PV_{x,\mu})^{\frac{N}{N-2}}\omega_{1}\\
&\quad \quad -\frac{N}{2(N+2)}(PU_{x,\mu})^{\frac{2}{N-2}}(PV_{x,\mu})^{\frac{2}{N-2}}\omega_{2}+\sum_{h=1}^{N+1} c_{h}U_{x, \mu}^{2^{*}-2} Y_{h},\\
&\quad \quad K_2(y)(PV_{x, \mu})^{2^{*}-2} \omega_{2}-\frac{1}{N+2}(PV_{x,\mu})^{\frac{4-N}{N-2}}(PU_{x,\mu})^{\frac{N}{N-2}}\omega_{2}\\
&\quad \quad -\frac{N}{2(N+2)}(PV_{x,\mu})^{\frac{2}{N-2}}(PU_{x,\mu})^{\frac{2}{N-2}}\omega_{1}+\sum_{h=1}^{N+1} c_{h}V_{x, \mu}^{2^{*}-2} Z_{h}\Bigr),
\end{split}\end{align*}
we know from the  $L^{p}$  estimate that  $T_{1} (\omega_{1,n},\omega_{2,n})$ and $T_{2} (\omega_{1,n},\omega_{2,n})$  are bounded in  $C_{loc}^{1}(\Omega)$. As a resut,  $T (\omega_{1,n},\omega_{2,n})$  has a subsequence, still denoted by  $T (\omega_{1,n},\omega_{2,n})$ , which converges to  $(u_{0},v_{0})$  uniformly in  $B_{R}(0) \cap \Omega$  for any  $R>0$. Moreover, from  $\left\|T (\omega_{1,n},\omega_{2,n})\right\|_{*} \leq C<+\infty$, we find  $\left\|(u_{0},v_{0})\right\|_{*} \leq C$  and this gives  $(u_{0},v_{0}) \in \mathbf{X}$. Let $n \rightarrow+\infty$ in \eqref{eqs3.38} and \eqref{eqs3.38(1)}, we obtain
$$
\Bigl(\sigma(y) \sum_{j=0}^{\infty} \frac{\mu^{\frac{N-2}{2}}}{\left(1+\mu\left|y-x_{j}\right|\right)^{\frac{N-2}{2}+\tau}}\Bigr)^{-1}\left|u_{0}\right|<\varepsilon
$$
and
$$
\Bigl(\sigma(y) \sum_{j=0}^{\infty} \frac{\mu^{\frac{N-2}{2}}}{\left(1+\mu\left|y-x_{j}\right|\right)^{\frac{N-2}{2}+\tau}}\Bigr)^{-1}\left|v_{0}\right|<\varepsilon
$$
for any $y \in \Omega \backslash B_{R}(0)$ if $R>0$ is large. This, together with \eqref{eqs3.38}, \eqref{eqs3.38(1)} and  $T (\omega_{1,n},\omega_{2,n}) \rightarrow (u_{0},v_{0})$  uniformly in  $B_{R}(0) \cap \Omega$, implies that  $\left\|T (\omega_{1,n},\omega_{2,n})-(u_{0},v_{0})\right\|_{*} \rightarrow 0$. So  $T$  is compact.
\end{proof}

Now we prove that  $I-T$  is injective in  $\mathbf{E}$. For this purpose, we consider  $(\omega_1,\omega_2)-T (\omega_1,\omega_2)=(-\Delta)^{-1} (f_1,f_2)$, where  $(\omega_1,\omega_2) \in \mathbf{E}$  and  $(f_1,f_2) \in \mathbf{F}$. Equivalently, we consider the following linear problem:
\begin{align}\begin{split}\label{eqs3.39}
&-\Delta \omega_1-\left(2^{*}-1\right)\Bigl(K_1(y)(PU_{x, \mu})^{2^{*}-2} \omega_{1}-\frac{1}{N+2}(PU_{x,\mu})^{\frac{4-N}{N-2}}(PV_{x,\mu})^{\frac{N}{N-2}}\omega_{1}\\
&\quad -\frac{N}{2(N+2)}(PU_{x,\mu})^{\frac{2}{N-2}}(PV_{x,\mu})^{\frac{2}{N-2}}\omega_2\Bigr)=f_1+\sum_{h=1}^{N+1} c_{h} U_{x, \mu}^{2^{*}-2} Y_{h}\\
\end{split}
\end{align}
and
\begin{align}\begin{split}\label{eqs3.39(1)}
&-\Delta \omega_2-\left(2^{*}-1\right)\Bigl(K_2(y)(PV_{x, \mu})^{2^{*}-2} \omega_{2}-\frac{1}{N+2}(PV_{x,\mu})^{\frac{4-N}{N-2}}(PU_{x,\mu})^{\frac{N}{N-2}}\omega_{2}\\
&\quad -\frac{N}{2(N+2)}(PV_{x,\mu})^{\frac{2}{N-2}}(PU_{x,\mu})^{\frac{2}{N-2}}\omega_1\Bigr)=f_2+\sum_{h=1}^{N+1} c_{h} V_{x, \mu}^{2^{*}-2} Z_{h}\\
\end{split}
\end{align}
for some constants  $c_{h}$.

\begin{lemma}\label{lem3.6}
Suppose that  $(\omega_1,\omega_2) \in \mathbf{E}$  solves \eqref{eqs3.39} and \eqref{eqs3.39(1)} for L sufficiently large. Then  $\|(\omega_1,\omega_2)\|_{*} \leq C\|(f_1,f_2)\|_{* *}$, for some constant  $C>0$, independent of  $(x, \mu)$. In particular,  $I-T$  is injective in  $\mathbf{E}$.
\end{lemma}
\begin{proof}
We write
\begin{align}\begin{split}\label{eqs3.40}
\omega_1(y)=&(2^{*}-1) \int_{\Omega} G(z, y) \Bigl[K_1(z)(PU_{x, \mu})^{2^{*}-2} \omega_1
-\frac{1}{N+2}(PU_{x,\mu})^{\frac{4-N}{N-2}}(PV_{x,\mu})^{\frac{N}{N-2}}\omega_{1}\\
&-\frac{N}{2(N+2)}(PU_{x,\mu})^{\frac{2}{N-2}}(PV_{x,\mu})^{\frac{2}{N-2}}\omega_2 dz\Bigr]+\int_{\Omega} G(z, y)\Bigl(f_1(z)+\sum_{h=1}^{N+1} c_{h} U_{x, \mu}^{2^{*}-2} Y_{h}\Bigr)dz
\end{split}
\end{align}
and
\begin{align}\begin{split}\label{eqs3.41}
\omega_2(y)= & \left(2^{*}-1\right) \int_{\Omega} G(z, y)\Bigl[K_2(z)\left(PV_{x, \mu}\right)^{2^{*}-2} \omega_2
-\frac{1}{N+2}(PV_{x,\mu})^{\frac{4-N}{N-2}}(PU_{x,\mu})^{\frac{N}{N-2}}\omega_2 \cr
&-\frac{N}{2(N+2)}(PV_{x,\mu})^{\frac{2}{N-2}}(PU_{x,\mu})^{\frac{2}{N-2}}\omega_1 dz\Bigr]
+\int_{\Omega} G(z, y)\Bigl(f_2(z)+\sum_{h=1}^{N+1} c_{h} V_{x, \mu}^{2^{*}-2} Z_{h}\Bigr)dz.
\end{split}
\end{align}
Form the proof of \eqref{eqs3.36}, we have
\begin{align}\begin{split}\label{eqs3.41(1)}
&\Bigl|\int_{\Omega} G(z, y) \Bigl[K_1(z)(PU_{x, \mu})^{2^{*}-2} \omega_1-\frac{1}{N+2}(PU_{x,\mu})^{\frac{4-N}{N-2}}(PV_{x,\mu})^{\frac{N}{N-2}}\omega_1\\
&\quad \quad -\frac{N}{2(N+2)}(PU_{x,\mu})^{\frac{2}{N-2}}(PV_{x,\mu})^{\frac{2}{N-2}}\omega_2\Bigr] dz\Bigr| \\
&\quad \leq\|(\omega_1,\omega_2)\|_{*} \sigma(y)\Bigl[C_{\varepsilon} \sum_{j=0}^{\infty} \frac{\mu^{\frac{N-2}{2}}}{\left(1+\mu\left|y-x_{j}\right|\right)^{\frac{N-2}{2}+\tau+\theta}}+\frac{C \varepsilon^{2-\frac{4 \tau}{N-2}}}{\mu^{2-\frac{4 \tau}{N-2}}} \sum_{j=0}^{\infty} \frac{\mu^{\frac{N-2}{2}}}{\left(1+\mu\left|y-x_{j}\right|\right)^{\frac{N-2}{2}+\tau}}\Bigr].
\end{split}
\end{align}
On the other hand, it is easy to see that
\begin{align}\begin{split}\label{eqs3.42}
\Bigl|\int_{\Omega} G(z, y) f_1(z) dz\Bigr|
&\leq C\|f_1\|_{* *} \int_{\Omega} G(z, y) \sigma(z) \sum_{j=0}^{\infty} \frac{\mu^{\frac{N+2}{2}}}{\left(1+\mu\left|z-x_{j}\right|\right)^{\frac{N+2}{2}+\tau}} d z \\
&=C\|f_1\|_{* *} \int_{\mathbb{R}^{N}} \frac{1}{|y-z|^{N-2}} \sigma(z) \sum_{j=0}^{\infty} \frac{\mu^{\frac{N+2}{2}}}{\left(1+\mu\left|z-x_{j}\right|\right)^{\frac{N+2}{2}+\tau}} d z \\
&\leq C\|f_1\|_{* *} \sigma(y) \sum_{j=0}^{\infty} \frac{\mu^{\frac{N-2}{2}}}{\left(1+\mu\left|y-x_{j}\right|\right)^{\frac{N-2}{2}+\tau}}
\end{split}
\end{align}
and
\begin{align}\begin{split}\label{eqs3.43}
\Bigl|\int_{\Omega} G(z, y) U_{x, \mu}^{2^{*}-2} Y_{h} dz\Bigr|
&\leq C \mu^{\alpha(h)} \int_{\Omega} G(z, y) U_{x, \mu}^{2^{*}-1} dz \cr
&\leq C \mu^{\alpha(h)} \sum_{j=0}^{\infty} \frac{\mu^{\frac{N-2}{2}}}{\left(1+\mu\left|y-x_{j}\right|\right)^{\frac{N-2}{2}+\tau}}.
\end{split}
\end{align}
Similarly, we have
\begin{equation}\label{eqs3.42(1)}
\Bigl|\int_{\Omega} G(z, y) f_2(z) dz\Bigr| \leq C\|f_2\|_{* *} \sigma(y) \sum_{j=0}^{\infty} \frac{\mu^{\frac{N-2}{2}}}{\left(1+\mu\left|y-x_{j}\right|\right)^{\frac{N-2}{2}+\tau}}
\end{equation}
and
\begin{equation}\label{eqs3.43(1)}
\Bigl|\int_{\Omega} G(z, y) V_{x, \mu}^{2^{*}-2} Z_{h} dz\Bigr|
\leq C \mu^{\alpha(h)} \sum_{j=0}^{\infty} \frac{\mu^{\frac{N-2}{2}}}{\left(1+\mu\left|y-x_{j}\right|\right)^{\frac{N-2}{2}+\tau}}.
\end{equation}

It remains to estimate  $c_{h}$. We use \eqref{eqs3.39} and \eqref{eqs3.39(1)} to find
\begin{align}\begin{split}\label{eqs3.44}
&-\sum_{h=1}^{N+1} c_{h} \int_{\Omega} U_{x, \mu}^{2^{*}-2} Y_{h} \partial_{m} P U_{x, \mu}\\
&\quad =\int_{\Omega}\left(\left(2^{*}-1\right) K_1(z)\left(P U_{x, \mu}\right)^{2^{*}-2} \omega_1-\frac{1}{N-2}(PU_{x_{L},\mu_{L}})^{\frac{4-N}{N-2}}(PV_{x,\mu})^{\frac{N}{N-2}}\omega_{1}\right.\\
&\quad \quad \left.-\frac{N}{2(N-2)}(PU_{x_{L},\mu_{L}})^{\frac{2}{N-2}}(PV_{x_{L},\mu_{L}})^{\frac{2}{N-2}}\omega_2\right)\partial_{m} PU_{x, \mu}+\int_{\Omega} f_1\partial_{m} PU_{x, \mu}
\end{split}
\end{align}
and
\begin{align}\begin{split}\label{eqs3.44(1)}
&-\sum_{h=1}^{N+1} c_{h} \int_{\Omega} V_{x, \mu}^{2^{*}-2} Z_{h} \partial_{m} P V_{x, \mu}\\
&\quad =\int_{\Omega}\left(\left(2^{*}-1\right) K_2(z)\left(P V_{x, \mu}\right)^{2^{*}-2} \omega_2-\frac{1}{N-2}(PV_{x_{L},\mu_{L}})^{\frac{4-N}{N-2}}(PU_{x,\mu})^{\frac{N}{N-2}}\omega_{2}\right.\\
&\quad \quad \left.-\frac{N}{2(N-2)}(PV_{x_{L},\mu_{L}})^{\frac{2}{N-2}}(PU_{x_{L},\mu_{L}})^{\frac{2}{N-2}}\omega_1\right)\partial_{m} PV_{x, \mu}+\int_{\Omega} f_2\partial_{m} PV_{x, \mu}.
\end{split}
\end{align}
We have
\begin{equation}\label{eqs3.45}
\Bigl|\int_{\Omega} f_1 \partial_{m} PU_{x, \mu}\Bigr| \leq C\|f_1\|_{* *} \mu^{\alpha(m)} \int_{\Omega} \sigma(y) \sum_{j=0}^{\infty} \frac{\mu^{\frac{N+2}{2}}}{\left(1+\mu\left|y-x_{j}\right|\right)^{\frac{N+2}{2}+\tau}} \sum_{j=0}^{\infty} U_{x_{j}, \mu}.
\end{equation}
Similar to \eqref{eqs3.18} and \eqref{eqs3.19}, we can prove
\begin{align}\begin{split}\label{eqs3.46}
&\int_{B_{1}(x)} \sigma(y) \sum_{j=0}^{\infty} \frac{\mu^{\frac{N+2}{2}}}{\left(1+\mu\left|y-x_{j}\right|\right)^{\frac{N+2}{2}+\tau}} \sum_{j=0}^{\infty} U_{x_{j}, \mu} \\
&\quad \leq C \mu^{N} \int_{B_{1}(x)} \frac{\sigma(y)}{(1+\mu|y-x|)^{N-2+\frac{N+2}{2}+\tau}}\\
&\quad \leq C \mu^{N} \int_{B_{1}(x)}\Bigl(\frac{1+\mu|y-x|}{\mu}\Bigr)^{\tau-1} \frac{1}{(1+\mu|y-x|)^{N-2+\frac{N+2}{2}+\tau}} \leq \frac{C}{\mu^{\tau-1}}.
\end{split}
\end{align}
On the other hand, we have
\begin{align}\begin{split}\label{eqs3.47}
&\int_{\Omega \backslash B_{1}(x)} \sigma(y) \frac{\mu^{\frac{N+2}{2}}}{(1+\mu|y-x|)^{\frac{N+2}{2}+\tau}} U_{x, \mu} \\
&\quad \leq \int_{\mathbb{R}^{N} \backslash B_{\mu}(0)} \frac{1}{(1+|y|)^{N-2+\frac{N+2}{2}+\tau}} \leq \frac{C}{\mu^{N-2-\vartheta}}.
\end{split}
\end{align}
For  $y \in \Omega$, it holds
\begin{align*}
\sum_{j=1}^{\infty} U_{x_{j}, \mu} & \leq \frac{C \mu^{\frac{N-2}{2}}}{(1+\mu|y-x|)^{2 \vartheta}} \sum_{j=1}^{\infty} \frac{1}{\left(1+\mu\left|y-x_{j}\right|\right)^{N-2-2 \vartheta}} \\
& \leq \frac{C \mu^{\frac{N-2}{2}}}{(1+\mu|y-x|)^{2 \vartheta}} \frac{1}{(\mu L)^{N-2-2 \vartheta}}
\end{align*}
and
\begin{align*}
\sum_{j=1}^{\infty} \frac{\mu^{\frac{N+2}{2}}}{\left(1+\mu\left|y-x_{j}\right|\right)^{\frac{N+2}{2}+\tau}} \leq \frac{C \mu^{\frac{N+2}{2}}}{(1+\mu|y-x|)^{2+\vartheta}} \frac{1}{(\mu L)^{\frac{N+2}{2}+\tau-2-\vartheta}}.
\end{align*}
Thus, we obtain
\begin{align}\begin{split}\label{eqs3.48}
&\int_{\Omega \backslash B_{1}(x)} \sigma(y) \frac{\mu^{\frac{N+2}{2}}}{(1+\mu|y-x|)^{\frac{N+2}{2}+\tau}} \sum_{j=1}^{\infty} U_{x_{j}, \mu} \\
&\quad \leq \int_{\mathbb{R}^{N} \backslash B_{\mu}(0)} \frac{1}{(1+|y|)^{2 \vartheta+\frac{N+2}{2}+\tau}} \frac{1}{(\mu L)^{N-2-2 \vartheta}} \leq \frac{C}{(\mu L)^{N-2-2 \vartheta}}
\end{split}
\end{align}
and
\begin{align}\begin{split}\label{eqs3.49}
&\int_{\Omega \backslash B_{1}(x)} \sigma(y) \sum_{j=1}^{\infty} \frac{\mu^{\frac{N+2}{2}}}{\left(1+\mu\left|y-x_{j}\right|\right)^{\frac{N+2}{2}+\tau}} U_{x, \mu} \\
&\quad \leq \int_{\mathbb{R}^{N} \backslash B_{\mu}(0)} \frac{1}{(1+|y|)^{N+\vartheta}} \frac{1}{(\mu L)^{\frac{N+2}{2}+\tau-2-\vartheta}} \leq \frac{C}{(\mu L)^{\frac{N+2}{2}+\tau-2-\vartheta}}.
\end{split}
\end{align}
Moreover, there holds
\begin{align}\begin{split}\label{eqs3.50}
&\int_{\Omega \backslash B_{1}(x)} \sigma(y) \sum_{j=1}^{\infty} \frac{\mu^{\frac{N+2}{2}}}{(1+\mu|y-x_j|)^{\frac{N+2}{2}+\tau}} \sum_{j=1}^{\infty} U_{x_{j}, \mu}\\
&\quad \leq C \mu^{N} \int_{\Omega \backslash B_{1}(x)} \frac{1}{(1+\mu|y-x|)^{\frac{N+2}{2}}} \frac{1}{(\mu L)^{\tau}} \frac{1}{(1+\mu|y-x|)^{\frac{N-2}{2}+\vartheta}} \frac{1}{(\mu L)^{\frac{N-2}{2}-\vartheta}} \\
&\quad \leq C \int_{\mathbb{R}^{N} \backslash B_{\mu}(0)} \frac{1}{(1+|y|)^{N+\vartheta}} \frac{1}{(\mu L)^{N-2-2 \vartheta}} \leq \frac{C}{(\mu L)^{N-2-2 \vartheta}}.
\end{split}
\end{align}
Combining \eqref{eqs3.47}-\eqref{eqs3.50}, we obtain
\begin{equation}\label{eqs3.51}
\int_{\Omega \backslash B_{1}(x)} \sigma(y) \sum_{j=0}^{\infty} \frac{\mu^{\frac{N+2}{2}}}{\left(1+\mu\left|y-x_{j}\right|\right)^{\frac{N+2}{2}+\tau}} \sum_{j=0}^{\infty} U_{x_{j}, \mu} \leq \frac{C}{\mu^{N-2-\vartheta}}.
\end{equation}
It follows from \eqref{eqs3.45}, \eqref{eqs3.46} and \eqref{eqs3.51} that
\begin{equation}\label{eqs3.52}
\Bigl|\int_{\Omega} f_1 \partial_{m} PU_{x, \mu}\Bigr| \leq C\frac{\|f_1\|_{* *} \mu^{\alpha(m)}}{\mu^{\tau-1}}.
\end{equation}
Similarly, we also have
\begin{equation}\label{eqs3.52(1)}
\Bigl|\int_{\Omega} f_2 \partial_{m} PV_{x, \mu}\Bigr| \leq C\frac{\|f_2\|_{* *} \mu^{\alpha(m)}}{\mu^{\tau-1}}.
\end{equation}
Using \eqref{eqs3.13}, \eqref{eqs3.35}, \eqref{eqs3.44}, \eqref{eqs3.44(1)}, \eqref{eqs3.52} and \eqref{eqs3.52(1)}, we see
\begin{equation}\label{eqs3.53}
\left|c_{h}\right| \leq C\left(\|(f_1,f_2)\|_{* *}+o(1)\|(\omega_1,\omega_2)\|_{*}\right) \frac{\mu^{-\alpha(h)}}{\mu^{\tau-1}},
\end{equation}
which gives
\begin{align}\begin{split}\label{eqs3.54}
&\Bigl|\int_{\Omega} G(z, y) \sum_{h=1}^{N+1} c_{h} U_{x, \mu}^{2^{*}-2} Y_{h} d z\Bigr|\\
&\quad \leq C\left(\|(f_1,f_2)\|_{* *}+o(1)\|(\omega_1,\omega_2)\|_{*}\right) \frac{1}{\mu^{\tau-1}} \sum_{j=0}^{\infty} \frac{\mu^{\frac{N-2}{2}}}{\left(1+\mu\left|y-x_{j}\right|\right)^{\frac{N-2}{2}+\tau}} \\
&\quad \leq C\left(\|(f_1,f_2)\|_{* *}+o(1)\|(\omega_1,\omega_2)\|_{*}\right) \sigma(y) \sum_{j=0}^{\infty} \frac{\mu^{\frac{N-2}{2}}}{\left(1+\mu\left|y-x_{j}\right|\right)^{\frac{N-2}{2}+\tau}}
\end{split}
\end{align}
and
\begin{align}\begin{split}\label{eqs3.54(1)}
&\Bigl|\int_{\Omega} G(z, y) \sum_{h=1}^{N+1} c_{h} V_{x, \mu}^{2^{*}-2} Z_{h} d z\Bigr|\\
&\quad \leq C\left(\|(f_1,f_2)\|_{* *}+o(1)\|(\omega_1,\omega_2)\|_{*}\right) \sigma(y) \sum_{j=0}^{\infty} \frac{\mu^{\frac{N-2}{2}}}{\left(1+\mu\left|y-x_{j}\right|\right)^{\frac{N-2}{2}+\tau}}.
\end{split}
\end{align}
Combining \eqref{eqs3.36(1)}, \eqref{eqs3.40}-\eqref{eqs3.42}, \eqref{eqs3.42(1)}, \eqref{eqs3.54} and \eqref{eqs3.54(1)}, we are led to
\begin{align}\begin{split}\label{eqs3.55}
&|\omega_1(y)|\Bigl(\sigma(y) \sum_{j=0}^{\infty} \frac{\mu^{\frac{N-2}{2}}}{\left(1+\mu\left|y-x_{j}\right|\right)^{\frac{N-2}{2}+\tau}}\Bigr)^{-1} \\
&\quad \leq C\Bigl(\|(f_1,f_2)\|_{* *}+o(1)\|(\omega_1,\omega_2)\|_{*}+\frac{\sum_{j=0}^{\infty} \frac{1}{\left(1+\mu\left|y-x_{j}\right|\right)^{\frac{N-2}{2}+\tau+\bar{\theta}}}}{\sum_{j=0}^{\infty} \frac{1}{\left(1+\mu \mid y-x_{j}\right)^{\frac{N-2}{2}+\tau}}}\|(\omega_1,\omega_2)\|_{*}\Bigr)
\end{split}
\end{align}
and
\begin{align}\begin{split}\label{eqs3.55(1)}
&|\omega_2(y)|\Bigl(\sigma(y) \sum_{j=0}^{\infty} \frac{\mu^{\frac{N-2}{2}}}{\left(1+\mu\left|y-x_{j}\right|\right)^{\frac{N-2}{2}+\tau}}\Bigr)^{-1} \\
&\quad \leq C\Bigl(\|(f_1,f_2)\|_{* *}+o(1)\|(\omega_1,\omega_2)\|_{*}+\frac{\sum_{j=0}^{\infty} \frac{1}{\left(1+\mu\left|y-x_{j}\right|\right)^{\frac{N-2}{2}+\tau+\bar{\theta}}}}{\sum_{j=0}^{\infty} \frac{1}{\left(1+\mu \mid y-x_{j}\right)^{\frac{N-2}{2}+\tau}}}\|(\omega_1,\omega_2)\|_{*}\Bigr).
\end{split}
\end{align}
Thus we have proved $\|(\omega_1,\omega_2)\|_{*} \leq C\|(f_1,f_2)\|_{* *}$.

Now we prove $I-T$ is injective in $\mathbf{E}$. Suppose that there are  $L_{n} \rightarrow+\infty$, $x_{n} \rightarrow 0$, $\mu_{n} \rightarrow+\infty$, $f_{n}=(f_{1,n},f_{2,n}) \in \mathbf{F}$  and  $\omega_{n}=(\omega_{1,n},\omega_{2,n}) \in \mathbf{E}$, such that  $\omega_{n}$  solves \eqref{eqs3.39} and \eqref{eqs3.39(1)} with  $\left\|\omega_{n}\right\|_{*}=1$ and  $\left\|f_{n}\right\|_{* *} \rightarrow 0$  as  $n \rightarrow+\infty$. Denote  $x_{n, j}=x_{n}-L_{n} P_{j}$.

From \eqref{eqs3.55} and \eqref{eqs3.55(1)}, we see that the maximum of
$$\left|\omega_{1,n}(y)\right|\Bigl(\sigma(y) \sum_{j=0}^{\infty} \frac{\mu_{n}^{\frac{N-2}{2}}}{(1+\mu_{n}\left|y-x_{n, j}\right|)^{\frac{N-2}{2}+\tau}}\Bigr)^{-1}$$
and
$$\left|\omega_{2,n}(y)\right|\Bigl(\sigma(y) \sum_{j=0}^{\infty} \frac{\mu_{n}^{\frac{N-2}{2}}}{(1+\mu_{n}\left|y-x_{n, j}\right|)^{\frac{N-2}{2}+\tau}}\Bigr)^{-1}$$
in $\Omega$  can only be achieved in  $B_{R \mu_{n}^{-1}}(0)$  for some large  $R>0$. On the other hand,  we see that $(\mu_{n}^{-\frac{N-2}{2}} \omega_{1,n}\left(\mu_{n}^{-1} y+x_{n}\right), \mu_{n}^{-\frac{N-2}{2}} \omega_{2,n}\left(\mu_{n}^{-1} y+x_{n}\right))$  converges to  $(\Psi,\Phi)$  uniformly in  $B_{R}(0)$  for any  $R>0$, and  $(\Psi,\Phi)$  satisfies
\begin{eqnarray*}\begin{cases}
-\Delta \Psi=\left(2^{*}-1\right) U_{0,1}^{2^{*}-2} \Psi+\frac{1}{2}(\frac{2^*}{2}-1)U_{0,1}^{\frac{2^*}{2}-2}V_{0,1}^{\frac{2^*}{2}}\Psi+\frac{1}{2}\frac{2^*}{2}U_{0,1}^{\frac{2^*}{2}-1}V_{0,1}^{\frac{2^*}{2}-1}\Phi, \quad \text { in } \mathbb{R}^{N},\\
-\Delta \Phi=\left(2^{*}-1\right) V_{0,1}^{2^{*}-2} \Phi+\frac{1}{2}(\frac{2^*}{2}-1)V_{0,1}^{\frac{2^*}{2}-2}U_{0,1}^{\frac{2^*}{2}}\Phi+\frac{1}{2}\frac{2^*}{2}V_{0,1}^{\frac{2^*}{2}-1}U_{0,1}^{\frac{2^*}{2}-1}\Psi, \quad \text { in } \mathbb{R}^{N}.
\end{cases}
\end{eqnarray*}
Since  $\omega_{n} \in \mathbf{E}$, we conclude  $(\Psi,\Phi)=0$. This is a contradiction to  $\left\|(\omega_{1n},\omega_{2n})\right\|_{*}=1$.
\end{proof}

\begin{proposition}\label{pro3.6}
$I-T$  is a bijective bounded linear operator from $\mathbf{E}$ to itself.
\end{proposition}
\begin{proof}
 Since  $T$  is a compact bounded linear operator and  $I-T$  is injective, the Fredholm theorem implies that  $I-T$  is surjective.
\end{proof}

\section{existence of periodic solutions}
In this section, we will prove the main theorem. We will first solve the following problem
\begin{equation}\label{eqs4.1}
(\omega_{1},\omega_{2})-T (\omega_{1},\omega_{2})=(-\Delta)^{-1} (\mathbf{P}\ell+\mathbf{P} N(\omega_{1},\omega_{2})).
\end{equation}
For this purpose, we first give the estimates of $N(\omega_{1},\omega_{2})$ and $\ell$.

\begin{lemma}\label{lem4.2}
If  $N \geq 5$, then
$$\|N(\omega_1,\omega_2)\|_{* *} \leq C\|(\omega_1,\omega_2)\|_{*}^{(1+\delta)}.$$
\end{lemma}
\begin{proof}
Note that
\begin{equation*}\begin{split}
N(\omega_{1},\omega_{2})&=(N_1(\omega_{1},\omega_{2}),N_2(\omega_{1},\omega_{2}))\\
&:=(N_{1,1}(\omega_{1},\omega_{2})+N_{1,2}(\omega_{1},\omega_{2}),N_{2,1}(\omega_{1},\omega_{2})+N_{2,2}(\omega_{1},\omega_{2})),
\end{split}
\end{equation*}
where
\begin{align*}
N_{1,1}(\omega_{1},\omega_{2})=K_1(y)\Bigl((PU_{x, \mu}+\omega_{1})_{+}^{2^{*}-1}-(PU_{x, \mu})^{2^{*}-1}-(2^{*}-1)(PU_{x, \mu})^{2^{*}-2}\omega_{1}\Bigr),
\end{align*}
\begin{equation*}\begin{split}
N_{1,2}(\omega_{1},\omega_{2})&=\frac{1}{2}\Bigl((PU_{x,\mu}+\omega_{1})^{\frac{2}{N-2}}(PV_{x,\mu}+\omega_{2})^{\frac{N}{N-2}}-(PU_{x,\mu})^{\frac{2}{N-2}}(PV_{x,\mu})^{\frac{N}{N-2}}\Bigr)\\
&\quad -\frac{1}{N-2}(PU_{x,\mu})^{\frac{4-N}{N-2}}(PV_{x,\mu})^{\frac{N}{N-2}}\omega_{1}-\frac{N}{2(N-2)}(PU_{x,\mu})^{\frac{2}{N-2}}(PV_{x,\mu})^{\frac{2}{N-2}}\omega_{2},
\end{split}
\end{equation*}
\begin{equation*}\begin{split}
N_{2,1}(\omega_{1},\omega_{2})=K_2(y)\Bigl((PV_{x,\mu}+\omega_{2})_{+}^{2^{*}-1}-(PV_{x, \mu})^{2^{*}-1}-(2^{*}-1)(PV_{x, \mu})^{2^{*}-2}\omega_{2}\Bigr)
\end{split}
\end{equation*}
and
\begin{equation*}\begin{split}
N_{2,2}(\omega_{1},\omega_{2})&=\frac{1}{2}\Bigl((PV_{x,\mu}+\omega_{2})^{\frac{2}{N-2}}(PU_{x,\mu}+\omega_{1})^{\frac{N}{N-2}}-(PV_{x,\mu})^{\frac{2}{N-2}}(PU_{x,\mu})^{\frac{N}{N-2}}\Bigr)\\
&\quad -\frac{1}{N-2}(PV_{x,\mu})^{\frac{4-N}{N-2}}(PU_{x,\mu})^{\frac{N}{N-2}}\omega_{2}-\frac{N}{2(N-2)}(PV_{x,\mu})^{\frac{2}{N-2}}(PU_{x,\mu})^{\frac{2}{N-2}}\omega_{1}.
\end{split}
\end{equation*}

First, for $N_{1,1}$, we have
$$
|N_{1,1}(\omega_1,\omega_2)| \leq\left\{\begin{aligned}
&C|\omega_1|^{2^{*}-1}, & N \geq 6; \\
&C\left(P U_{x, \mu}\right)^{2^{*}-3} \omega_1^{2}+C|\omega_1|^{2^{*}-1}, & N = 5.
\end{aligned}\right.
$$

Since  $\tau<\frac{N-2}{2}$, we find from Lemma 2.1 that for  $N=5$,
$$\left(P U_{x, \mu}\right)^{2^{*}-3} \leq C\Bigl(\sum_{j=0}^{\infty} \frac{\mu^{\frac{N-2}{2}}}{\left(1+\mu\left|z-x_{j}\right|\right)^{\frac{N-2}{2}+\tau}}\Bigr)^{2^{*}-3}.$$
As a result, there holds
$$|N_{1,1}(\omega_1,\omega_2)| \leq C \sigma(z)\Bigl(\sum_{j=0}^{\infty} \frac{\mu^{\frac{N-2}{2}}}{\left(1+\mu\left|z-x_{j}\right|\right)^{\frac{N-2}{2}+\tau}}\Bigr)^{2^{*}-1}\|(\omega_1,\omega_2)\|_{*}^{\min \left(2^{*}-1,2\right)}.$$
On the other hand, from  $\tau>k$, we can check that for  $y \in \Omega$,
$$\sum_{j=0}^{\infty} \frac{1}{\left(1+\mu\left|y-x_{j}\right|\right)^{\tau}} \leq C.$$
So we have the following inequality
\begin{align}\begin{split}\label{eqs4.3}
\Bigl(\sum_{j=0}^{\infty} \frac{\mu^{\frac{N-2}{2}}}{\left(1+\mu\left|y-x_{j}\right|\right)^{\frac{N-2}{2}+\tau}}\Bigr)^{2^{*}-1}
&\leq C \sum_{j=0}^{\infty} \frac{\mu^{\frac{N+2}{2}}}{\left(1+\mu\left|y-x_{j}\right|\right)^{\frac{N+2}{2}+\tau}}\Bigl(\sum_{j=0}^{\infty} \frac{1}{\left(1+\mu\left|y-x_{j}\right|\right)^{\tau}}\Bigr)^{\frac{4}{N-2}} \\
&\leq C \sum_{j=0}^{\infty} \frac{\mu^{\frac{N+2}{2}}}{\left(1+\mu\left|y-x_{j}\right|\right)^{\frac{N+2}{2}+\tau}}.
\end{split}
\end{align}
Then we get
\begin{equation}\label{eqs4.3(1)}
\|N_{1,1}(\omega_1,\omega_2)\|_{**}\leq C\|(\omega_1,\omega_2)\|_{*}^{\min\{2^*-1,2\}}.
\end{equation}
 Rewrite
\begin{align}\begin{split}\label{eqs4.4}
N_{1,2}=&\frac{1}{2}(PU_{x,\mu}+\omega_1)^{\frac{2}{N-2}}\Big((PV_{x,\mu}+\omega_2)^{\frac{N}{N-2}}-PV_{x,\mu}^{\frac{N}{N-2}}-\frac N{N-2}PV_{x,\mu}^{\frac{2}{N-2}}\omega_{2}\Big)\\
&+\frac{1}{2}\Big((PU_{x,\mu}+\omega_1)^{\frac{2}{N-2}}-PU_{x,\mu}^{\frac{2}{N-2}}-\frac 2{N-2} PU_{x,\mu}^{\frac{4-N}{N-2}}\omega_{1}
\Big)PV_{x,\mu}^{\frac{N}{N-2}}\\
&+\frac N{2(N-2)}\Big( (PU_{x,\mu}+\omega_1)^{\frac{2}{N-2}}-PU_{x,\mu}^{\frac{2}{N-2}}\Big)PV_{x,\mu}^{\frac{2}{N-2}}\omega_{2}\\
:=&I+II+III.
\end{split}
\end{align}
First for $I$, we have
\begin{align*}
&|I|\leq C(|\omega_{2}|^{\frac{N}{N-2}}PU_{x,\mu}^{\frac{2}{N-2}}+|\omega_{2}|^{\frac{N}{N-2}}|\omega_1|^{\frac{2}{N-2}}):=I_1+I_2.
\end{align*}
We estimate by H\"older inequalities,
\begin{align*}
I_1&\leq C\|\omega_{2}\|_*^{\frac{N}{N-2}}\Big(\sigma(y)\sum_{j=0}^\infty\frac{\mu^{\frac{N-2}{2}}}{(1+\mu|y-x_j|)^{\frac{N-2}{2}+\tau}}\Big)^{\frac N{N-2}}
\Big(\sum_{j=0}^\infty\frac{\mu^{\frac{N-2}{2}}}{(1+\mu|y-x_j|)^{N-2}}\Big)^{\frac{2}{N-2}}\cr
&\leq C \|\omega_{2}\|_*^{\frac{N}{N-2}}\sigma(y)\Big(\sum_{j=0}^\infty\frac{\mu^{\frac{N-2}{2}}}{(1+\mu|y-x_j|)^{\frac{N-2}{2}+\tau}}\Big)^{2^*-1}\cr
&\leq C\|\omega_2\|_*^{\frac{N}{N-2}}\sigma(y)\sum_{j=0}^\infty\frac{\mu^{\frac{N+2}{2}}}{(1+\mu|y-x_j|)^{\frac{N+2}{2}+\tau}}\Big(\sum_{j=0}^\infty\frac{1}{(1+\mu|y-x_j|)^{\tau}}\Big)^{\frac4{N-2}}\cr
&\leq C\|\omega_2\|_*^{\frac{N}{N-2}}\sigma(y)\sum_{j=0}^\infty\frac{\mu^{\frac{N+2}{2}}}{(1+\mu|y-x_j|)^{\frac{N+2}{2}+\tau}}.
\end{align*}

Moreover, we have
\begin{align*}
I_2&\leq C\|\omega_{2}\|_*^{\frac{N}{N-2}}\|\omega_{1}\|_*^{\frac{2}{N-2}}\Big(\sigma(y)\sum_{j=0}^\infty\frac{\mu^{\frac{N-2}{2}}}{(1+\mu|y-x_j|)^{\frac{N-2}{2}+\tau}}\Big)^{\frac {N+2}{N-2}}\\
&\leq C\|\omega_{2}\|_*^{\frac{N}{N-2}}\|\omega_{1}\|_*^{\frac{2}{N-2}}\sigma(y)\sum_{j=0}^\infty\frac{\mu^{\frac{N+2}{2}}}{(1+\mu|y-x_j|)^{\frac{N+2}{2}+\tau}}
\Big(\sum_{j=0}^\infty\frac{1}{(1+\mu|y-x_j|)^{\tau}}\Big)^{\frac {4}{N-2}}\\
&\leq C\|\omega_{2}\|_*^{\frac{N}{N-2}}\|\omega_{1}\|_*^{\frac{2}{N-2}}\sigma(y)\sum_{j=0}^\infty\frac{\mu^{\frac{N+2}{2}}}{(1+\mu|y-x_j|)^{\frac{N+2}{2}+\tau}}.
\end{align*}
Therefore, we get
\begin{align}\label{2.4-3}
\begin{split}
&I\leq C(\|\omega_{2}\|_*^{\frac{N}{N-2}}+\|\omega_{2}\|_*^{\frac{N}{N-2}}\|\omega_{1}\|_*^{\frac{2}{N-2}})\sigma(y)\sum_{j=0}^\infty\frac{\mu^{\frac{N+2}{2}}}{(1+\mu|y-x_j|)^{\frac{N+2}{2}+\tau}}.
\end{split}
\end{align}

It follows from Lemma B.5 of \cite{guo-wang-wang} that $\max\{|\frac{\omega_1}{PU_{x,\mu}}|,|\frac{\omega_2}{PV_{x,\mu}}|\}<\frac{1}{2}.$ Using the basic inequality
$$(1+t)^{\frac{2}{N-2}}-1-\frac{2}{N-2}t=O(t^{2}),\,\,\,\text{for}\,\,\,|t|<\frac12,$$
if $N=5$ by using H\"older inequality again, we obtain
\begin{align}\label{2.4-4}
\begin{split}
 |II|&=\Bigl|\Big((PU_{x,\mu}+\omega_1)^{\frac{2}{N-2}}-PU_{x,\mu}^{\frac{2}{N-2}}-\frac 2{N-2} PU_{x,\mu}^{\frac{4-N}{N-2}}\omega_{1}
\Big)PV_{x,\mu}^{\frac{N}{N-2}}\Bigr|\\
&=PU_{x,\mu}^{\frac{2}{N-2}}PV_{x,\mu}^{\frac{N}{N-2}}\Bigl|(1+\frac{\omega_1}{PU_{x,\mu}})^{\frac{2}{N-2}}-1-\frac 2{N-2} \frac{\omega_1}{PU_{x,\mu}}
\Bigr|\\
&\leq C PU_{x,\mu}^{\frac{2}{N-2}-2}PV_{x,\mu}^{\frac{N}{N-2}}|\omega_1|^2\\
&\leq C \|\omega_{1}\|_*^{2}\sigma(y)\Big(\sum_{j=0}^\infty\frac{\mu^{\frac{N-2}{2}}}{(1+\mu|y-x_j|)^{N-2}}\Big)^{\frac{6-N}{N-2}}\Big(\sum_{j=0}^\infty\frac{\mu^{\frac{N-2}{2}}}{(1+\mu|y-x_j|)^{\frac{N-2}{2}+\tau}}\Big)^{2}\\
&\leq  C \|\omega_{1}\|_*^{2}\sigma(y)\Big(\sum_{j=0}^\infty\frac{\mu^{\frac{N-2}{2}}}{(1+\mu|y-x_j|)^{\frac{N-2}{2}+\tau}}\Big)^{\frac {N+2}{N-2}}\\
&\leq  C \|\omega_{1}\|_*^{2}\sigma(y)
\sum_{j=0}^\infty\frac{\mu^{\frac{N+2}{2}}}{(1+\mu|y-x_j|)^{\frac{N+2}{2}+\tau}}.
\end{split}
\end{align}

Also for $|t|<1$ and $N\geq 6,$ there holds
$$(1+t)^{\frac{2}{N-2}}-1-\frac{2}{N-2}t=O(t^{1+\delta}),$$
where $\frac{4}{N-2}-\delta>0.$
Hence when $N\geq 6,$ we have
\begin{align*}
\begin{split}
 |II|&=\Bigl|\Bigl((PU_{x,\mu}+\omega_1)^{\frac{2}{N-2}}-PU_{x,\mu}^{\frac{2}{N-2}}-\frac 2{N-2} PU_{x,\mu}^{\frac{4-N}{N-2}}\omega_{1}
\Bigr)PV_{x,\mu}^{\frac{N}{N-2}}\Bigr|\\
&=PU_{x,\mu}^{\frac{2}{N-2}}PV_{x,\mu}^{\frac{N}{N-2}}\Bigl|\Bigl(1+\frac{\omega_1}{PU_{x,\mu}}\Bigr)^{\frac{2}{N-2}}-1-\frac 2{N-2} \frac{\omega_1}{PU_{x,\mu}}
\Bigr|\\
&\leq C PU_{x,\mu}^{\frac{2}{N-2}}PV_{x,\mu}^{\frac{N}{N-2}}\Bigl|\frac{\omega_1}{PU_{x,\mu}}\Bigr|^{1+\delta}\cr
&\leq C\|\omega_1\|_{*}^{1+\delta}\sigma(y)\Bigl(\sum\limits_{j=0}^\infty\frac{\mu^{\frac{N-2}{2}}}{(1+\mu|y-x_j|)^{N-2}}\Bigr)^{\frac{4}{N-2}-\delta}\Bigl(\sum\limits_{j=0}^\infty\frac{\mu^{\frac{N-2}{2}}}
{(1+\mu|y-x_j|)^{\frac{N-2}{2}+\tau}}\Bigr)^{1+\delta}\cr
&\leq C\|\omega_1\|_{*}^{1+\delta}\sigma(y)\Bigl(\sum\limits_{j=0}^\infty\frac{\mu^{\frac{N-2}{2}}}{(1+\mu|y-x_j|)^{\frac{N-2}{2}+\tau}}\Bigr)^{\frac{N+2}{N-2}}\\
&\leq C\|\omega_1\|_{*}^{1+\delta}\sigma(y)\sum\limits_{j=0}^\infty\frac{\mu^{\frac{N+2}{2}}}{(1+\mu|y-x_j|)^{\frac{N+2}{2}+\tau}},
\end{split}
\end{align*}
which together with \eqref{2.4-4}, gives
\begin{align}\label{2.4-4''}
\begin{split}
 |II|\leq C \|\omega_{1}\|_*^{1+\delta}\sigma(y)\sum_{j=0}^\infty\frac{\mu^{\frac{N+2}{2}}}{(1+\mu|y-x_j|)^{\frac{N+2}{2}+\tau}}.
\end{split}
\end{align}

Finally, for $III$, we have
\begin{align}\label{2.4-5}
\begin{split}
 |III|&\leq C
|\omega_1|^{\frac{2}{N-2}}|\omega_{2}|PU_{x,\mu}^{\frac{2}{N-2}}\\
&\leq C\|\omega_{1}\|_*^{\frac{2}{N-2}}\|\omega_{2}\|_*\sigma(y)\Big(\sum_{j=0}^\infty\frac{\mu^{\frac{N-2}{2}}}{(1+\mu|y-x_j|)^{\frac{N-2}{2}+\tau}}\Big)^{\frac{N}{N-2}}
\Big(\sum_{j=0}^\infty\frac{\mu^{\frac{N-2}{2}}}{(1+\mu|y-x_j|)^{N-2}}\Big)^{\frac{2}{N-2}}\\
&\leq C\|\omega_{1}\|_*^{\frac{2}{N-2}}\|\omega_{2}\|_*\sigma(y)\Big(\sum_{j=0}^\infty\frac{\mu^{\frac{N-2}{2}}}{(1+\mu|y-x_j|)^{\frac{N-2}{2}+\tau}}\Big)^{\frac {N+2}{N-2}}\\&
\leq  C \|\omega_{1}\|_*^{\frac{2}{N-2}}\|\omega_{2}\|_*\sigma(y)
\sum_{j=0}^\infty\frac{\mu^{\frac{N+2}{2}}}{(1+\mu|y-x_j|)^{\frac{N+2}{2}+\tau}}.
\end{split}
\end{align}

From \eqref{eqs4.3(1)} to \eqref{2.4-5}, we obtain
\begin{align*}
\|N_1(\omega_1,\omega_2)\|_{**}\leq C\|(\omega_1,\omega_2)\|_*^{1+\delta}.
\end{align*}
Similarly, we have
\begin{align*}
\|N_2(\omega_1,\omega_2)\|_{**}\leq C\|(\omega_1,\omega_2)\|_*^{1+\delta}.
\end{align*}
Thus the result follows.
\end{proof}

\begin{lemma}\label{lem4.3}
If  $N \geq 5$, then
$$\|(\ell_1,\ell_2)\|_{* *} \leq C\Bigl(\frac{1}{\mu^{2}}+\frac{|\mu x|^{\beta}}{\mu^{\beta-\tau+1}}\Bigr).$$
\end{lemma}
\begin{proof}
We have
\begin{align*}
\begin{split}
\ell_1=&K_1(y)\Bigl(\left(P U_{x, \mu}\right)^{2^{*}-1}-U_{x, \mu}^{2^{*}-1}\Bigr)+(K_1(y)-1) U_{x, \mu}^{2^{*}-1}-\frac{1}{2}\Bigl(U_{x,\mu}^{\frac{2^*}{2}-1}V_{x,\mu}^{\frac{2^*}{2}}-PU_{x,\mu}^{\frac{2^*}{2}-1}PV_{x,\mu}^{\frac{2^*}{2}}\Bigr)\\
:=&J_{1}+J_{2}+J_3.
\end{split}
\end{align*}

It follows from Lemma 2.2 that
\begin{equation}\label{eqs4.4}
\left|J_{1}\right|\leq C\Bigl|(P U_{x, \mu})^{2^{*}-1}-U_{x, \mu}^{2^{*}-1}\Bigr| \leq C U_{x, \mu}^{2^{*}-2}\left|\varphi_1\right|+C\left|\varphi_1\right|^{2^{*}-1}:=J_{11}+J_{12}.
\end{equation}
To estimate  $J_{11}$, noting that  $\beta>N-2$, for  $y \in B_{1}(x)$, we have
\begin{align}\begin{split}\label{eqs4.5}
\left|J_{11}\right| & \leq \frac{C \mu^{2}}{(1+\mu|y-x|)^{4}} \frac{1}{\mu^{\frac{N-2}{2}} L^{N-2}}\leq \frac{C \mu^{\frac{N+2}{2}}}{(1+\mu|y-x|)^{4}} \frac{1}{\mu^{\beta}} \\
& \leq \frac{C \mu^{\frac{N+2}{2}}}{(1+\mu|y-x|)^{\frac{N+2}{2}+\tau}}\Bigl(\frac{1+\mu|y-x|}{\mu}\Bigr)^{\tau-1} \frac{1}{\mu^{2}} \\
& \leq \frac{C \sigma(y) \mu^{\frac{N+2}{2}}}{(1+\mu|y-x|)^{\frac{N+2}{2}+\tau}} \frac{1}{\mu^{2}}.
\end{split}
\end{align}
For  $y \in \Omega \backslash B_{1}(x)$, we have
\begin{equation}\label{eqs4.6}
\left|J_{11}\right| \leq C U_{x, \mu}^{2^{*}-2} \sum_{j=0}^{\infty} U_{x_{j}, \mu},
\end{equation}
\begin{equation}\label{eqs4.7}
\frac{\mu^{\frac{N+2}{2}}}{(1+\mu|y-x|)^{N+2}} \leq \frac{C \mu^{\frac{N+2}{2}}}{(1+\mu|y-x|)^{\frac{N+2}{2}+\tau}} \frac{1}{\mu^{2}}
\end{equation}
and
\begin{align}\begin{split}\label{eqs4.8}
\frac{\mu^{\frac{N+2}{2}}}{(1+\mu|y-x|)^{4}} \sum_{j=1}^{\infty} \frac{1}{\left(1+\mu\left|y-x_{j}\right|\right)^{N-2}}
\leq &\frac{C \mu^{\frac{N+2}{2}}}{(1+\mu|y-x|)^{\frac{N+2}{2}+\tau}} \sum_{j=1}^{\infty} \frac{1}{\left(1+\mu\left|y-x_{j}\right|\right)^{\frac{N+2}{2}-\tau}} \cr
\leq &\frac{C \mu^{\frac{N+2}{2}}}{(1+\mu|y-x|)^{\frac{N+2}{2}+\tau}} \frac{1}{(\mu L)^{\frac{N+2}{2}-\tau}}.
\end{split}
\end{align}
Combining \eqref{eqs4.5}-\eqref{eqs4.8}, we obtain
\begin{equation}\label{eqs4.9}
\left\|J_{11}\right\|_{* *} \leq \frac{C}{\mu^{2}}.
\end{equation}
Similarly, for  $y \in B_{1}(x)$, we have
\begin{align}\begin{split}\label{eqs4.10}
\left|J_{12}\right|  \leq \frac{C}{\mu^{\frac{N+2}{2}} L^{N+2}}\leq \frac{C \mu^{\frac{N+2}{2}}}{\mu^{\left(2^{*}-1\right) \beta}}
\leq& \frac{C \mu^{\frac{N+2}{2}}}{(1+\mu|y-x|)^{\frac{N+2}{2}+\tau}}\Bigl(\frac{1+\mu|y-x|}{\mu}\Bigr)^{\tau-1} \frac{1}{\mu^{2}} \\
 \leq &\frac{C \sigma(y) \mu^{\frac{N+2}{2}}}{(1+\mu|y-x|)^{\frac{N+2}{2}+\tau}} \frac{1}{\mu^{2}}.
\end{split}
\end{align}
For  $y \in \Omega \backslash B_{1}(x)$, similar to \eqref{eqs4.3}, we get
\begin{align}\begin{split}\label{eqs4.11}
\Bigl(\sum_{j=1}^{\infty} \frac{1}{\left(1+\mu\left|y-x_{j}\right|\right)^{N-2}}\Bigr)^{2^{*}-1}
&\leq\Bigl(\sum_{j=1}^{\infty} \frac{1}{\left(1+\mu\left|y-x_{j}\right|\right)^{\frac{N-2}{2}+\tau}}\Bigr)^{2^{*}-1} \cr
&\leq C \sum_{j=1}^{\infty} \frac{1}{\left(1+\mu\left|y-x_{j}\right|\right)^{\frac{N+2}{2}+\tau}}\Bigl(\sum_{j=1}^{\infty} \frac{1}{\left(1+\mu\left|y-x_{j}\right|\right)^{\tau}}\Bigr)^{\frac{4}{N-2}}\cr
&\leq C \sum_{j=1}^{\infty} \frac{\mu^{\frac{N+2}{2}}}{\left(1+\mu\left|y-x_{j}\right|\right)^{\frac{N+2}{2}+\tau}} \frac{1}{(\mu L)^{\frac{4 \tau}{N-2}}}.
\end{split}
\end{align}
Thus by \eqref{eqs4.7} and \eqref{eqs4.11}, for  $y \in \Omega \backslash B_{1}(x)$, we see
\begin{align}\begin{split}\label{eqs4.12}
\left|J_{12}\right| & \leq\Bigl(\sum_{j=0}^{\infty} U_{x_j, \mu}\Bigr)^{2^{*}-1} \leq C U_{x, \mu}^{2^{*}-1}+C\Bigl(\sum_{j=1}^{\infty} U_{x_j, \mu}\Bigr)^{2^{*}-1} \\
& \leq \frac{C \mu^{\frac{N+2}{2}}}{(1+\mu|y-x|)^{\frac{N+2}{2}+\tau}} \frac{1}{\mu^{2}}+C \sum_{j=1}^{\infty} \frac{\mu^{\frac{N+2}{2}}}{\left(1+\mu\left|y-x_{j}\right|\right)^{\frac{N+2}{2}+\tau}} \frac{1}{(\mu L)^{\frac{4 \tau}{N-2}}}.
\end{split}
\end{align}
Combining \eqref{eqs4.10} and \eqref{eqs4.12}, we obtain
\begin{equation}\label{eqs4.13}
\left\|J_{12}\right\|_{* *} \leq \frac{C}{\mu^{2}}.
\end{equation}
Thus, \eqref{eqs4.9} and \eqref{eqs4.13} yield
\begin{equation}\label{eqs4.14}
\left\|J_{1}\right\|_{* *} \leq \frac{C}{\mu^{2}}.
\end{equation}
Now, we estimate  $J_{2}$. If  $|y-x| \geq 1$, then we have
\begin{equation}\label{eqs4.15}
|K_1(y)-1| U_{x, \mu}^{2^{*}-1} \leq \frac{C \mu^{\frac{N+2}{2}}}{(1+\mu|y-x|)^{\frac{N+2}{2}+\tau}} \frac{1}{\mu^{\frac{N+2}{2}-\tau}}.
\end{equation}
If  $|y-x| \leq 1$, then we get

\begin{align}\begin{split}\label{eqs4.16}
|K_1(y)-1| U_{x, \mu}^{2^{*}-1} & \leq C|y|^{\beta} U_{x, \mu}^{2^{*}-1} \\
& \leq \frac{C \sigma(y) \mu^{\frac{N+2}{2}}}{(1+\mu|y-x|)^{\frac{N+2}{2}+\tau}} \frac{\mu^{\tau-1}\left(|x|^{\beta}+|y-x|^{\beta}\right)}{(1+\mu|y-x|)^{\frac{N+2}{2}-1}}.
\end{split}
\end{align}

Noting that  $\beta>N-2>\frac{N}{2}$, we have
$$\frac{\mu^{\tau-1}|y-x|^{\beta}}{(1+\mu|y-x|)^{\frac{N+2}{2}-1}} \leq \frac{C}{\mu^{\frac{N+2}{2}-\tau}} \leq \frac{C}{\mu^{2}}.$$
Hence, we have proved
\begin{equation}\label{eqs4.17}
\left\|J_{2}\right\|_{* *} \leq \frac{C}{\mu^{2}}+C \mu^{\tau-1}|x|^{\beta}.
\end{equation}
Since for $V_{x,\mu}=\kappa U_{x,\mu}$, we have
\begin{align*}
J_3=-\frac{1}{2}\Bigl(U_{x,\mu}^{\frac{2^*}{2}-1}V_{x,\mu}^{\frac{2^*}{2}}-PU_{x,\mu}^{\frac{2^*}{2}-1}PV_{x,\mu}^{\frac{2^*}{2}}\Bigr)=C\Bigl(PU_{x,\mu}^{2^*-1}-U_{x,\mu}^{2^*-1}\Bigr).
\end{align*}
Similar as $J_1$, we obtain $\|J_3\|_{**}\leq \frac{C}{\mu^2}$.

For $\ell_2$, we proceed it as $\ell_1$. As a result, we complete the proof.
\end{proof}

\begin{proposition}\label{pro4.1}
Under the assumptions of Theorem 1.1, if  $L>0$  is sufficiently large, then \eqref{eqs4.1} admits a unique solution  $(\omega_1,\omega_2)=(\omega_1(x, \mu),\omega_2(x, \mu)) \in \mathbf{E}$ , such that
$$
\|(\omega_1,\omega_2)\|_{*} \leq C\Bigl(\frac{1}{\mu^{2}}+\frac{|\mu x|^{\beta}}{\mu^{\beta-\tau+1}}\Bigr).
$$
Moreover  $\|(\omega_1,\omega_2)\|_{*}$  is  $C^{1}$  in  $(x, \mu)$.
\end{proposition}
\begin{proof}
In view of Proposition 3.2, \eqref{eqs4.1} is equivalent to
\begin{equation}\label{eqs4.2(1)}
(\omega_1,\omega_2)=B (\omega_1,\omega_2):=(I-T)^{-1}(-\Delta)^{-1}(\mathbf{P}\ell+\mathbf{P} N(\omega_{1},\omega_{2})), \quad (\omega_1,\omega_2) \in \mathbf{E}.
\end{equation}
By Lemma 3.6, we have
$$\|B (\omega_1,\omega_2)\|_{*} \leq C\|\ell\|_{**}+C\|N(\omega_1,\omega_2)\|_{* *},$$
and
$$\left\|B\left(\omega_{1}, \omega_{2}\right)-B\left(\widetilde{\omega_{1}}, \widetilde{\omega_{2}}\right)\right\|_{*} \leq C\left\|N\left(\omega_{1}, \omega_{2}\right)-N\left(\widetilde{\omega_{1}}, \widetilde{\omega_{2}}\right)\right\|_{* *}.$$
By Lemmas 4.2 and 4.3, we can prove that  $B$  is a contraction map from the set
 $$\Bigl\{(\omega_1,\omega_2) \in   \left.\mathbf{E}:\|(\omega_1,\omega_2)\|_{*} \leq \frac{M}{\mu^{2}}\right\}$$
to itself, where  $M>0$  is a large constant. So the contraction mapping theorem implies that for large  $L>0$, \eqref{eqs4.2(1)} has a solution  $(\omega_1,\omega_2) \in \mathbf{E}$, satisfying
$$
\|(\omega_1,\omega_2)\|_{*} \leq C\|(\ell_1,\ell_2)\|_{* *}.
$$
Moreover, similar to the proof in \cite{Li-Wei}, we can show that  $\left\|\left(\omega_{1}, \omega_{2}\right)\right\|_{*}$  is  $C^{1}$  in  $(x, \mu)$  by implicit function theorem.
\end{proof}

Now we will prove Theorem 1.1

\textbf{Proof of Theorem 1.1} We need to show the existence of  $\left(x, \mu\right)=\left(x_{L}, \mu_{L}\right)$, such that
\begin{align}\begin{split}\label{eqs4.18}
-&\int_{\Omega} \Delta\left(P U_{x, \mu}+\omega_1\right) \partial_{h}\left(P U_{x, \mu}\right)
-\int_{\Omega} K_1(y)\left(P U_{x, \mu}+\omega_1\right)_{+}^{2^{*}-1} \partial_{h}\left(P U_{x, \mu}\right)\cr
\quad -&\frac{1}{2}\int_{\Omega}(PU_{x,\mu}+\omega_1)^{\frac{2^*}{2}-1}(PV_{x,\mu}+\omega_2)^{\frac{2^*}{2}}\partial_{h}PU_{x, \mu}
-\int_{\Omega} \Delta\left(P V_{x, \mu}+\omega_2\right) \partial_{h}\left(P V_{x, \mu}\right)\\
\quad -&\int_{\Omega} K_2(y)\left(P V_{x, \mu}+\omega\right)_{+}^{2^{*}-1} \partial_{h}\left(P V_{x, \mu}\right)
-\frac{1}{2}\int_{\Omega}(PV_{x,\mu}+\omega_2)^{\frac{2^*}{2}-1}(PU_{x,\mu}+\omega_1)^{\frac{2^*}{2}}\partial_{h}P V_{x, \mu}=0
\end{split}
\end{align}
for  $h=1, \ldots, N+1$.

From  $(\omega_1,\omega_2) \in \mathbf{E}$, we obtain
\begin{equation}\label{eqs4.19}
-\int_{\Omega} \Delta \omega_1 \partial_{h}\left(P U_{x, \mu}\right)=-\int_{\Omega} \Delta\left(\partial_{h}\left(P U_{x, \mu}\right)\right) \omega_1=C \int_{\Omega} \omega_1 U_{x, \mu}^{2^{*}-2} Z_{h}=0.
\end{equation}

Noting that for any  $\gamma>1$, we have  $(1+t)_{+}^{\gamma}-1-\gamma t=O\left(t^{2}\right)$  for all  $t \in \mathbb{R}$  if  $\gamma \leq 2$; and  $\left|(1+t)_{+}^{\gamma}-1-\gamma t\right| \leq C\left(t^{2}+|t|^{\gamma}\right)$  for all  $t \in \mathbb{R}$  if  $\gamma>2$. So, we can deduce that if  $N \geq 6$,
\begin{align}\begin{split}\label{eqs4.20}
&\Bigl|\int_{\Omega} K_1(y)\left[\left(P U_{x, \mu}+\omega_1\right)_{+}^{2^{*}-1}-\left(P U_{x, \mu}\right)^{2^{*}-1}-\left(2^{*}-1\right)\left(P U_{x, \mu}\right)^{2^{*}-2} \omega_1\right] \partial_{h}\left(P U_{x, \mu}\right)\Bigr| \\
&\quad \leq C \int_{\Omega}\left(P U_{x, \mu}\right)^{2^{*}-3}|\omega_1|^{2}\left|\partial_{h}\left(P U_{x, \mu}\right)\right| \leq C \mu^{\alpha(h)} \int_{\Omega}\left(P U_{x, \mu}\right)^{2^{*}-2}|\omega_1|^{2}.
\end{split}
\end{align}
We have
\begin{align}\label{eqs4.21}
&\int_{\Omega \backslash B_{1}(x)}\left(P U_{x, \mu}\right)^{2^{*}-2}|\omega_1|^{2}\\
&\quad \leq C\|\omega_1\|_{*}^{2} \int_{\Omega \backslash B_{1}(x)}\left(P U_{x, \mu}\right)^{2^{*}-2}\Bigl[\sum_{j=0}^{\infty} \frac{\mu^{\frac{N-2}{2}}}{\left(1+\mu\left|y-x_{j}\right|\right)^{\frac{N-2}{2}+\tau}}\Bigr]^{2} \cr
&\quad \leq C\|\omega_1\|_{*}^{2} \int_{\Omega \backslash B_{1}(x)}\Bigl[\sum_{j=0}^{\infty} \frac{\mu^{\frac{N-2}{2}}}{\left(1+\mu\left|y-x_{j}\right|\right)^{\frac{N-2}{2}+\tau}}\Bigr]^{2^{*}} \cr
&\quad \leq C\|\omega_1\|_{*}^{2} \mu^{N} \int_{\Omega \backslash B_{1}(x)}\Bigl(\frac{1}{(1+\mu|y-x|)^{N-2-\vartheta}}+\frac{1}{(1+\mu|y-x|)^{\frac{N-2}{2}+\vartheta}} \frac{1}{(\mu L)^{\frac{N-2}{2}-2 \vartheta}}\Bigr)^{2^{*}} \cr
&\quad \leq C\|\omega_1\|_{*}^{2}\Bigl(\frac{1}{\mu^{N-2^{*} \vartheta}}+\frac{1}{(\mu L)^{N-2^{*}(2 \vartheta)}}\Bigr)
\end{align}
and
\begin{align}\begin{split}\label{eqs4.22}
&\int_{B_{1}(x)}(P U_{x, \mu})^{2^{*}-2}|\omega_1|^{2}\\
&\quad \leq C\|\omega_1\|_{*}^{2} \int_{B_{1}(x)}(P U_{x, \mu})^{2^{*}-2}\Bigl[\frac{(1+\mu|y-x|)^{\tau-1}}{\mu^{\tau-1}} \sum_{j=0}^{\infty} \frac{\mu^{\frac{N-2}{2}}}{(1+\mu|y-x_{j}|)^{\frac{N-2}{2}+\tau}}\Bigr]^{2}\cr
&\quad \leq C\|\omega_1\|_{*}^{2} \mu^{N} \int_{B_{1}(x)} \frac{1}{(1+\mu|y-x|)^{4}}\Bigl[\frac{(1+\mu|y-x|)^{\tau-1}}{\mu^{\tau-1}} \frac{1}{(1+\mu|y-x|)^{\frac{N-2}{2}+\tau}}\Bigr]^{2}\\
&\quad \leq C(\mu^{1-\tau}\|\omega_1\|_{*})^{2}.
\end{split}\end{align}
Combining \eqref{eqs4.20}-\eqref{eqs4.22}, we obtain
\begin{align}\label{eqs4.23}
&\Bigl|\int_{\Omega} K_1(y)\Bigl[\left(P U_{x, \mu}+\omega_1\right)_{+}^{2^{*}-1}-\left(P U_{x, \mu}\right)^{2^{*}-1}-\left(2^{*}-1\right)\left(P U_{x, \mu}\right)^{2^{*}-2} \omega_1\Bigr] \partial_{h}\left(P U_{x, \mu}\right)\Bigr| \cr
&\quad \leq C\left(\mu^{1-\tau}\|\omega_1\|_{*}\right)^{2} \mu^{\alpha(h)}.
\end{align}
Similarly, if  $N=5$, then we have
\begin{align}\begin{split}\label{eqs4.24}
&\Bigl|\int_{\Omega} K_1(y)\Bigl[\left(P U_{x, \mu}+\omega_1\right)_{+}^{2^{*}-1}-\left(P U_{x, \mu}\right)^{2^{*}-1}-\left(2^{*}-1\right)\left(P U_{x, \mu}\right)^{2^{*}-2} \omega_1\Bigr] \partial_{h}\left(P U_{x, \mu}\right)\Bigr| \\
&\quad \leq C \int_{\Omega}\Bigl[\left(P U_{x, \mu}\right)^{2^{*}-3}|\omega_1|^{2}+|\omega_1|^{2^{*}-1}\Bigr]\left|\partial_{h}\left(P U_{x, \mu}\right)\right| \leq C\left(\mu^{1-\tau}\|\omega_1\|_{*}\right)^{2} \mu^{\alpha(h)}.
\end{split}
\end{align}
From \eqref{eqs4.19}-\eqref{eqs4.24}, we find
\begin{align}\begin{split}\label{eqs4.25}
&-\int_{\Omega} \Delta\left(P U_{x, \mu}+\omega_1\right) \partial_{h}\left(PU_{x, \mu}\right)-\int_{\Omega} K_1(y)\left(P U_{x, \mu}+\omega_1\right)_{+}^{2^{*}-1} \partial_{h}(PU_{x, \mu})\cr
&\quad -\frac{1}{2}\int_{\Omega}(PU_{x,\mu}+\omega_1)^{\frac{2^*}{2}-1}(PV_{x,\mu}+\omega_2)^{\frac{2^*}{2}}\partial_{h}PU_{x, \mu} \cr
=&-\int_{\Omega}\Delta\left(P U_{x,\mu}\right)\partial_{h}\left(P U_{x, \mu}\right)-\int_{\Omega} K_1(y)\left(P U_{x,\mu}\right)^{2^{*}-1}\partial_{h}\left(P U_{x, \mu}\right)\\
&-(2^*-1)\int_{\Omega}K_1(y)(PU_{x,\mu})^{2^*-2}\omega_1\partial_{h}\left(P U_{x, \mu}\right)\cr
&-\frac{1}{2}\int_{\Omega}\Bigl((PU_{x,\mu}+\omega_1)^{\frac{2^*}{2}-1}(PV_{x,\mu}+\omega_2)^{\frac{2^*}{2}}-(PU_{x,\mu})^{\frac{2^*}{2}-1}PV_{x,\mu}^{\frac{2^*}{2}}\Bigr)\partial_{h}PU_{x, \mu}\\
&-\frac{1}{2}\int_{\Omega}PU_{x,\mu}^{\frac{2^*}{2}-1}PV_{x,\mu}^{\frac{2^*}{2}}\partial_{h}PU_{x, \mu}\\
&-\int_{\Omega} K_1(y)\Bigl(\left(P U_{x, \mu}+\omega_1\right)_{+}^{2^{*}-1}-\left(P U_{x, \mu}\right)^{2^{*}-1}-\left(2^{*}-1\right)\left(P U_{x, \mu}\right)^{2^{*}-2} \omega_1\Bigr) \partial_{h}\left(P U_{x, \mu}\right)\\
=&-\int_{\Omega}\Delta\left(P U_{x,\mu}\right)\partial_{h}\left(P U_{x, \mu}\right)-\int_{\Omega} K_1(y)\left(P U_{x,\mu}\right)^{2^{*}-1}\partial_{h}\left(P U_{x, \mu}\right)-\frac{1}{2}\int_{\Omega}PU_{x,\mu}^{\frac{2^*}{2}-1}PV_{x,\mu}^{\frac{2^*}{2}}\partial_{h}PU_{x, \mu}\cr
&-(2^*-1)\int_{\Omega}K_1(y)PU_{x,\mu}^{2^*-2}\omega_1\partial_{h}\left(P U_{x, \mu}\right)+\mu^{\alpha(h)}O(\mu^{1-\tau}\|\omega_1\|_{*}^2)\cr
&-\frac{1}{2}\int_{\Omega}\Bigl((PU_{x,\mu}+\omega_1)^{\frac{2^*}{2}-1}(PV_{x,\mu}+\omega_2)^{\frac{2^*}{2}}-PU_{x,\mu}^{\frac{2^*}{2}-1}PV_{x,\mu}^{\frac{2^*}{2}}\Bigr)\partial_{h}PU_{x, \mu}.
\end{split}
\end{align}
From \eqref{eqs3.13}, we have
\begin{equation}\label{eqs4.26}
\Bigl|\int_{\Omega} K_1(y)\left(P U_{x, \mu}\right)^{2^{*}-2} \partial_{h}\left(P U_{x, \mu}\right) \omega_1\Bigr| \leq \frac{C \mu^{\alpha(h)}\|(\omega_1,\omega_2)\|_{*}}{\mu^{\beta_1}} \leq \frac{C \mu^{\alpha(h)}}{\mu^{\beta_1+2}}
\end{equation}
and
\begin{equation}\begin{split}\label{eqs4.26(1)}
&\Bigl|\int_{\Omega}\Bigl((PU_{x,\mu}+\omega_1)^{\frac{2^*}{2}-1}(PV_{x,\mu}+\omega_2)^{\frac{2^*}{2}}-PU_{x,\mu}^{\frac{2^*}{2}-1}PV_{x,\mu}^{\frac{2^*}{2}}\Bigr)\partial_{h}PU_{x, \mu}\Bigr|\\
&\quad \leq \frac{C \mu^{\alpha(h)}\|(\omega_1,\omega_2)\|_{*}}{\mu^{\beta_1}} \leq \frac{C \mu^{\alpha(h)}}{\mu^{\beta_1+2}}.
\end{split}\end{equation}

We also have
\begin{equation}\label{eqs4.27}
\left(\mu^{1-\tau}\|(\omega_1,\omega_2)\|_{*}\right)^{2} \leq \frac{C}{\mu^{N-2 \vartheta}}.
\end{equation}
Putting \eqref{eqs4.26} and \eqref{eqs4.27} into \eqref{eqs4.25}, we obtain
\begin{align}\begin{split}\label{eqs4.28}
&-\int_{\Omega}\Delta\left(PU_{x,\mu}+\omega_1\right)\partial_{h}\left(PU_{x,\mu}\right)-\int_{\Omega}K_1(y)\left(PU_{x,\mu}+\omega_1\right)_{+}^{2^{*}-1}\partial_{h}\left(PU_{x, \mu}\right)\cr
&\quad -\frac{1}{2}\int_{\Omega}(PU_{x,\mu}+\omega_1)^{\frac{2^*}{2}-1}(PV_{x,\mu}+\omega_2)^{\frac{2^*}{2}}\partial_{h}PU_{x,\mu} \cr
=&-\int_{\Omega} \Delta\left(PU_{x, \mu}\right) \partial_{h}\left(PU_{x, \mu}\right)-\int_{\Omega}K_1(y)\left(P U_{x, \mu}\right)^{2^{*}-1} \partial_{h}\left(PU_{x, \mu}\right) \\
&\quad -\frac{1}{2}\int_{\Omega}PU_{x,\mu}^{\frac{2^*}{2}-1}PV_{x,\mu}^{\frac{2^*}{2}}\partial_{h}PU_{x, \mu}+\mu^{\alpha(h)} O\Bigl(\frac{1}{\mu^{N-2 \vartheta}}\Bigr).
\end{split}
\end{align}
Similarly, we have
\begin{align}\begin{split}\label{eqs4.28(1)}
&-\int_{\Omega}\Delta\left(PV_{x,\mu}+\omega_2\right)\partial_{h}\left(PV_{x,\mu}\right)-\int_{\Omega}K_2(y)\left(PV_{x,\mu}+\omega_2\right)_{+}^{2^{*}-1}\partial_{h}\left(PV_{x, \mu}\right)\cr
&\quad -\frac{1}{2}\int_{\Omega}(PV_{x,\mu}+\omega_2)^{\frac{2^*}{2}-1}(PU_{x,\mu}+\omega_1)^{\frac{2^*}{2}}\partial_{h}PV_{x,\mu} \cr
=&-\int_{\Omega} \Delta\left(PV_{x, \mu}\right) \partial_{h}\left(PV_{x, \mu}\right)-\int_{\Omega}K_2(y)\left(P V_{x, \mu}\right)^{2^{*}-1} \partial_{h}\left(PV_{x, \mu}\right) \\
&\quad -\frac{1}{2}\int_{\Omega}PV_{x,\mu}^{\frac{2^*}{2}-1}PU_{x,\mu}^{\frac{2^*}{2}}\partial_{h}PV_{x, \mu}+\mu^{\alpha(h)} O\Bigl(\frac{1}{\mu^{N-2 \vartheta}}\Bigr).
\end{split}
\end{align}
It follows from Propositions 2.3 and 2.4, \eqref{eqs4.28} and \eqref{eqs4.28(1)} that \eqref{eqs4.18} is equivalent to
\begin{equation}\label{eqs4.29}
\frac{B_{l}\mu x_{h}}{\mu^{\beta_{1h}-1}}+\frac{B_{l}\mu x_{h}}{\mu^{\beta_{2h}-1}}=O\Bigl(\frac{1}{\mu^{\beta_{1M}-1+\sigma}}+\frac{\left(\mu x_{h}\right)^{2}}{\mu^{\beta_{1h}-1}}+\frac{1}{\mu^{\beta_1} L}+\frac{1}{\mu^{\beta_{2M}-1+\sigma}}+\frac{\left(\mu x_{h}\right)^{2}}{\mu^{\beta_{2h}-1}}+\frac{1}{\mu^{\beta_2} L}\Bigr)
\end{equation}
and
%\begin{equation}\label{eqs4.290}
%\frac{B_{h}\mu x_{h}}{\mu^{\beta_{2h}-1}}=O\Bigl(\frac{1}{\mu^{\beta_{2M}-1+\sigma}}+\frac{\left(\mu x_{h}\right)^{2}}{\mu^{\beta_{2h}-1}}+\frac{1}{\mu^{\beta_2} L}\Bigr),
%\end{equation}
\begin{align}\begin{split}\label{eqs4.30}
&- \frac{1}{\mu^{\beta_1}} \sum_{i \in J} a_{1i} \int_{\mathbb{R}^{N}}|y_{i}|^{\beta_1} U_{0,1}^{2^{*}-1} \psi_{0}
-\frac{\left(2^{*}-1\right) B_1 \ds\int_{\mathbb{R}^{N}} U_{0,1}^{2^{*}-2} \psi_{0}}{\mu^{N-2} L^{N-2}} \sum_{j=1}^{\infty} \Gamma\left(P_{j}, 0\right)\cr
&- \frac{1}{\mu^{\beta_2}} \sum_{i \in J} a_{2i} \int_{\mathbb{R}^{N}}|y_{i}|^{\beta_2} V_{0,1}^{2^{*}-1} \xi_{0}
-\frac{\left(2^{*}-1\right) B_2 \ds\int_{\mathbb{R}^{N}} V_{0,1}^{2^{*}-2} \xi_{0}}{\mu^{N-2} L^{N-2}} \sum_{j=1}^{\infty} \Gamma\left(P_{j}, 0\right)
=o\Bigl(\frac{1}{\mu^{\beta_1}}+\frac{1}{\mu^{\beta_2}}\Bigr).
\end{split}
\end{align}
Since $\beta=\min\{\beta_1,\beta_2\}$ and $L \sim \mu^{\frac{\beta}{N-2}-1}$, we see
$$\frac{1}{\mu^{\beta} L}=O\Bigl(\frac{1}{\mu^{\beta-1+\theta}}\Bigr)$$
for some small  $\theta>0$. Note that
$$\int_{\mathbb{R}^{N}} U_{0,1}^{2^{*}-2} \psi_{0}=-\frac{N-2}{2\left(2^{*}-1\right)} \int_{\mathbb{R}^{N}} U_{0,1}^{2^{*}-1}<0,$$
$$\int_{\mathbb{R}^{N}} V_{0,1}^{2^{*}-2} \xi_{0}=-\frac{N-2}{2\left(2^{*}-1\right)} \int_{\mathbb{R}^{N}} V_{0,1}^{2^{*}-1}<0,$$
$$\int_{\mathbb{R}^{N}}|y_i|^{\beta} U_{0,1}^{2^{*}-2} \psi_{0}=-\frac{\beta+\frac{N-2}{2}}{2^{*}-1} \int_{\mathbb{R}^{N}}|y_i|^{\beta} U_{0,1}^{2^{*}-1}<0$$
and
$$\int_{\mathbb{R}^{N}}|y_i|^{\beta} V_{0,1}^{2^{*}-2} \xi_{0}=-\frac{\beta+\frac{N-2}{2}}{2^{*}-1} \int_{\mathbb{R}^{N}}|y_i|^{\beta} V_{0,1}^{2^{*}-1}<0.$$
Clearly, \eqref{eqs4.29} and \eqref{eqs4.30} have a solution  $\left(x_{L}, \mu_{L}\right)$, satisfying
$$\left|x_{L}\right| \leq \frac{C}{\mu^{1+\theta}}, \quad \mu_{L}=L^{\frac{N-2}{\beta-N+2}}\left(C_{0}+o(1)\right),$$
where
$$
C_{0}=\Bigl[C_1\frac{-\sum_{i \in J} a_{i} \ds\int_{\mathbb{R}^{N}}|y_i|^{\beta} U_{0,1}^{2^{*}-1} \psi_{0}}{\left(2^{*}-1\right)  \ds\int_{\mathbb{R}^{N}} \Bigl(B_1U_{0,1}^{2^{*}-2} \psi_{0}+B_2V_{0,1}^{2^{*}-2} \xi_{0}\Bigr) \sum_{j=1}^{\infty} \Gamma\left(P_{j}, 0\right)}\Bigr]^{\frac{1}{\beta-(N-2)}}>0.
$$
So, we have proved that  $(u,v)=(P U_{x_{L}, \mu_{L}}+\omega_{1},PV_{x_{L},\mu_{L}}+\omega_2)$ satisfies
\begin{eqnarray*}\begin{cases}
-\Delta u=K_1(y) u_{+}^{2^*-1}+\frac{1}{2}u_+^{\frac{2^*}{2}-1}v_+^{\frac{2^*}{2}},\,\,\,\,\,y\in\Omega,\cr
-\Delta v=K_2(y) v_{+}^{2^*-1}+\frac{1}{2}u_+^{\frac{2^*}{2}}v_+^{\frac{2^*}{2}-1},\,\,\,\,\,y\in\Omega.
\end{cases}
\end{eqnarray*}

Now we prove that  $u>0$. It is easy to check that
$$
|u(y)| \leq \frac{C}{\left(1+\left|y^{\prime \prime}\right|\right)^{N-2-k-\vartheta}}, \quad y=\left(y^{\prime}, y^{\prime \prime}\right) \in \Omega, \quad y^{\prime \prime} \in \mathbb{R}^{N-k}.
$$
From
$$
u(y)=\int_{\Omega} G(z, y) \Bigl(K_1(z) u_{+}^{2^{*}-1}+\frac{1}{2}u_+^{\frac{2^*}{2}-1}v_+^{\frac{2^*}{2}}\Bigr) dz,
$$
we obtain

\begin{align*}
|\nabla u(y)| & =\Bigl|\int_{\Omega} \nabla_{y} G(z, y) \Bigl(K_1(z) u_{+}^{2^{*}-1}+\frac{1}{2}u_+^{\frac{2^*}{2}-1}v_+^{\frac{2^*}{2}}\Bigr) \Bigr|\cr
&\leq C \int_{\Omega} \frac{1}{|z-y|^{N-1}} \frac{C}{\left(1+\left|z^{\prime \prime}\right|\right)^{\left(2^{*}-1\right)(N-2-k-\vartheta)}} \cr
& \leq \frac{C}{\left(1+\left|y^{\prime \prime}\right|\right)^{\min \left[\left(2^{*}-1\right)(N-2-k-\vartheta)-1, N-1-\vartheta\right]}}.
\end{align*}

Note that  $N-2-k-\vartheta>\frac{N-2}{2}$  and  $\left(2^{*}-1\right)(N-2-k-\vartheta)-1>\frac{N}{2}$. Thus,  $\ds\int_{\Omega \cap \partial B_{R}(0)}|u||\nabla u| \rightarrow 0$  as  $R \rightarrow \infty$. Let  $u_{-}=u$  if  $u<0$ , and  $u_{-}=0$  if  $u \geq 0$. We now multiply
\begin{eqnarray*}\begin{cases}
-\Delta u=K_1(y) u_{+}^{2^*-1}+\frac{1}{2}u_+^{\frac{2^*}{2}-1}v_+^{\frac{2^*}{2}},\,\,\,\,\,y\in\Omega,\cr
-\Delta v=K_2(y) v_{+}^{2^*-1}+\frac{1}{2}u_+^{\frac{2^*}{2}}v_+^{\frac{2^*}{2}-1},\,\,\,\,\,y\in\Omega
\end{cases}
\end{eqnarray*}
 by  $(u_{-},v_{-})$ and integrate over  $\Omega$  to obtain  $\ds\int_{\Omega}\left|\nabla u_{-}\right|^{2}=0,\,\ds\int_{\Omega}\left|\nabla v_{-}\right|^{2}=0$. This gives  $(u_{-},v_{-})=(0,0)$. Strong maximum principle implies that $u$ and $v$ are positive in  $\Omega$.

\appendix
\section{Proof of Lemma \ref{lem3.4}}
In this section, we mainly give the proof of Lemma \ref{lem3.4}.

\textbf{Proof of Lemma \ref{lem3.4}.} Let  $\varepsilon \in(0,1)$  be any fixed small constant. Suppose that  $z \in B_{\varepsilon^{-1}}(x) \cap \Omega$. Then, similar to \eqref{eqs3.18} and \eqref{eqs3.19}, there is a constant  $C_{\varepsilon}>0$, depending on  $\varepsilon>0$, such that
$$\sum_{j=0}^{\infty} U_{x_{j}, \mu} \leq \frac{C_{\varepsilon} \mu^{\frac{N-2}{2}}}{(1+\mu|z-x|)^{N-2}}$$
and
$$\sum_{j=0}^{\infty} \frac{\mu^{\frac{N-2}{2}}}{\left(1+\mu\left|z-x_{j}\right|\right)^{\frac{N-2}{2}+\tau}} \leq \frac{C_{\varepsilon} \mu^{\frac{N-2}{2}}}{(1+\mu|z-x|)^{\frac{N-2}{2}+\tau}}.$$
So we have that for  $z \in B_{\varepsilon^{-1}}(x) \cap \Omega$,
\begin{equation}\label{eqsA.1}
\Bigl(\sum_{j=0}^{\infty} U_{x_{j}, \mu}\Bigr)^{\frac{4}{N-2}} \sum_{j=0}^{\infty} \frac{\mu^{\frac{N-2}{2}}}{\left(1+\mu\left|z-x_{j}\right|\right)^{\frac{N-2}{2}+\tau}} \leq \frac{C_{\varepsilon} \mu^{\frac{N+2}{2}}}{(1+\mu|z-x|)^{\frac{N-2}{2}+4+\tau}},
\end{equation}
which gives
\begin{align}\begin{split}\label{eqsA.2}
&\int_{B_{\varepsilon^{-1}}(x) \cap \Omega} G(z, y)\Bigl(\sum_{j=0}^{\infty} U_{x_{j}, \mu}\Bigr)^{\frac{4}{N-2}} \sigma(z) \sum_{j=0}^{\infty} \frac{\mu^{\frac{N-2}{2}}}{\left(1+\mu\left|z-x_{j}\right|\right)^{\frac{N-2}{2}+\tau}} dz\\
&\quad \leq C_{\varepsilon}\int_{B_{\varepsilon^{-1}}(x) \cap \Omega} G(z, y)\frac{\mu^{\frac{N+2}{2}}}{(1+\mu|z-x|)^{\frac{N-2}{2}+4+\tau}}\Bigl(\frac{1+\mu|z-x|}{\mu}\Bigr)^{\tau-1}\\
&\quad \leq C_{\varepsilon} \int_{\Omega} G(z, y) \frac{\mu^{\frac{N+2}{2}-\tau+1}}{(1+\mu|z-x|)^{\frac{N-2}{2}+5}}\\
&\quad =C_{\varepsilon} \int_{\mathbb{R}^{N}} \Gamma(z, y) \sum_{j=0}^{\infty} \frac{\mu^{\frac{N+2}{2}-\tau+1}}{\left(1+\mu\left|z-x_{j}\right|\right)^{\frac{N-2}{2}+5}} \\
&\quad \leq C_{\varepsilon} \sum_{j=0}^{\infty} \frac{\mu^{\frac{N-2}{2}-\tau+1}}{\left(1+\mu\left|y-x_{j}\right|\right)^{\min \left(\frac{N-2}{2}+3, N-2\right)}}.
\end{split}
\end{align}
We have
$$\sigma(y) \geq 2^{1-\tau} \varepsilon^{\tau-1}\Bigl(\frac{1+\mu|y-x|}{\mu}\Bigr)^{\tau-1}, \quad y \in B_{\varepsilon^{-1}}(x) \cap \Omega.$$
Thus, for  $y \in B_{\varepsilon^{-1}}(x) \cap \Omega$, there holds
\begin{align}\begin{split}\label{eqsA.3}
\frac{\mu^{\frac{N-2}{2}-\tau+1}}{(1+\mu|y-x|)^{\min \left(\frac{N-2}{2}+3, N-2\right)}}
 &=\Bigl(\frac{1+\mu|y-x|}{\mu}\Bigr)^{\tau-1} \frac{\mu^{\frac{N-2}{2}}}{(1+\mu|y-x|)^{\min \left(\frac{N-2}{2}+3, N-2\right)+\tau-1}} \\
&\leq \frac{C \varepsilon^{1-\tau} \sigma(y) \mu^{\frac{N-2}{2}}}{(1+\mu|y-x|)^{\frac{N-2}{2}+\tau+\theta}}.
\end{split}
\end{align}
On the other hand, we have
\begin{equation}\label{eqsA.4}
\sum_{j=1}^{\infty} \frac{\mu^{\frac{N-2}{2}-\tau+1}}{\left(1+\mu\left|y-x_{j}\right|\right)^{\min \left(\frac{N-2}{2}+3, N-2\right)}} \leq \frac{C \mu^{\frac{N-2}{2}-\tau+1}}{(\mu L)^{\min \left(\frac{N-2}{2}+3, N-2\right)}}
\end{equation}
and
\begin{align}\begin{split}\label{eqsA.5}
&\frac{\sigma(y) \mu^{\frac{N-2}{2}}}{(1+\mu|y-x|)^{\frac{N-2}{2}+\tau+1}}\\
&\quad \geq 2^{1-\tau} \varepsilon^{\tau-1}\Bigl(\frac{1+\mu|y-x|}{\mu}\Bigr)^{\tau-1} \frac{\mu^{\frac{N-2}{2}}}{\left(1+\mu\left|y-x\right|\right)^{\frac{N-2}{2}+\tau+1}} \\
&\quad =2^{1-\tau} \varepsilon^{\tau-1} \frac{\mu^{\frac{N-2}{2}-\tau+1}}{(1+\mu|y-x|)^{\frac{N-2}{2}+2}}\\
&\quad \geq 2^{1-\tau} \varepsilon^{\tau-1} \frac{\mu^{\frac{N-2}{2}-\tau+1}}{\left(1+\mu\left(\varepsilon^{-1}+1\right)\right)^{\frac{N-2}{2}+2}} \geq 2^{1-\tau} \varepsilon^{\tau-1+\frac{N-2}{2}+2} \frac{\mu^{\frac{N-2}{2}-\tau+1}}{(\mu L)^{\frac{N-2}{2}+2}}.
\end{split}
\end{align}
Combining \eqref{eqsA.4} and \eqref{eqsA.5}, we see that for  $y \in B_{\varepsilon^{-1}}(x) \cap \Omega$,
\begin{equation}\label{eqsA.6}
\sum_{j=1}^{\infty} \frac{\mu^{\frac{N-2}{2}-\tau+1}}{\left(1+\mu\left|y-x_{j}\right|\right)^{\min \left(\frac{N-2}{2}+3, N-2\right)}} \leq \frac{C_{\varepsilon} \sigma(y) \mu^{\frac{N-2}{2}}}{(1+\mu|y-x|)^{\frac{N-2}{2}+\tau+\theta}}.
\end{equation}
Combining \eqref{eqsA.2}, \eqref{eqsA.3} and \eqref{eqsA.6}, we obtain
\begin{align}\begin{split}\label{eqsA.7}
&\int_{B_{\varepsilon^{-1}}(x) \cap \Omega} G(z, y)\Bigl(\sum_{j=0}^{\infty} U_{x_{j}, \mu}\Bigr)^{\frac{4}{N-2}} \sigma(z) \sum_{j=0}^{\infty} \frac{\mu^{\frac{N-2}{2}}}{\left(1+\mu\left|z-x_{j}\right|\right)^{\frac{N-2}{2}+\tau}} d z \\
&\quad \leq C_{\varepsilon} \sigma(y)\sum_{j=0}^{\infty} \frac{\mu^{\frac{N-2}{2}}}{\left(1+\mu\left|y-x_{j}\right|\right)^{\frac{N-2}{2}+\tau+\theta}}.
\end{split}
\end{align}
Suppose that  $z \in \Omega \backslash B_{\varepsilon^{-1}}(x)$. Then
$$\sum_{j=0}^{\infty} U_{x_{j}, \mu}(z) \leq \frac{C \varepsilon^{\frac{N-2}{2}-\tau}}{\mu^{\frac{N-2}{2}-\tau}} \sum_{j=0}^{\infty} \frac{\mu^{\frac{N-2}{2}}}{\left(1+\mu\left|z-x_{j}\right|\right)^{\frac{N-2}{2}+\tau}}.$$
Thus, we have

\begin{align}\begin{split}\label{eqsA.8}
&\Bigl(\sum_{j=0}^{\infty} U_{x_{j}, \mu}\Bigr)^{\frac{4}{N-2}} \sum_{j=0}^{\infty} \frac{\mu^{\frac{N-2}{2}}}{\left(1+\mu\left|z-x_{j}\right|\right)^{\frac{N-2}{2}+\tau}}\\
&\quad \leq \frac{C \varepsilon^{2-\frac{4 \tau}{N-2}}}{\mu^{2-\frac{4 \tau}{N-2}}}\Bigl(\sum_{j=0}^{\infty} \frac{\mu^{\frac{N-2}{2}}}{\left(1+\mu\left|z-x_{j}\right|\right)^{\frac{N-2}{2}+\tau}}\Bigr)^{2^{*}-1} \\
&\quad \leq \frac{C \varepsilon^{2-\frac{4 \tau}{N-2}}}{\mu^{2-\frac{4 \tau}{N-2}}} \sum_{j=0}^{\infty} \frac{\mu^{\frac{N+2}{2}}}{\left(1+\mu\left|z-x_{j}\right|\right)^{\frac{N+2}{2}+\tau}}\Bigl(\sum_{j=0}^{\infty} \frac{1}{\left(1+\mu\left|z-x_{j}\right|\right)^{\tau}}\Bigr)^{\frac{4}{N-2}}\\
&\quad \leq \frac{C \varepsilon^{2-\frac{4 \tau}{N-2}}}{\mu^{2-\frac{4 \tau}{N-2}}} \sum_{j=0}^{\infty} \frac{\mu^{\frac{N+2}{2}}}{\left(1+\mu\left|z-x_{j}\right|\right)^{\frac{N+2}{2}+\tau}}, \quad \forall z \in \Omega \backslash B_{\varepsilon^{-1}}(x).
\end{split}
\end{align}
This gives
\begin{align}\begin{split}\label{eqsA.9}
&\int_{B_{\Omega \backslash \varepsilon^{-1}}(x)} G(z, y)\Bigl(\sum_{j=0}^{\infty} U_{x_{j}, \mu}\Bigr)^{\frac{4}{N-2}} \sum_{j=0}^{\infty} \frac{\mu^{\frac{N-2}{2}}}{\left(1+\mu\left|z-x_{j}\right|\right)^{\frac{N-2}{2}+\tau}} d z \\
&\quad \leq \frac{C \varepsilon^{2-\frac{4 \tau}{N-2}}}{\mu^{2-\frac{4 \tau}{N-2}}} \int_{\Omega} G(z, y) \sum_{j=0}^{\infty} \frac{\mu^{\frac{N+2}{2}}}{\left(1+\mu\left|z-x_{j}\right|\right)^{\frac{N+2}{2}+\tau}}\\
&\quad \leq \frac{C \varepsilon^{2-\frac{4 \tau}{N-2}}}{\mu^{2-\frac{4 \tau}{N-2}}} \sum_{j=0}^{\infty} \frac{\mu^{\frac{N-2}{2}}}{\left(1+\mu\left|y-x_{j}\right|\right)^{\frac{N-2}{2}+\tau}}.
\end{split}
\end{align}
So, the result follows from \eqref{eqsA.7} and \eqref{eqsA.9}.

Acknowledgement: The authors would like to thank Professor Chunhua Wang for the helpful discussion with her.

\end{document}